\documentclass[11pt,a4paper]{amsart}
\usepackage[margin=3cm]{geometry}

\usepackage{amsmath,amsfonts,amssymb,amscd,amsthm, mathtools}
\usepackage[usenames,dvipsnames,svgnames]{xcolor}

\usepackage[utf8]{inputenc}
\usepackage{graphicx}
\usepackage{caption}
\usepackage{subcaption}
\usepackage{doi}
\usepackage{autonum}
\usepackage{makecell}
\usepackage{hyperref}

\usepackage{acronym}
\usepackage{enumitem}

\def\n{\boldsymbol{n}}
\def\0{\boldsymbol{0}}
\def\uu{\boldsymbol{u}}
\def\vv{\boldsymbol{v}}
\def\g{\boldsymbol{g}}

\def\f{\boldsymbol{f}}

\def\R{\boldsymbol{R}}

\def\x{\boldsymbol{x}}

\def\I{\boldsymbol{I}}
\def\Hs{{H}^1(\Omega)}
\def\Hsr{{H}^r(\Omega)}
\def\Hso{{H}^1_0(\Omega)}
\def\Hcurl{{H}(\textbf{{curl}}, \Omega)}
\def\Hcurlr{{H}^r(\textbf{{curl}}, \Omega)}
\def\Hcurlo{{H}_0(\textbf{{curl}}, \Omega)}
\def\esp{{\mathcal{V}}}
\def\reffesp{\hat{\mathcal{V}}}
\def\nedespk{\mathcal{ND}_k}
\def\nedesphk{\mathcal{ND}_{k}(K_h^s)}
\def\nedesphh{\mathcal{ND}_{{k}}({K}_h)}

\def\nedesp2hk{\mathcal{ND}_{k}(K_{2h})}
\def\lagespk{\mathcal{L}_{\boldsymbol{k}}}
\def\lagesphh{\mathcal{L}_{\boldsymbol{k}}({K}_h)}
\def\lagespthk{\mathcal{L}_{\boldsymbol{k}}(\tilde{K}_h)}
\def\lagesphk{\mathcal{L}_{\boldsymbol{k}}(K_h^s)}
\def\lagesp2hk{\mathcal{L}_{\boldsymbol{k}}(K_{2h})}
\def\geophy{K}
\def\georef{{\hat K}}
\def\funmap{\Psi}
\def\geomap{\boldsymbol{\Phi}}
\def\P{\mathcal{P}}
\def\Q{\mathcal{Q}}
\def\Am{\boldsymbol{\mathcal{A}}}
\def\Km{\boldsymbol{\mathcal{K}}}
\def\Mm{\boldsymbol{\mathcal{M}}}
\def\jacobian{\boldsymbol{J}_K}
\def\fm{\boldsymbol{f}}
\def\bsnabla{\boldsymbol{\nabla}}

\DeclarePairedDelimiter{\norm}{\lVert}{\rVert}
\newtheorem{remark}{Remark}[section]
\newtheorem{theorem}{Theorem}[section]
\newtheorem{proposition}{Proposition}[section]
\newtheorem{Example}{Example}[section]
\newtheorem{definition}{Definition}[section]

\def\sect#1{Sect.~\ref{#1}}
\def\eq#1{(\ref{#1})}
\def\eqref#1{(\ref{#1})}
\def\fig#1{Fig.~\ref{#1}}

\acrodef{fe}[FE]{finite element}
\acrodefplural{fe}[FEs]{finite elements}
\acrodef{dof}[DoF]{degree of freedom}
\acrodefplural{dof}[DoFs]{degrees of freedom}
\acrodef{pde}[PDE]{partial differential equation}
\acrodefplural{pde}[PDEs]{partial differential equations}
\acrodef{amr}[AMR]{adaptive mesh refinement}
\acrodef{vem}[VEM]{virtual element method}

\def\FEMPAR{{\texttt{FEMPAR}~}}
\def\crop_s{2.25cm}

\def\mo#1{{ #1}}

\begin{document}

\title[]{On a general implementation of $h$- and $p$-adaptive curl-conforming finite elements}

\author[M. Olm]{Marc Olm$^\dag$}

\author[S. Badia]{Santiago Badia$^{\dag,\ddag}$}

\author[A. F. Mart\'in]{Alberto F. Mart\'in$^{\dag,\ddag}$}

\thanks{$\dag$ Centre Internacional de M\`etodes Num\`erics en Enginyeria, Esteve Terrades 5, E-08860 Castelldefels, Spain. $\ddag$ Universitat Polit\`ecnica de Catalunya, Jordi Girona 1-3, Edifici C1, E-08034 Barcelona, Spain. \\ MO gratefully acknowledges the support received from the Catalan Government through the FI-AGAUR grant. SB gratefully acknowledges the support received from the Catalan Government through the ICREA Acad\`emia Research Program. This work has partially been funded by the European Union through the FP7 project FORTISSIMO, under the grant agreement 609029. E-mails: {\tt molm@cimne.upc.edu} (MO), {\tt sbadia@cimne.upc.edu} (SB), {\tt amartin@cimne.upc.edu} (AM)}

\date{\today}

\begin{abstract}
  Edge (or  N\'ed\'elec) finite elements are theoretically sound and widely used by the computational electromagnetics community. However, its implementation, \mo{especially} for high order methods, is not trivial, since it involves many technicalities that are not properly described in the literature. To fill this gap, we provide a comprehensive description of a general implementation of edge elements of first kind within the scientific software project \FEMPAR. We cover into detail how to implement arbitrary order (i.e., $p$-adaptive) elements on hexahedral and tetrahedral meshes. First, we set the three classical ingredients of the finite element definition by Ciarlet, both in the reference and the physical space: cell topologies, polynomial spaces and moments. With these ingredients, shape functions are automatically implemented by defining a judiciously chosen polynomial \emph{pre-basis} that spans the local finite element space combined with a change of basis to automatically obtain a canonical basis with respect to the moments at hand. Next, we discuss global finite element spaces putting emphasis on the construction of global shape functions through oriented meshes, appropriate geometrical mappings, and equivalence classes of moments, in order to preserve the inter-element continuity of tangential components of the magnetic field. Finally, we  extend the proposed methodology to generate global curl-conforming spaces on \emph{non-conforming} hierarchically refined (i.e., $h$-adaptive) meshes with arbitrary order finite elements. Numerical results include experimental convergence rates to test the proposed implementation.
\end{abstract}

\maketitle



\section{Introduction} \label{sec:intro}
Edge elements were originally proposed in the seminal work by N\'ed\'elec~\cite{nedelec_mixed_1980}. They are a natural choice in electromagnetic \ac{fe} simulations due to their sound mathematical structure~\cite{monk_finite_2003}. In short, edge \ac{fe} spaces represent curl-conforming fields with continuous tangential components and discontinuous normal components. It is recognized to be their greatest advantage against Lagrangian \acp{fe}~\cite{mur_advantages}.\footnote{There are different approaches to design discretizations of the Maxwell equations that rely on Lagrangian \acp{fe} and can theoretically converge to singular solutions (see, e.g.,~\cite{badia_nodal_2012}). However, these methods are not robust for complex electromagnetic problems. In our experience~\cite{badia_HTS}, the usage of edge elements is of imperative importance in the modelling of fields near singularities by allowing normal components to jump across interfaces between two different media with highly contrasting properties.} The curl operator has null space, the gradient of any scalar field. To remove its kernel, one needs to include a zero-order term, a transient problem or a divergence-free constraint over the magnetic field times resistivity\footnote{Edge \acp{fe} provide solutions that are pointwise divergence-free in the element interiors. However, the normal component of the field can freely jump across inter-element faces. The zero inter-element jump constraint has to be enforced by a Lagrange multiplier~\cite{badia_nodal_2012} to eliminate the kernel.} (also known as Coulomb gauge)~\cite{li_vectorial_2015}.

\mo{The edge \ac{fe} method has the status of the method of choice in the computational electromagnetics community. Nevertheless, other discretization techniques are available for electromagnetic problems. Apart from the aforementioned nodal \acp{fe} \cite{badia_nodal_2012}, nodal discontinuous Galerkin methods \cite{perugia_dG_2003, badia_dG_2011} and isogeometric analysis \cite{buffa_iga} have also been considered for the Maxwell problem. Furthermore, the \ac{vem} \cite{beirao_2013}, which is designed to handle meshes consisting of arbitrary polygonal or polyhedral elements, has also been defined for curl-conforming spaces \cite{beirao_2016}. Local edge \ac{vem} spaces \cite{beirao_2016} rely on a space formulated on the boundary of the element, thus they naturally allow to design adaptive methods with conforming, arbitrary meshes.}

Edge \acp{fe} are widespread in the computational electromagnetics community, mainly for the lowest order case. In fact, they receive the name of edge elements because for first order approximations each \ac{dof} is associated with an edge of the mesh, and \acp{dof} can be understood as circulations along edges. However, its implementation, \mo{especially} for high order, is not a trivial task and involves many technicalities. On the other hand, there are few works addressing the implementation details of arbitrary-order edge \ac{fe} schemes.

The construction of local bases for edge elements has been addressed by few authors. A basis for high order methods, expressed in terms of the (barycentric) affine coordinates of the simplex (i.e., tetrahedra) is described in~\cite{gopalakrishnan_2005, Bonazzoli_2017}, including also implementation details in~\cite{Bonazzoli_implementation_2017}. \mo{The definition of affine-related coordinates is used to define expressions for the bases in tetrahedra, prisms and pyramids in \cite{fuentes_nedelec}.
Furthermore, hierarchical bases of functions can be computed for hexahedra by tensor product of 1D Legendre polynomials and the so-called 1D $H$(curl)-shape functions \cite{demkowicz_book, demkowicz_book2}, which are in turn obtained by differentiation of the former ones. Nevertheless, we prefer to rely on the \ac{fe} definition by Ciarlet, which induces the basis of shape functions as the dual basis with respect to suitable edge \acp{dof} \cite{amor_castillo, castillo_third}.
In this work, arbitrary order edge basis functions are automatically implemented by defining polynomial \emph{pre-bases} that span the local finite element spaces combined with a change of basis.}

Better documented is the construction of global curl-conforming spaces~\cite{schneebeli_2003, rognes_efficient_2009}, which relies on the Piola mapping to achieve the global continuity of the tangent components by using moments defined in the reference cell. In the case of unstructured hexahedral meshes, where non-affine geometrical mappings must be used in general, optimal convergence of \emph{standard} edge elements in the $H$(curl) norm cannot be achieved~\cite{falk_2011}. Complex geometries can still be considered using, e.g., unfitted \ac{fe} techniques~\cite{badia_unfitted_2018} on octree background meshes to avoid non-affine mappings.

Special care has to be taken in the definition of edge \acp{dof} to enforce the right continuity across cells. To prevent the so-called sign conflict, different alternatives have been proposed. A simple sign flip ~\cite{anjam_2015} can assure consistency, {i.e., all local \acp{dof} shared by two or more tetrahedra represent the same global \ac{dof},} for first order 3D \acp{fe} only. For higher order methods, a remedy can be based on the construction of all possible shape functions combined with a permutation that depends on local edge/face orientations \cite{fuentes_nedelec, castillo_99}. Our preferred solution, for simplicity and ease of implementation, is to \emph{orient} tetrahedral meshes, requiring that local nodes numbering within every element are always sorted based on their global indices~\cite{rognes_efficient_2009}. {On the other hand, although there are hexahedral meshes that cannot be oriented in 3D~\cite{agelek_orientation_2017}, we restrict
ourselves to forest-of-octrees meshes~\cite{BursteddeWilcoxGhattas11} such that that the initial coarse mesh can be oriented, e.g., a uniform, structured mesh.} 

Few open source \ac{fe} software projects include the implementation of edge \acp{fe} of arbitrary order. FEniCS supports arbitrary order \acp{fe}~\cite{alnaes_fenics_2015}, but is restricted to tetrahedra and has its own domain specific language for weak forms to automatically generate the corresponding \ac{fe} code, which may prevent $hp$-adaptivity. On the other hand, the deal.ii library also provides arbitrary order edge \acp{fe}~\cite{bangerth_deal.ii-general-purpose_2007,bangerth_textttdeal.ii_2016}, even on $h$-adaptive meshes, but it is restricted to hexahedral meshes. Recently, a C++ plugin that defines high order edge \ac{fe} ingredients~\cite{Bonazzoli_implementation_2017} has been developed for the FreeFEM++ library~\cite{hecht_new_2012}, but it is restricted to up to third order and tetrahedral meshes.
\mo{As far as we know, the most complete
 approaches in the literature are the implementations in MFEM~\cite{_mfem_????}, Netgen/NGSolve~\cite{Schoberl1997, netgen} and hp2D~\cite{demkowicz_book}/hp3D~\cite{demkowicz_book2}.
MFEM provides arbitrary order tetrahedral/hexhahedral edge FEs and $h$-adaptivity on non-conforming hexahedral meshes.
Netgen/NGSolve and hp3D implement arbitrary order edge prism and pyramidal FEs, besides the tetrahedral and hexahedral ones,
and hp3D provides the $hp$-adaptive method on non-conforming hexahedral meshes.}

The usage of non-nodal based (compatible or structure-preserving) \acp{fe}, especially in multi-physics applications in combination with incompressible fluid and solid mechanics \cite{Shadid2016,Ortigosa2016,GuolinWang2015}, is still scarce. It is motivated by the fact that most codes in computational mechanics are designed from inception for being used with nodal \acp{fe}, and the extension to other types of \acp{fe} seems to be a highly demanding task. \mo{Although non-nodal based \acp{fe} are theoretically well known, the poor, fractioned, and spread information on key implementation issues does not contribute to its greater diffusion in the computational mechanics community.}

The motivation of this work is to provide \emph{a comprehensive description of a novel and general implementation of edge \acp{fe} of arbitrary order on tetrahedral and hexahedral}\footnote{Although this work is restricted to hexahedra and tetrahedra, one can also find definitions for prism and pyramid edge elements~\cite{bergot_2013} {with optimal rates of convergence of the numerical solution towards the exact solution in the $\Hcurl$ norm.}} \emph{non-conforming meshes}, confronted with typical hard-coded implementations of the shape functions that preclude high order methods. With this aim, we will address three critical points that are \mo{especially} complex in the construction of arbitrary order edge \acp{fe}: 1) an automatic manner of generating the associated polynomial spaces; 2) a general solution for the so-called sign conflict; and 3) a construction of the discrete $H({\rm curl})$-conforming \ac{fe} spaces atop \emph{non-conforming} meshes. With regard to 1), we generate arbitrary order edge \ac{fe} shape functions in a simple but extensible manner. First, we propose a \emph{pre-basis} of vector-valued polynomial functions spanning arbitrary order local edge \ac{fe} spaces for both tetrahedral and hexahedral topologies; see~\cite{bergot_2010} for a related approach on tetrahedra. Next, we compute a change of basis from the \emph{pre-basis} to the canonical basis with respect to the edge \ac{fe} \acp{dof}, i.e., the basis of shape functions. To address 2) and properly determine the inter-element continuity, we rely on oriented meshes. Finally, in order to address 3),
we combine a standard \ac{amr} nodal-based implementation of constraints on non-conforming geometrical objects (a.k.a. as hanging) \cite{Bangerth_2009} with the relation between a Lagrangian \emph{pre-basis} and the edge basis, to avoid the evaluation of edge moments on the refined cells.
Besides, the same machinery can also be applied to other polynomial based \acp{fe}, like Raviart-Thomas \acp{fe}, Brezzi-Douglas-Marini \acp{fe} \cite{brezzi_mixed_1991}, second kind edge \acp{fe} \cite{nedelec_new_mixed_1986}, or recent divergence-free \acp{fe} \cite{neilan_stokes_2015}.
\mo{An alternative approach, more suitable for hierarchical bases of functions, is to compute the constraints on non-conforming objects by comparison of the two representations of a function (atop the coarse or the refined cell) at uniformly distributed collocation points \cite{demkowicz_book}.}

\mo{ In the course of an $hp-$adaptive finite element procedure, an error estimator indicates at which mesh cells the error of the computed \ac{fe} solution is higher, which will be potential candidates to be refined. For curl-conforming problems, several types of \emph{a posteriori} error estimators have been developed and analyzed. Efficient and reliable residual-based error estimators were introduced in \cite{hiptmair_adaptive_2000} and have been further developed and analyzed in \cite{schoberl_estimators, nicaise_estimators}. Among others, we also find hierarchical error estimators \cite{beck_adaptive_1999}, equilibrated estimators \cite{braess_estimator} or recovery-based estimators \cite{nicaise_estimator}.
} 

The proposed approach is the result of the experience gained by the authors after the  implementation in \FEMPAR~\cite{badia-fempar,fempar-web-page}, a scientific software project for the simulation of problems governed by \acp{pde}. In any case, we present an automatic, comprehensive strategy rather than focusing on a particular software implementation.  Consequently, thorough definitions and examples will be provided, but the exact software implementation will not be shown. The reader is directed to~\cite{badia-fempar} for code details, where an exhaustive introduction to the software abstractions of \FEMPAR is presented. 

The text is organized as follows. In \sect{sec-formulation}, we introduce some notation and a general problem with the curl-curl formulation. In \sect{sec-ned_fes}, we will cover the definition and construction of the edge \ac{fe} of arbitrary order, focusing on an straightforward manner of generating the involved polynomial spaces. \sect{sec-assembly} is devoted to the construction of global curl-conforming \ac{fe} spaces. In \sect{sec-h_adaptive}, a strategy to deal with $h$-adaptive spaces with the edge element will be described. Finally, in \sect{sec-results}, we will show some convergence results, which will validate the implementation of the edge elements.

\section{Problem statement}\label{sec-formulation}
\subsection{Notation}\label{sec-notation}
In this section, we introduce the model problem to be solved and some mathematical notation. Bold characters will be used to describe vector functions or tensors, while regular ones will determine scalar functions or values. No difference is intended by using upper-case or lower-case letters for functions.

Let $\Omega\subset\mathbb{R}^d$ be a bounded domain with $d=2,3$ the space dimension. In 3D, the rotation operator of a vector field $\vv$ is defined as:
\begin{align}\label{eq-curl_operator}
\bsnabla \times \vv = \begin{bmatrix}
\partial_2 v_3 - \partial_3 v_2 \\
\partial_3 v_1 - \partial_1 v_3 \\
\partial_1 v_2 - \partial_2 v_1
\end{bmatrix}.
\end{align}
We use the standard multi-index notation for derivatives, with $\boldsymbol{\alpha} := (\alpha_1, \ldots, \alpha_d)^T \in \mathbb{Z}_ +^d$ and $|\boldsymbol{\alpha}|_1 = \sum_{i=1}^d |\alpha_i|$. Let us define the spaces
\begin{eqnarray}
  \Hsr  &:=& \{ v \in L^2(\Omega)  \hbox{ }| \hbox{ } \partial^{\boldsymbol{\alpha}} v \in L^2(\Omega) \hbox{ for all } |\boldsymbol{\alpha}|_1 \leq r  \}, \\
\Hcurlr &:=& \{ \vv \in \Hsr^d  \hbox{ }| \hbox{ } \bsnabla \times \vv \in \Hsr^{d} \}. \label{eq-Hcurl}
\end{eqnarray}
The space ${H}^0(\textbf{{curl}}, \Omega)$ is represented with ${H}(\textbf{{curl}}, \Omega)$. We also consider the subspaces
\begin{eqnarray}
\Hso  &:=& \{ v \in \Hs \hbox{ }| \hbox{ } v=0 \hbox{ in } \partial\Omega \}, \\
\Hcurlo &:=& \{ \vv \in \Hcurl \hbox{ }| \hbox{ } \n \times \vv = \0 \hbox{ in } \partial\Omega \},
\end{eqnarray}
where $\n$ denotes the outward unit normal to the boundary of the domain $\Omega$. The space $\Hcurl$ in \eq{eq-Hcurl} is equipped with the norm
\[
\norm{\uu}_{H(\textbf{{curl}}, \Omega)} := \left( \norm{\uu}_{L^2(\Omega)^d}^2 + \norm{\bsnabla \times \uu}_{L^2(\Omega)^{d}}^2 \right)^{\frac{1}{2}}.
\]
{In 2D, a scalar version of the curl operator can be defined as $\mathrm{curl} (\vv) \doteq \partial_1 v_2 - \partial_2 v_1$. We can similarly define $H^r({\rm curl}, \Omega)$ and its related subspaces}. In the following, we consider the notation for the 3D case.


\subsection{Formulation}\label{sec-cc_form}
The model problem consists in finding a vector field $\uu$ (magnetic field) solution of
\begin{align}
\bsnabla \times \alpha \bsnabla \times \uu + \beta \uu = \f & \qquad \hbox{in } \, \Omega,   \label{eq-dcurl_p}
\end{align}
where $\f$ is a given source term. Problem parameters $\alpha, \beta$ can range from scalar, positive values for isotropic materials to positive definite tensors for anisotropic materials. Besides, (\ref{eq-dcurl_p}) needs to be supplied with appropriate boundary conditions. The boundary of the domain $\partial \Omega$ is divided into its Dirichlet boundary part, i.e., $\partial \Omega_D$, and its Neumann boundary part, i.e., $\partial \Omega_N$, such that $\partial \Omega_D \cup \partial \Omega_N = \partial \Omega$ and $ \partial \Omega_D \cap \Omega_N = \emptyset$. Then, boundary conditions for the problem at hand read
\begin{align}
\uu \times \n = \g_D && \quad \hbox{on } \, \partial \Omega_D, \\
\n \times (\alpha \bsnabla \times \uu) = \g_N && \quad \hbox{on } \, \partial \Omega_N,
\end{align}
where $\n$ is a unit normal to the boundary. 
Dirichlet (essential) boundary conditions prescribe the tangent component of the field $\uu$ on the boundary of the domain. On the other hand, Neumann (natural) boundary conditions arise in the integration by parts of the curl-curl term
\begin{eqnarray}\label{eq-neum_cond}
\int_\Omega (  \bsnabla \times \alpha \bsnabla \times \uu ) \cdot \vv &=& \int_\Omega ( \alpha \bsnabla \times \uu ) \cdot (\bsnabla \times \vv) - \int_{\Omega_N} (\alpha \bsnabla \times \uu) \cdot (\n \times \vv ).
\end{eqnarray}
Consider now two different non-overlapping regions on the domain $\Omega$ corresponding to two different media (usual case in electromagnetic simulations), namely $\Omega_1$ and $\Omega_2$, such that $\Omega_1 \cup \Omega_2 = \Omega$, and let us define the \emph{interface} as $\Gamma := \Omega_1 \cap \Omega_2$. Let us denote by $\n_1, \n_2$ the unit normal pointing outwards of $\Omega_1$ and $\Omega_2$ on $\Gamma$, resp. The \emph{transmission} conditions (in absence of other sources) for \eq{eq-dcurl_p} are stated as follows:
\[
 \n \times (\alpha_1 \bsnabla \times \uu_1 - \alpha_2 \bsnabla \times \uu_2 ) = 0 \qquad \text{ on } \Gamma,
\]
where $\n$ can, e.g., be $\n_1=-\n_2$. For the sake of simplicity in the presentation (not in the implementation), let us consider $\partial \Omega_D = \partial \Omega$, so that the Neumann boundary term \eq{eq-neum_cond} can be removed from the weak form, and homogeneous Dirichlet boundary conditions. Hence, the variational form of the double curl formulation \eq{eq-dcurl_p} reads: find $\uu \in \Hcurlo$ such that
\begin{eqnarray}
 (\alpha \bsnabla \times \uu, \bsnabla \times \vv) + ( \beta \uu, \vv ) = (\f, \vv),  & \quad \forall \vv \in \Hcurlo. \label{eq-variational_form}
\end{eqnarray}

\section{Edge \acp{fe}}\label{sec-ned_fes}
Let $\mathcal{T}_h$ be a quasi-uniform partition of $\Omega$ into a set of hexahedra (quadrilaterals in 2D) or tetrahedra (triangles in 2D) geometrical cells $\geophy$. For every $\geophy \in \mathcal{T}_h$, we denote by $h_\geophy$ its diameter and set the characteristic mesh size as $h = \max_{\geophy \in \mathcal{T}_h } h_\geophy$.
Let us denote by $v\in\mathcal{N}$, $e\in\mathcal{E}$ and $f\in\mathcal{F}$ the components and global sets of vertices, edges and faces of $\mathcal{T}_h$, with cardinality $n_\mathcal{N}, n_\mathcal{E}$ and $n_\mathcal{F}$, resp. Using Ciarlet's definition, a \ac{fe} is represented by the triplet $ \{ \geophy, \esp, \Sigma \}$, where $\esp$ is the space of functions on $\geophy$ and $\Sigma$ is a set of linear functionals on $\esp$. The elements of $\Sigma$ are called \acp{dof} (moments) of the \ac{fe}. We will denote the number of functionals on the cell as $n_\Sigma$, and $\Sigma \doteq \{ \sigma_a \}_{a=1}^{n_\Sigma}$.
$\Sigma$ is a basis for $\esp'$, which is dual to the so-called basis of shape functions $\{\phi^a\}_{a=1}^{n_\Sigma}$ for $\esp$, i.e., $\sigma_a(\phi^b)=\delta_{ab}$,  $\forall a,b \in \{ 1, \ldots, n_\Sigma \}$.

At this point, we must distinguish between the \emph{reference} \ac{fe}, built on a reference cell, and the \ac{fe} in the physical space. Our implementation of the space of functions and moments is based on a unique {reference} \ac{fe} $\{ \georef, \reffesp, \hat{\Sigma} \}$. Then, in the physical space, the \ac{fe} triplet on a cell $\geophy$ relies on its {reference} \ac{fe}, a geometrical mapping $\geomap_K$ such that $K=\geomap({\hat{K}})$ and a linear bijective function mapping $\hat{\funmap}_\geophy: \reffesp \rightarrow \reffesp$. It is well known that quadrilateral \acp{fe} may result in a loss of accuracy on general meshes, e.g., when elements are given as images of hexahedra under invertible bilinear maps, in comparison with the accuracy achieved with squares or cubes~\cite{boffi_2002}. In this work, hexahedra in the physical space are obtained with affine transformations from the reference square or cube, thus optimal convergence properties are guaranteed. Nevertheless, this fact does not restrict the simulations to simple geometries, since unfitted approaches~\cite{badia_unfitted_2017} may be employed to model complex geometries with structured background meshes. Let us denote by $\jacobian$ the Jacobian of the geometrical mapping, i.e., $ \jacobian \doteq \frac{\partial \geomap_K}{\partial \x}$. The functional space is defined as $\esp \doteq \{ \hat{\funmap}_\geophy(\hat{v}) \circ \geomap^{-1}_\geophy : \, \hat{v} \in \reffesp\}$; we will also use the mapping ${\funmap}_\geophy: \reffesp \rightarrow \esp$ defined as ${\funmap}_\geophy(\hat{v}) \doteq \hat{\funmap}_\geophy (\hat{v}) \circ \geomap^{-1}_\geophy$. Finally, the set of \acp{dof} on the physical \ac{fe} is defined as $\Sigma \doteq \{ \hat{\sigma} \circ {\funmap}_\geophy^{-1} \, : \, \hat{\sigma} \in \hat{\Sigma} \}$ from the set of the {reference} \ac{fe} linear functionals. In the following subsections, we go into detail into these concepts and provide some practical examples.

\subsection{{Reference} cell}\label{subsec-ref_fe}
A \emph{polytope} is mathematically defined as the convex hull of a finite set of points (vertices). The concept of polytope may be of practical importance, because it allows one to develop codes that can be applied to any topology of arbitrary dimension that fits into the framework, see~\cite{badia-fempar} for a thorough exposition. For the sake of ease, we restrict ourselves to two possible polytopes: $d$-cubes and $d$-simplices (with $d=2,3$), defined as the convex hull of a set of $n_v=2^d$ and $n_v=d+1$ geometrically distinct points, respectively. An \emph{ordered} set of vertices $\{\vv_1, \ldots, \vv_{n_v}\}$ not only defines the topology of the cell $\georef$ but an \emph{orientation}. Edges and faces are polytopes of lower dimension, and an ordered set of their vertices also defines their orientation.
For the sake of illustration, the local indexing at the cell level is depicted in Figs. \ref{fig-3D_hex_vefs} and \ref{fig-3D_tet_vefs} for a 3-cube and a 3-simplex, resp. The 2D cases follow by simply considering the restriction of the 3D cell to one of its faces.

\subsection{Polynomial spaces}\label{subsec-polspaces}
Local \ac{fe} spaces are usually spanned by polynomial functions. Let us start by defining basic polynomial spaces that will be needed in the forthcoming definitions.
{ We define here all polynomial spaces in an arbitrary polytope $K$, but they will only be used for the reference cell in the next sections. }The space of polynomials of degree less than or equal to $k>0$ in all the variables $\{x_i\}_{i=1}^d$ is denoted by $\mathcal{Q}_k(K)$. Analogously, we can define the space of polynomials of degree less than or equal to $\{k_i\}_{i=1}^d$ for the variable $\{x_i\}_{i=1}^d$, denoted by $\Q_{\boldsymbol{k}}(K)$, with $\boldsymbol{k}=[k_1, \ldots, k_d]$. Clearly, the dimension of this space, denoted by dim($\mathcal{Q}_{\boldsymbol{k}}(K)$), is $\prod_{i=1}^d (k_i+1)$. Let us also define the corresponding truncated polynomial space $\mathcal{P}_k(K)$ as the span of the monomials with degree less than or equal to $k$. To determine the dimension of the truncated space, we note that the number of components can be expressed with the triangular or tetrahedral number $T^d_n$, $d=\{2,3\}$, such that
\begin{eqnarray}\label{eq-trinum}
{\rm dim}(\mathcal{P}_k(K))=T^d_{k+1},  \qquad T_n^d = \frac{\prod_{i=1}^{d}(n+i-1)}{d!}.
\end{eqnarray}

\subsubsection{Construction of polynomial spaces}\label{subsec-lag_polys}
Given an order $k$, it is trivial to form the set of monomials $q^k = \{x^i\}_{i=0}^k$, that spans the 1D space $\mathcal{Q}_k(K)$. We construct a basis spanning the multi-dimensional space $\mathcal{Q}_{\boldsymbol{k}}(K)$ with a Cartesian product of the monomials for each dimension, i.e., $\{ q_1^{k_1} \times \ldots \times q_d^{k_d} \}$, thus we have $\mathcal{Q}_{\boldsymbol{k}}(K) = {\rm span} \{ \prod_{i=1}^d x_i^{\alpha_i} \hspace{0.2cm} {\rm s.t. } \hspace{0.2cm} \alpha_i \leq k_i  \}$. Let us denote by $|\alpha|$ the summation of the exponents for a given monomial. Then, the multi-dimensional truncated space is defined as $\mathcal{P}_k(K) = {\rm span} \{ \prod_{i=1}^d x_i^{\alpha_i} \hspace{0.2cm} {\rm s.t. } \hspace{0.2cm} |\alpha| \leq k  \}$.

Let us also define Lagrangian polynomials spanning $\mathcal{Q}_{\boldsymbol{k}}(K)$, which will be used in the definition of moments in \sect{subsec-ned_moms}. Given an order $k$ and a set $\mathcal{N}_k$ of different nodes in $\mathbb{R}$, usually equidistant in the interval $[x_0, x_k]$, we can define the set of Lagrangian polynomials $\{\ell_i\}_{i=0}^k$ as,
\begin{eqnarray}
\ell_i &:=& \displaystyle\frac{ \prod_{n\in \mathcal{N}_k \setminus i}( x - x_n )}{\prod_{n\in \mathcal{N}_k \setminus i} (x_i - x_n)}, \label{eq-lag_pols}
\end{eqnarray}
where we indistinctly represent nodes by their index $i$ or its position $x_i$. The set of all polynomials $\{\ell_i^k\}_{i=0}^k$ defines the Lagrangian basis $L^k$. In the grad-conforming (Lagrangian) \ac{fe}, the definition of moments $\{ \sigma_a \}_{a=1}^{k}$ simply consists in the evaluation of the functions on the given points $x \in \mathcal{N}_k$. Clearly, polynomials in \eq{eq-lag_pols} evaluated at points $x_a$ satisfy the duality $\sigma_a(\ell_i) = \ell_i(x_a) = \delta_{ai}$ for every $a,i \in \{0, \ldots k\}$, i.e., they are shape functions. For multi-dimensional spaces, we define the set of nodes as the Cartesian product of 1D nodes. Given a $d$-dimensional space of order $\boldsymbol{k}=[k_1, \ldots, k_d]$, the set of nodes is defined as $\mathcal{N}^{\boldsymbol{k}} = \mathcal{N}^{k_1}\times \ldots \times \mathcal{N}^{k_d}$.

\subsubsection{Hexahedra}
The space of functions on the cell $\esp_{k}({K})$ for this sort of elements is defined as the space of gradients for the scalar polynomial space $\Q_k(K)$, i.e.,
\begin{eqnarray}
 \esp_{k}({K}) &:=& \{ \Q_{k-1,k}({K}) \times \Q_{k,k-1}({K}) \}, \label{eq-ps_hex1}\\
 \esp_{k}({K}) &:=& \{ \Q_{k-1,k,k}({K}) \times  \Q_{k,k-1,k}({K}) \times \Q_{k,k,k-1}({K}) \}, \label{eq-ps_hex}
\end{eqnarray}
in 2D and 3D, resp. Let us illustrate the polynomial space with a couple of examples.
\begin{Example} The polynomial space for the lowest order ($k=1$) edge hexahedral element is defined as
\[
\esp_{k}({\geophy}) =  \{ \Q_{0,1,1}({K}) \times  \Q_{1,0,1}({K}) \times \Q_{1,1,0}({K}) \}
\]
which can be represented as the span of the set of vector-valued functions
\begin{eqnarray*}
\left\lbrace
\begin{bmatrix}
1 \\ 0 \\ 0
\end{bmatrix},
\begin{bmatrix}
x_2 \\ 0 \\ 0
\end{bmatrix},
\begin{bmatrix}
x_3 \\ 0 \\ 0
\end{bmatrix},
\begin{bmatrix}
x_2 x_3 \\ 0 \\ 0
\end{bmatrix},
\begin{bmatrix}
0 \\ 1 \\ 0
\end{bmatrix},
\begin{bmatrix}
0 \\ x_1 \\ 0
\end{bmatrix},
\begin{bmatrix}
0 \\ x_3 \\ 0
\end{bmatrix},
\begin{bmatrix}
0  \\ x_1 x_3 \\ 0
\end{bmatrix},
\begin{bmatrix}
0 \\ 0 \\ 1
\end{bmatrix},
\begin{bmatrix}
0 \\ 0 \\ x_1
\end{bmatrix},
\begin{bmatrix}
0 \\ 0 \\ x_2
\end{bmatrix},
\begin{bmatrix}
0  \\ 0 \\ x_1 x_2
\end{bmatrix}
\right\rbrace.
\end{eqnarray*}
\end{Example}

\begin{Example} The polynomial space for the second order quadrilateral element is defined as
\[
\esp_{k}({K}) =  \{ \Q_{1,2}({K}) \times  \Q_{2,1}({K}) \},
\]
which can be represented by the spanning set
\begin{eqnarray*}
\left\lbrace
\begin{bmatrix}
1 \\ 0
\end{bmatrix},
\begin{bmatrix}
x_1 \\ 0
\end{bmatrix},
\begin{bmatrix}
x_2 \\ 0
\end{bmatrix},
\begin{bmatrix}
x_1 x_2 \\ 0
\end{bmatrix},
\begin{bmatrix}
x_2^2 \\ 0
\end{bmatrix},
\begin{bmatrix}
x_1 x_2^2 \\ 0
\end{bmatrix},
\begin{bmatrix}
 0 \\ 1
\end{bmatrix},
\begin{bmatrix}
 0 \\ x_1
\end{bmatrix},
\begin{bmatrix}
 0 \\ x_1^2
\end{bmatrix},
\begin{bmatrix}
 0 \\ x_2
\end{bmatrix},
\begin{bmatrix}
 0 \\ x_1 x_2
\end{bmatrix},
\begin{bmatrix}
 0 \\ x_1^2 x_2
\end{bmatrix}
\right\rbrace.
\end{eqnarray*}
\end{Example}

\subsubsection{Tetrahedra}\label{subsec-poltet}
The polynomial space for tetrahedral elements is slightly more involved. For the sake of brevity, let us omit the cell (i.e., $K$) in the notation for the local \ac{fe} spaces. Let us start by defining the homogeneous polynomial space $[\mathcal{\tilde{P}}_k]^d:=[\mathcal{P}_k]^d \setminus [\mathcal{P}_{k-1}]^d$, where $[\mathcal{P}_k]^2 = \mathcal{P}_k \times \mathcal{P}_k$ and $[\mathcal{P}_k]^3 = \mathcal{P}_k \times \mathcal{P}_k \times \mathcal{P}_k$. The space $[\mathcal{\tilde{P}}_k]^d$ has dimension $d\cdot {\rm dim}(\mathcal{\tilde{P}}_k)$, and using \eq{eq-trinum} we have:
\begin{eqnarray}\label{eq-homonum}
{\rm dim}(\mathcal{\tilde{P}}_k) = {\rm dim}({\mathcal{P}_k}) - {\rm dim}({\mathcal{P}_{k-1}})
                                 = T_{k+1}^d - T_k^d \nonumber
                                 = T^{d-1}_{k+1}.  \nonumber
\end{eqnarray}
The function space of order $k$ on a tetrahedron is then defined as
\begin{eqnarray}\label{eq-ps_tet}
\esp_{k}({K}) = [\mathcal{P}_{k-1}]^d \oplus \mathcal{S}_k,
\end{eqnarray}
where $\mathcal{S}_k$ is the space of polynomials
\[
\mathcal{S}_k := \{ p(\x) \in [\tilde{\mathcal{P}}_k]^d \text{ such that } p(\x)\cdot \x = 0 \}.
\]
Note that if $ p(\x)\in [\tilde{\mathcal{P}}_k]^d$, then $p\cdot \x \in \tilde{\mathcal{P}}_{k+1}$ and any polynomial in $\mathcal{\tilde{P}}_{k+1}$ may be written in this way. The dimension of the space $\mathcal{S}_k$ is
\begin{eqnarray}\label{eq-dimSk}
{\rm dim}(\mathcal{S}_k) = d \cdot {\rm dim}(\mathcal{\tilde{P}}_k) - {\rm dim}(\mathcal{\tilde{P}}_{k+1}),
\end{eqnarray}
which leads to $(k+2)k$ in 3D and $k$ in 2D, resp. Among all the possible forms of representing the space $\mathcal{S}_k$ in 3D, we consider the following spanning set
\begin{subequations}\label{eq-basis_Sk}
\begin{align}
\mathcal{S}_k = {\rm span} & \left\lbrace \bigcup_{\beta=1}^{k} \bigcup_{\alpha=1}^{k+1-\beta}
\left( \begin{bmatrix}
 -x_1^{\alpha-1} x_2^{k-\alpha-\beta+2} x_3^{\beta-1}  \\
  x_1^{\alpha} x_2^{k-\alpha-\beta+1} x_3^{\beta-1}  \\
   0
\end{bmatrix},
\begin{bmatrix}
 -x_1^{k-\alpha-\beta+1} x_2^{\beta-1} x_3^{\alpha}  \\
   0                                          \\
  x_1^{k-\alpha-\beta+2} x_2^{\beta-1} x_3^{\alpha-1}
\end{bmatrix}\right),\right.\label{eq-basis_Sk1}  \\
& \left. \bigcup_{\alpha=1}^k  \begin{bmatrix}
0                                                      \\
- x_1^0 x_2^{\alpha-1}  x_3^{k-\alpha+1}                \\
  x_1^0 x_2^{\alpha}    x_3^{k-\alpha}
\end{bmatrix} \right\rbrace. \label{eq-basis_Sk2}
\end{align}
\end{subequations}
\begin{proposition}
The set of vector functions defined in \eq{eq-basis_Sk} forms a basis of the space $\mathcal{S}_k$.
\end{proposition}

\begin{proof}
First, we note that the proposed basis contains $\{ p_i(\x) \}_{i=1}^{(k+2)k}$ vector functions. Clearly, the total degree of the monomials found on each component for all functions is $k$, thus $S \subset \tilde{\mathcal{P}}_k^d$. Further, it is easy to check that $ p_i(\x)\cdot \x=0$ for any $p_i(\x) \in S$. The proof is completed by showing that all the functions are linearly independent. It is trivial to see that all the functions in the first set of functions are indeed linearly independent; the two sets have different non-zero components. Finally, vector functions from the last subset \eq{eq-basis_Sk2} are independent of $x_1$, whereas all functions in the previous subset in \eq{eq-basis_Sk1} do depend on $x_1$ and are linearly independent among them.
\end{proof}
In two dimensions, the analytical expression of the spanning set simply reduces to
\begin{align}
S_k = {\rm span} \left\lbrace \bigcup_{\alpha=1}^{k}
\begin{bmatrix}
- x_1^{\alpha-1} x_2^{k-\alpha+1}                \\
  x_1^{\alpha}   x_2^{k-\alpha}
\end{bmatrix} \right\rbrace.
\end{align}

The dimension of the space $\esp_{k}({K})$ for tetrahedra can be obtained by adding (\ref{eq-trinum}) and (\ref{eq-dimSk}),
\begin{eqnarray}\label{eq-dimVk}
{\rm dim}(\esp_{k}(K))=\frac{k \prod_{i=2}^d ( k+i ) }{(d-1)!}.
\end{eqnarray}
Let us give some examples of polynomial bases for tetrahedral edge elements spanning the requested spaces.
\begin{Example}\label{ex-tet_3D_ps}
A set of polynomial spanning $\esp_{1}(K) = [P_0]^3 \oplus \mathcal{S}_1$ (i.e., lowest order tetrahedral element) is
\begin{eqnarray*}
\left\lbrace
\begin{bmatrix}
1 \\ 0 \\ 0
\end{bmatrix},
\begin{bmatrix}
0 \\ 1 \\ 0
\end{bmatrix},
\begin{bmatrix}
0 \\ 0 \\ 1
\end{bmatrix},
\begin{bmatrix}
-x_2  \\ x_1 \\ 0
\end{bmatrix},
\begin{bmatrix}
-x_3 \\ 0 \\ x_1
\end{bmatrix},
\begin{bmatrix}
0 \\ -x_3 \\ x_2
\end{bmatrix}
\right\rbrace.
\end{eqnarray*}
\end{Example}

\begin{Example}\label{ex-tet_2D_ps}
A polynomial base spanning $\esp_{2}(K) = [P_1]^2 \oplus \mathcal{S}_2$ (i.e., second-order triangular element) is
\begin{eqnarray*}
\left\lbrace
\begin{bmatrix}
1 \\ 0
\end{bmatrix},
\begin{bmatrix}
x_1 \\ 0
\end{bmatrix},
\begin{bmatrix}
x_2 \\ 0
\end{bmatrix},
\begin{bmatrix}
0 \\ 1
\end{bmatrix},
\begin{bmatrix}
0 \\ x_1
\end{bmatrix},
\begin{bmatrix}
0 \\ x_2
\end{bmatrix},
\begin{bmatrix}
-x_2^2 \\ x_1 x_2
\end{bmatrix},
\begin{bmatrix}
-x_1 x_2 \\ x_1^2
\end{bmatrix}
\right\rbrace.
\end{eqnarray*}
\end{Example}
In both cases, vector functions that span the subspaces $[\mathcal{P}_{k-1}]^d \subset \esp_k(K)$ and $\mathcal{S}_k \subset \esp_k(K)$ can easily be identified from the full sets of vector-valued functions in Exs. \ref{ex-tet_3D_ps} and \ref{ex-tet_2D_ps}.

Note that  local spaces $\esp_k({K})$ lie between the full polynomial spaces of order $k-1$ and $k$, i.e., $[\Q_{k-1}({K})]^d \subset \esp_k({K}) \subset [\Q_{k}({K})]^d$, $[\P_{k-1}({K})]^d \subset \esp_k({K}) \subset [\P_{k}({K})]^d$ for hexahedra and tetrahedra, resp. This kind of elements are called edge \acp{fe} of the first kind~\cite{nedelec_mixed_1980}. Another edge \ac{fe}, the so-called second kind, was introduced also by N\'ed\'elec in~\cite{nedelec_new_mixed_1986}. It follows similar ideas but considers full polynomial spaces, i.e.,  $[\mathcal{Q}_k(K)]^d$ or $[\mathcal{P}_k(K)]^d$, instead of anisotropic polynomial spaces (see \eq{eq-ps_hex}) or incomplete polynomial spaces (see \eq{eq-ps_tet}) for hexahedra and tetrahedra, resp., which clearly simplifies the implementation of the polynomial spaces. Edge elements of the second kind offer better constants in the error estimates at the cost of increasing the number of \acp{dof} (see detailed definitions in~\cite[Ch.\ 2]{Cohen2017}).

\subsection{Edge moments in the reference element}\label{subsec-ned_moms}
In order to complete the definition of edge \acp{fe}, it remains to define a set of (linearly independent) functionals to be applied on $\esp_k({\hat{K}})$. Edge moments are integral quantities over geometrical sub-entities of the cell against some test functions. These moments involve directions in 1D entities (edges) and orientations in 2D (faces) or 3D (volumes). Thus, we need to define local orientations for the lower dimension geometrical entities of the {reference} cell, see Figs.~\ref{fig-3D_hex_vefs} and \ref{fig-3D_tet_vefs}. Let us distinguish between three different subsets of moments $\hat{\Sigma}$, namely \acp{dof} related to edges $\sigma_{\hat{e}}$, to faces $\sigma_{\hat{f}}$ and volume $\sigma_{\hat{K}}$ such that $\hat{\Sigma} = \sigma_{\hat{e}} \cup \sigma_{\hat{f}} \cup \sigma_{\hat{K}}$.

\subsubsection{Hexahedra}\label{subsubsec-hex-ref-moms}
The set of functionals, for the {reference} element, that form the basis in two dimensions reads (\cite[Ch.\ 6]{monk_finite_2003}):
\begin{subequations}
\begin{eqnarray}
\sigma_{\hat{e}}(\hat{\uu}_h)&:= & \int_{\hat{e}} ( \hat{\uu}_h\cdot \hat{\boldsymbol{\tau}} ) \hat q  \quad \forall \hat q \in \mathcal{P}_{k-1}(\hat{e}), \quad \forall \hat{e} \in \georef, \label{eq-hexnedge_2D} \\
\sigma_{\hat{K}}(\hat{\uu}_h)&:= &  \int_{{\georef}}  \hat{\uu}_h \cdot \hat{\boldsymbol{q}} \quad \forall \hat{\boldsymbol{q}} \in \Q_{k-1,k-2}(\hat{K}) \times \Q_{k-2,k-1}(\hat{K}). \label{eq-hexninner_2D}
\end{eqnarray}
\end{subequations}
We will have $4k$ \acp{dof} associated to edges and $2k(k-1)$ inner \acp{dof}. Therefore, the complete set of functionals has cardinality $n_{\hat{\Sigma}}=2k(k+1)$. In the case of $k=1$, only edge \acp{dof} $\sigma_{\hat{e}}$ appear. In three dimensions the set is defined as
\begin{subequations}
\begin{align}
\sigma_{\hat{e}}(\hat{\uu}_h):= & \int_{\hat{e}} ( \hat{\uu}_h \cdot \boldsymbol{\hat{\tau}} ) \hat q  \qquad \forall \hat q \in \mathcal{P}_{k-1}(\hat{e}), \quad \forall \hat{e} \in \georef \label{eq-hexnedge} \\
\sigma_{\hat{f}}(\hat{\uu}_h):= & \int_{\hat{\mathcal{F}}} ( \hat{\uu}_h \times \hat{\n} )\cdot \hat{\boldsymbol{q}} \qquad \forall \hat{\boldsymbol{q}} \in \Q_{k-2,k-1}(\hat{\mathcal{F}})\times \Q_{k-1,k-2}(\hat{\mathcal{{F}}}), \quad \forall \hat{\mathcal{F}} \in \georef \label{eq-hexneface} \\
\begin{split}
\sigma_\georef(\hat{\uu}_h):= & \int_{{\georef}}  \hat{\uu}_h \cdot \hat{\boldsymbol{q}}  \qquad \forall \hat{\boldsymbol{q}} \in \Q_{k-1,k-2,k-2}({\georef})\times \Q_{k-2,k-1,k-2}({\georef}) \times \label{eq-hexninner}
 \\
&  \Q_{k-2,k-2,k-1}({\georef}),
\end{split}
\end{align}
\end{subequations}
where $\hat{\boldsymbol{\tau}}$ is the unit vector along the edge and $\hat{\n}$ the unit normal to the face. In this case, we have $12k$ \acp{dof} associated to edges, $6\cdot 2k(k-1)$ \acp{dof} associated to the 6 faces of the cell and $3k(k-1)^2$ inner \acp{dof}. We have a total number of $n_{\hat{\Sigma}}=3k(k+1)^2$ \acp{dof}. Note that in the case of the lowest order elements, i.e., $k=1$, only \acp{dof} associated to edges appear. For higher order elements, i.e., $k\geq 2$, all kinds of \acp{dof} occur in both dimensions.
\begin{figure}[t!]
    \centering
    \begin{subfigure}[t]{0.3\textwidth}
        \includegraphics[width=\textwidth]{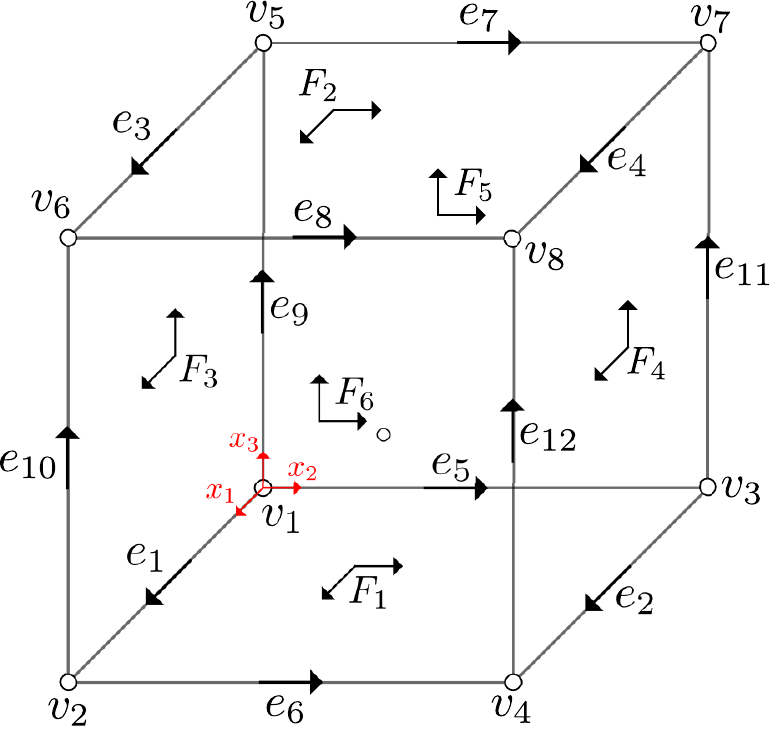}
        \caption{Oriented cell geometry. }
        \label{fig-3D_hex_vefs}
    \end{subfigure} \hspace{0.2cm}
    \begin{subfigure}[t]{0.3\textwidth}
        \includegraphics[width=\textwidth]{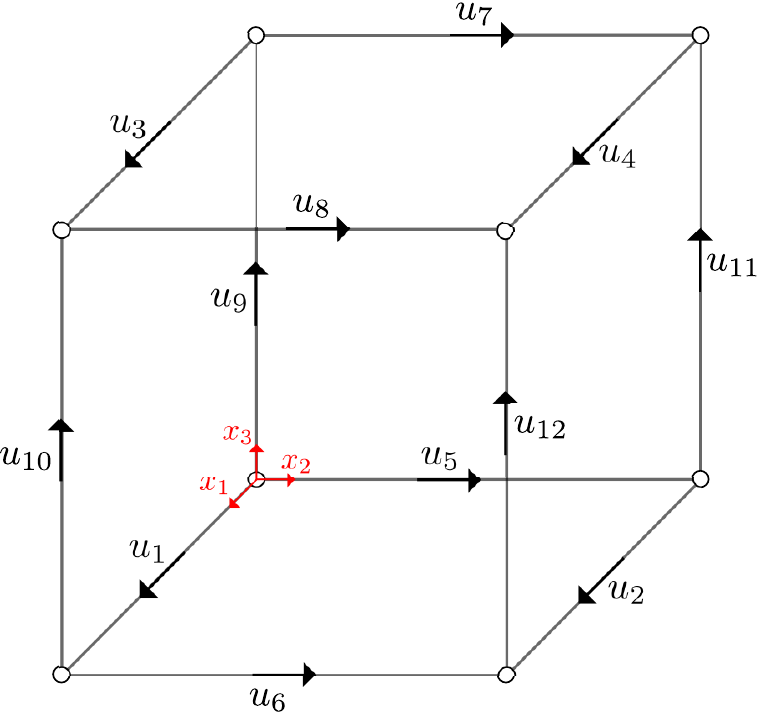}
        \caption{Lowest order \acp{dof}. }
        \label{fig-first_order_hex_dofs}
    \end{subfigure} \hspace{0.2cm}
    \begin{subfigure}[t]{0.3\textwidth}
        \includegraphics[width=\textwidth]{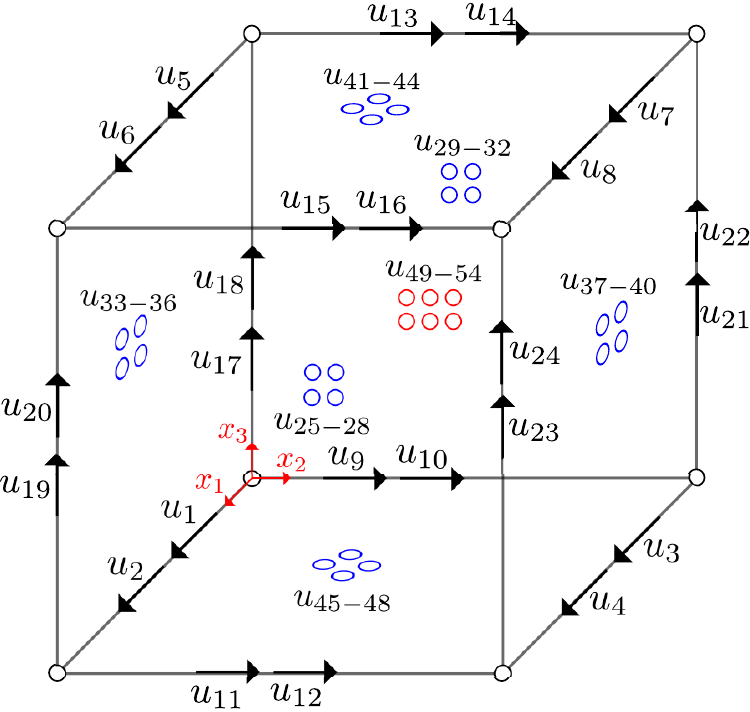}
        \caption{Second order \acp{dof}. Face \acp{dof} depicted in blue, inner \acp{dof} depicted in red. }
        \label{fig-second_order_hex_dofs}
    \end{subfigure}
    \caption{3D hexahedral {reference} \ac{fe}.}\label{fig-3D_hex}
\end{figure}

{\acp{dof} are labelled in the reference cell as follows. First, for every \ac{dof}, we can determine the geometrical entity that \emph{owns} it, e.g., an edge of the reference FE in \eqref{eq-hexnedge}. Within every geometrical entity, there is a one-to-one map between \acp{dof} and test functions. Thus, we can number \acp{dof} by the numbering of the test functions in the test space, e.g., the test functions in $\mathcal{P}_{k-1}(\hat{e})$ for the \acp{dof} in \eqref{eq-hexnedge}. We note that all the test spaces in the \ac{dof} definitions can be built using a nodal-based (Lagrangian) basis, and thus, every \ac{dof} in a geometrical object can be conceptually associated to one and only one node. The numbering of the nodes in the geometrical object is determined by its orientation in the reference FE. The composition of a geometrical object numbering (see Figs.~\ref{fig-3D_hex}--\ref{fig-3D_tet}) and the object-local node numbering provides the local numbering of \acp{dof} within the reference cell.

}

\subsubsection{Tetrahedra}
The set of functionals described in this section follows~\cite[Ch.\ 5]{monk_finite_2003}. In the 2D case, the set of moments for the {reference} \ac{fe} is defined as
\begin{subequations}
\begin{align}
\sigma_{\hat{e}}(\hat{\uu}_h):= & \int_{\hat{e}} ( \hat{\uu}_h \cdot \boldsymbol{\hat{\tau}} ) \hat{q}  \quad \forall \hat{q} \in \mathcal{P}_{k-1}(\hat{e}), \quad \forall \hat{e} \in \georef \label{eq-tetnedge_2D} \\
\sigma_\georef(\hat{\uu}_h):= &  \int_{\georef}  \hat{\uu}_h \cdot \boldsymbol{\hat{q}} \quad \forall \boldsymbol{\hat{q}} \in [\P_{k-2}(\georef)]^2.\label{eq-tetninner_2D}
\end{align}
\end{subequations}
Clearly, we have $3k$ edge \acp{dof} and $k(k-1)$ inner \acp{dof}, which lead to a total number of $n_{\hat{\Sigma}}=k(k+2)$ \acp{dof}. In three dimensions, the set $\hat{\Sigma}$ is defined as
\begin{subequations}\label{eq-tetmoms}
\begin{align}
\sigma_{\hat{e}}(\hat{\uu}_h):= & \int_{\hat{e}} ( \hat{\uu}_h\cdot \hat{\boldsymbol{\tau}} ) \hat{q}  \qquad \forall \hat{q} \in \P_{k-1}(\hat{e}), \quad \forall \hat{e} \in \georef \label{eq-tetnedge} \\
\sigma_{\hat{f}}(\hat{\uu}_h):= & \frac{1}{\| \mathcal{\hat{F}} \|} \int_{\mathcal{\hat{F}}} \hat{\uu}_h \cdot \hat{\boldsymbol{q}} \qquad \forall \hat{\boldsymbol{q}} \in [\P_{k-2}(\georef)]^3 \text{ s.t. } \hat{\boldsymbol{q}}\cdot \hat{\n}=0 , \quad \forall \hat{f} \in \georef \label{eq-tetneface} \\
\sigma_\georef(\hat{\uu}_h):= & \int_\georef  \hat{\uu}_h \cdot \hat{\boldsymbol{q}} \qquad \forall \hat{\boldsymbol{q}} \in [\P_{k-3}(\georef)]^3,\label{eq-tetninner}
\end{align}
\end{subequations}
where $\hat{\boldsymbol{n}}$ in \eq{eq-tetneface} is the unit normal to the face. The set of moments \eq{eq-tetmoms} contains $6k$ edge \acp{dof}, $4k(k-1)$ face \acp{dof} and $\frac{k(k-1)(k-2)}{2}$ inner \acp{dof}. Therefore, the tetrahedral edge element has a total number of $n_{\hat{\Sigma}}=\frac{k(k^2+5k+6)}{2}$ \acp{dof}. The local numbering of \acp{dof} is analogous as for the hexahedral case.

\begin{remark}
The face moments in the definition \eq{eq-tetneface} seem to differ from the rest of definitions in \eq{eq-tetmoms}, which are not scaled with a geometrical entity measure. We follow here~\cite[Ch.\ 5]{monk_finite_2003}, where this expression of face moments is used to prove affine equivalence, i.e., proving that \acp{dof} are affine invariant under the transformation from the {reference}  to the physical element.
\end{remark}

Note that in the case of the lowest order elements, i.e., $k=1$, only \acp{dof} associated to edges occur. For second order elements, i.e., $k=2$, we also find \acp{dof} related to faces (inner \acp{dof} in 2D), and it is in orders higher than two where all kinds of \acp{dof} occur. Note that vector-valued test functions $\hat{\boldsymbol{q}}$ involved in~(\ref{eq-tetninner_2D}), (\ref{eq-tetneface}) and (\ref{eq-tetninner}) can be understood as products between linearly independent vectors $\{ \boldsymbol{\hat{\tau}}_i \}_1^d$ that form a basis in the geometrical entity of dimension $d$ and scalar polynomials $\hat{q}\in \mathcal{P}_{k-d}$.

\begin{figure}[t!]
    \centering
    \begin{subfigure}[t]{0.32\textwidth}
        \includegraphics[width=\textwidth]{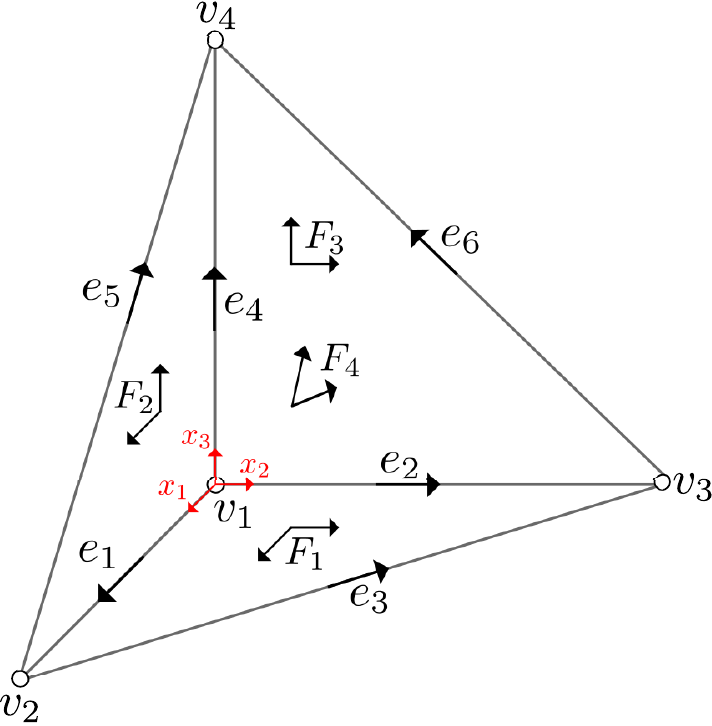}
        \caption{ Oriented cell geometry. }
        \label{fig-3D_tet_vefs}
    \end{subfigure} \hspace{0.2cm}
    \begin{subfigure}[t]{0.3\textwidth}
        \includegraphics[width=\textwidth]{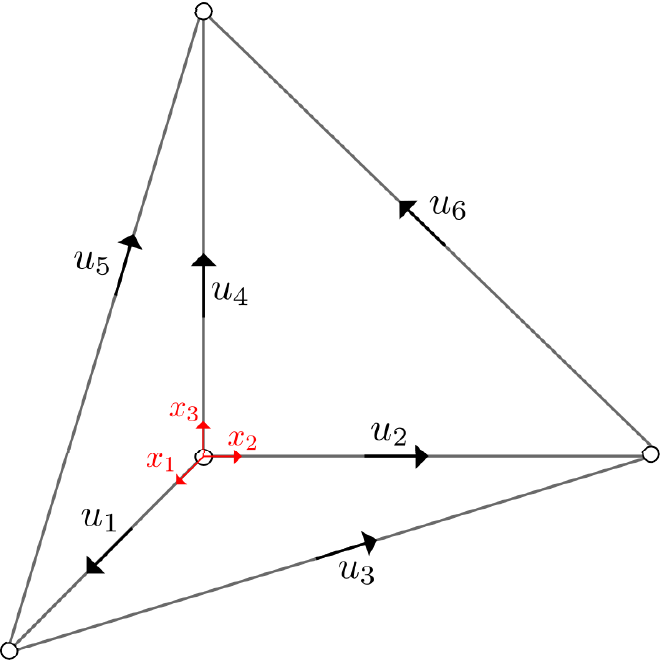}
        \caption{Lowest order \acp{dof}. }
        \label{fig-first_order_tet_dofs}
    \end{subfigure} \hspace{0.2cm}
    \begin{subfigure}[t]{0.3\textwidth}
        \includegraphics[width=\textwidth]{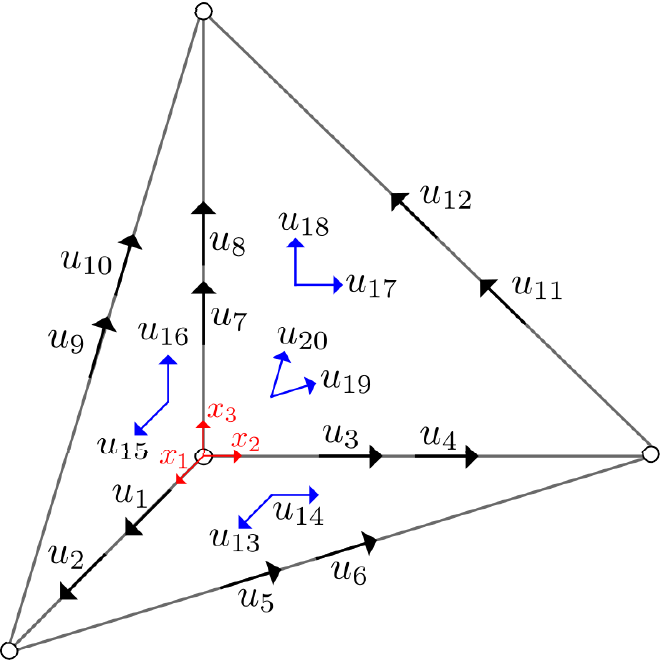}
        \caption{Second order \acp{dof}. Face \acp{dof} depicted in blue. }
        \label{fig-second_order_tet_dofs}
    \end{subfigure}
    \caption{3D tetrahedral {reference} \ac{fe}.}\label{fig-3D_tet}
\end{figure}

\subsection{Construction of edge shape functions}\label{subsec-ned_shapes}
Usually, low order edge elements are implemented via hard-coded expressions of their respective shape functions. However, such an approach is not suitable for  high order edge \acp{fe}, which involve complex analytical expressions of the shape functions. In this section, we provide an automatic generator of arbitrary order shape functions.

The definition of the moments for the edge \ac{fe} (in (\ref{eq-hexnedge_2D},\ref{eq-hexninner_2D}), (\ref{eq-hexnedge},\ref{eq-hexneface},\ref{eq-hexninner}) or (\ref{eq-tetnedge_2D},\ref{eq-tetninner_2D}), (\ref{eq-tetnedge}-\ref{eq-tetninner}) for hexahedra and tetrahedra, resp) requires the selection of functions spanning the requested polynomial space in each case.  First, we generate a \emph{pre-basis} $\{ \varphi^a \}_{a=1}^{n_\Sigma}$ that spans the local \ac{fe} space $\esp_k (\georef)$. To this end, we consider the tensor product of Lagrangian polynomial basis (see Sect. \ref{subsec-lag_polys}) for hexahedra or the combination of a Lagrangian basis of one order less in \eqref{eq-ps_tet}  plus the based of monomials in \eqref{eq-basis_Sk} for tetrahedra. Our goal is to build another (canonical) basis $\{ \phi^a \}_{a=1}^{n_\Sigma}$ that spans the same space and additionally satisfies $\sigma_a(\phi^b) = \delta_{ab}$ for $a,b\in\{1,\ldots,n_\Sigma\}$, i.e., the basis of shape functions for the edge element. Thus, we are interested in obtaining a linear combination of the functions of the \emph{pre-basis} such that the duality between moments and functions is satisfied. As a result, an edge shape function $\phi^a$ can be written as $\phi^a = \sum_{b=1}^{n_\Sigma} Q_{ab} \varphi^b$. Let us make use of Einstein's notation to provide the definition of the change of basis between the two of them. We have:
\begin{eqnarray}
\phi^b = Q_{bc} \varphi^c,   \qquad
\sigma_a(\phi^b) = \sigma_a(Q_{bc} \varphi^c), \qquad
\delta_{ab} = \sigma_a(\varphi^c) Q_{bc},
\end{eqnarray}
{or in compact form $\I = \boldsymbol{C} \boldsymbol{\Q}^T$, thus $\boldsymbol{\Q}^{T}=\boldsymbol{C}^{-1}$. As a result, the edge shape functions are obtained as  $\phi^a = Q_{ab} \varphi^b = C^{-1}_{ba} \varphi^b$}. For the sake of illustration, we provide some examples of full sets of edge shape functions for different orders. To make the visualization easy, we provide only 2D examples for first and second order square (Figs. \ref{fig-NED1_HEX_} and \ref{fig-NED_hex_s2}) and triangular elements (Figs. \ref{fig-NED_s1} and \ref{fig-NED_s2}). In these figures, we put a circle on top of the node corresponding to the \ac{dof} dual to the shape function. The superscript indicates the local moment numbering on the particular geometrical entity. 

\begin{figure}[t!]
    \centering
    \begin{subfigure}[t]{0.24\textwidth}
        \includegraphics[trim={2.0cm 2.0cm 2.0cm 2.0cm},clip,width=\textwidth]{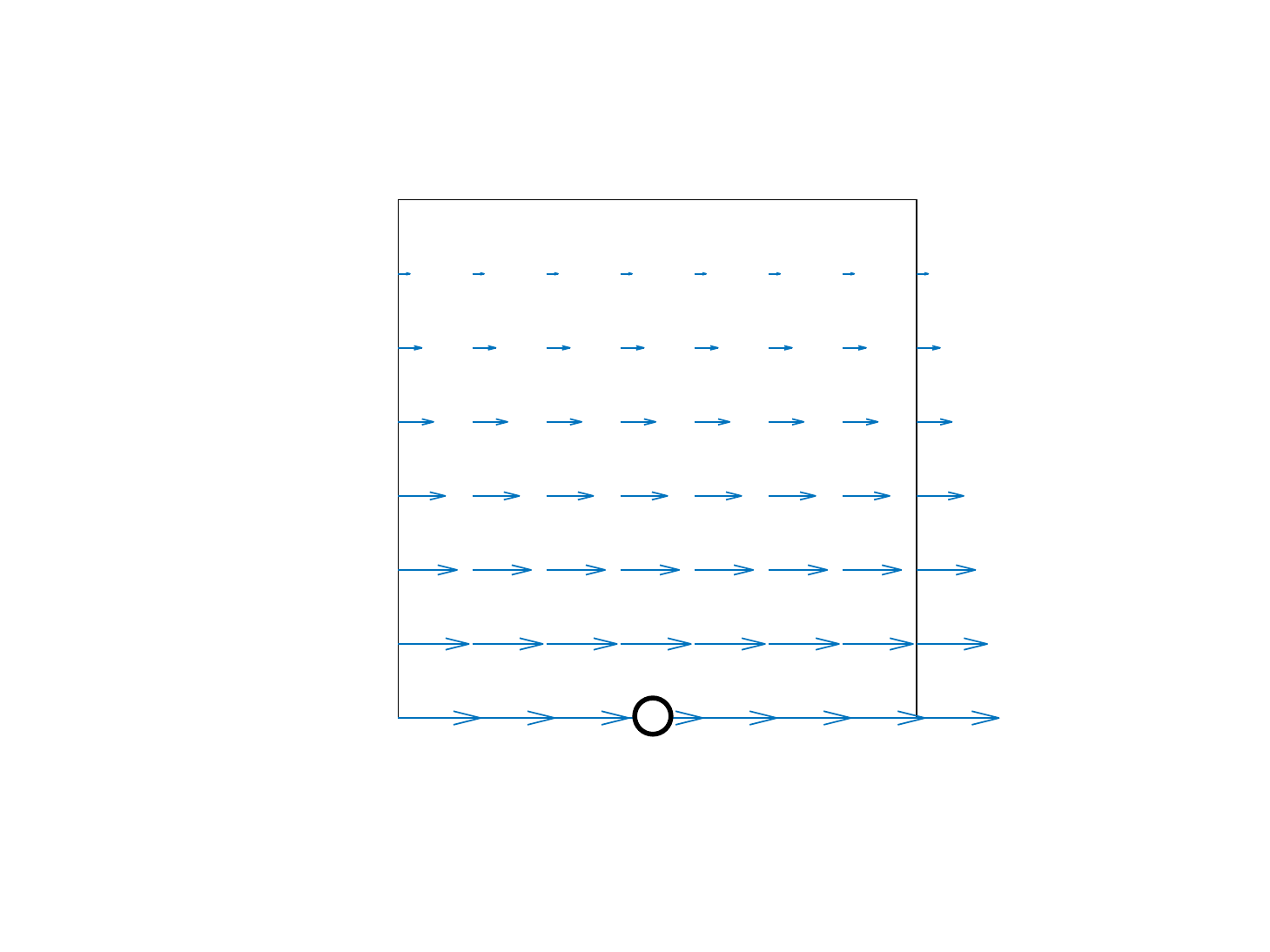}
        \caption{$\phi_{e_1}$ }
        \label{fig-nd1_hex_e1}
    \end{subfigure}
    \begin{subfigure}[t]{0.24\textwidth}
        \includegraphics[trim={2.0cm 2.0cm 2.0cm 2.0cm},clip,width=\textwidth]{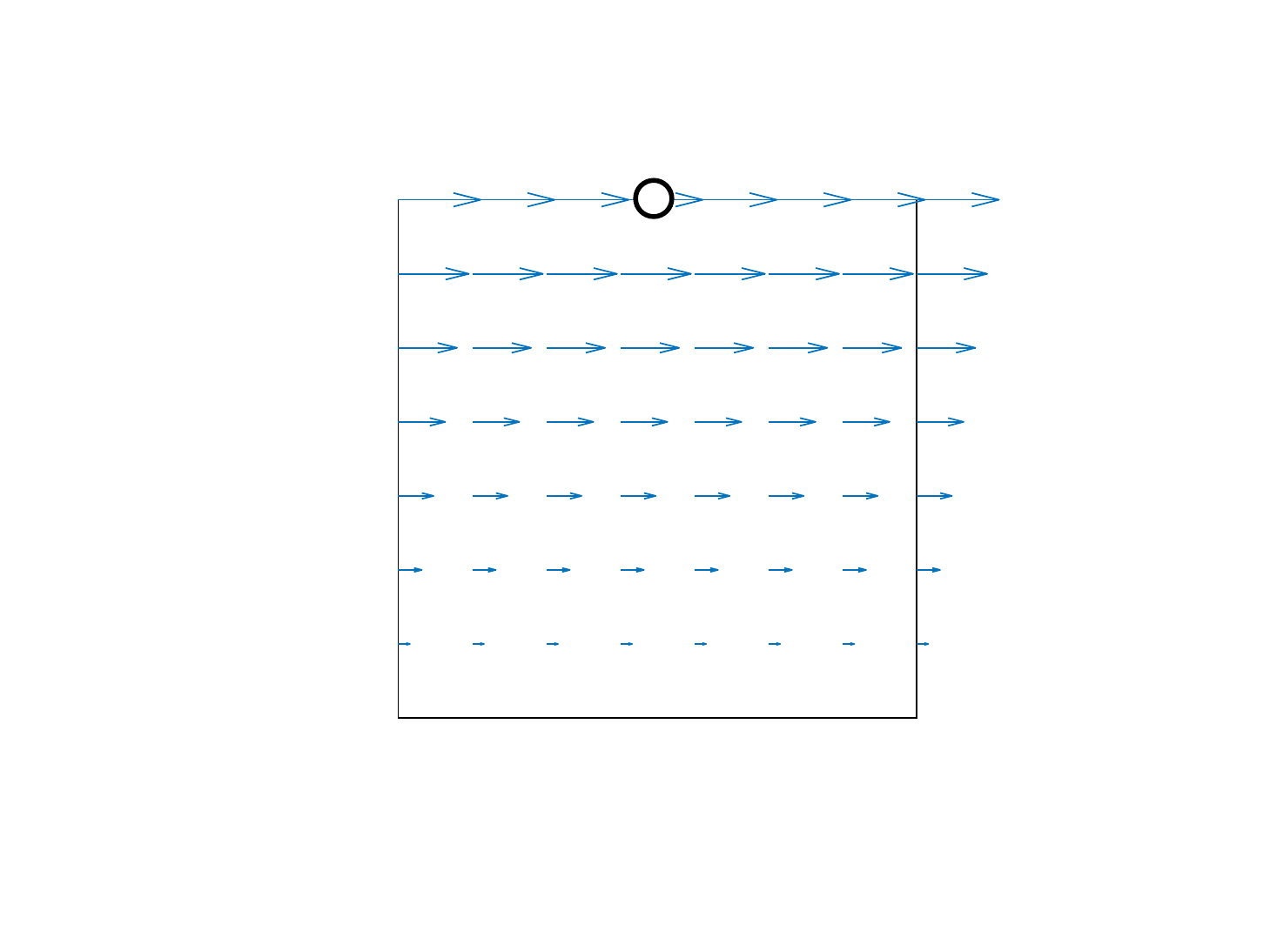}
        \caption{$\phi_{e_2}$ }
        \label{fig-nd1_hex_e2}
    \end{subfigure}
        \begin{subfigure}[t]{0.24\textwidth}
        \includegraphics[trim={2.0cm 2.0cm 2.0cm 2.0cm},clip,width=\textwidth]{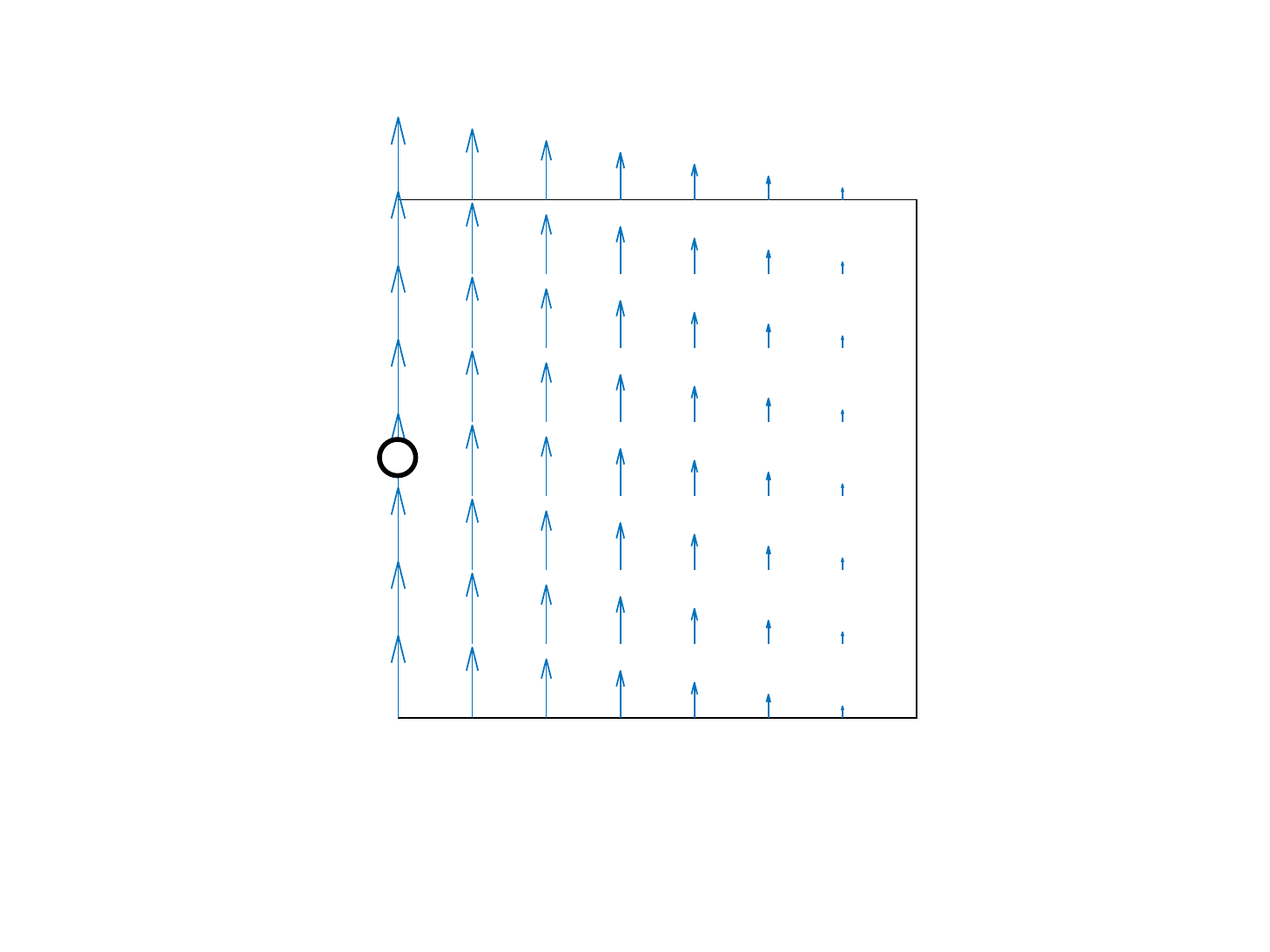}
        \caption{$\phi_{e_3}$}
        \label{fig-nd1_hex_e3}
    \end{subfigure}
        \begin{subfigure}[t]{0.24\textwidth}
        \includegraphics[trim={2.0cm 2.0cm 2.0cm 2.0cm},clip,width=\textwidth]{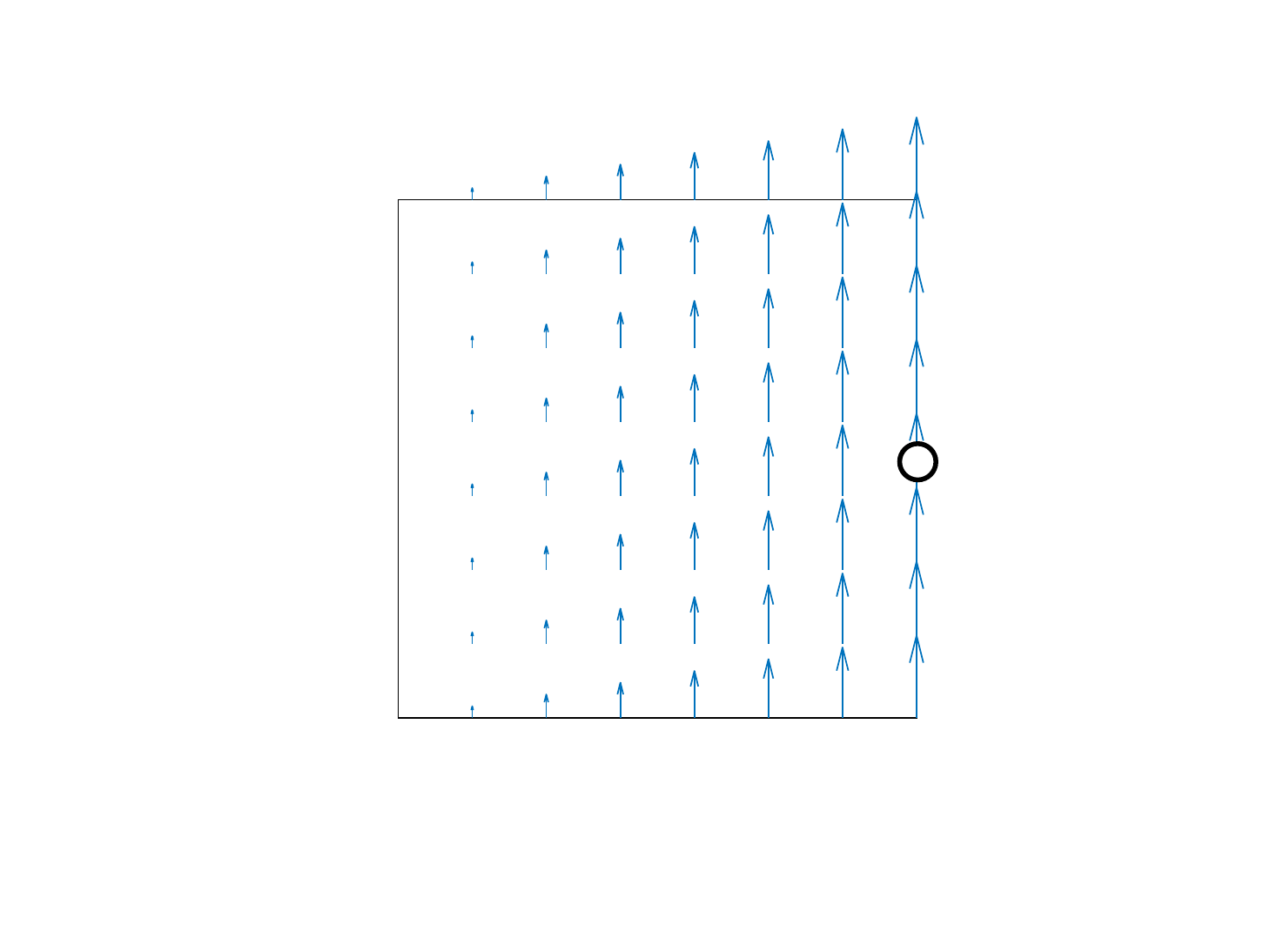}
        \caption{$\phi_{e_4}$ }
        \label{fig-nd1_hex_e4}
    \end{subfigure}
    \caption{Vector-field plots of the shape functions in the lowest order 2D hexahedra. The indices of edges follow the geometrical entities indices for the {reference} hexahedron in \fig{fig-3D_hex_vefs} restricted to the plane $z=0$. Auxiliary circles denote the geometrical entity where the moment is defined.}\label{fig-NED1_HEX_}
\end{figure}

\begin{figure}[t!]
    \centering
    \begin{subfigure}[t]{0.24\textwidth}
   \includegraphics[trim={0.80cm 0.80cm 0.80cm 0.80cm},clip,width=\textwidth]{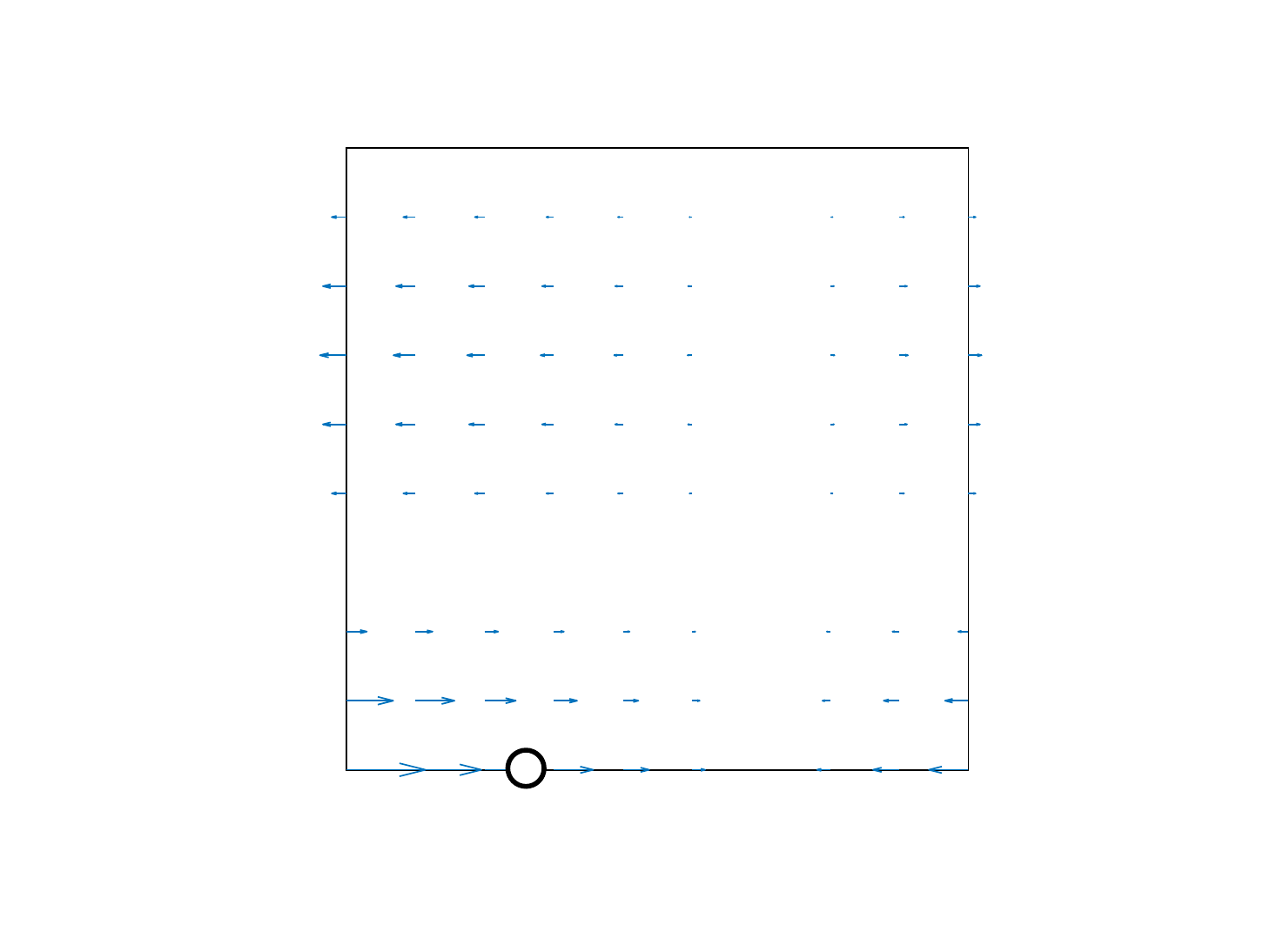}
        \caption{$\phi^1_{e_1}$ }
        \label{fig-nd2_hex_e1}
    \end{subfigure}
    \begin{subfigure}[t]{0.24\textwidth}
   \includegraphics[trim={0.80cm 0.80cm 0.80cm 0.80cm},clip,width=\textwidth]{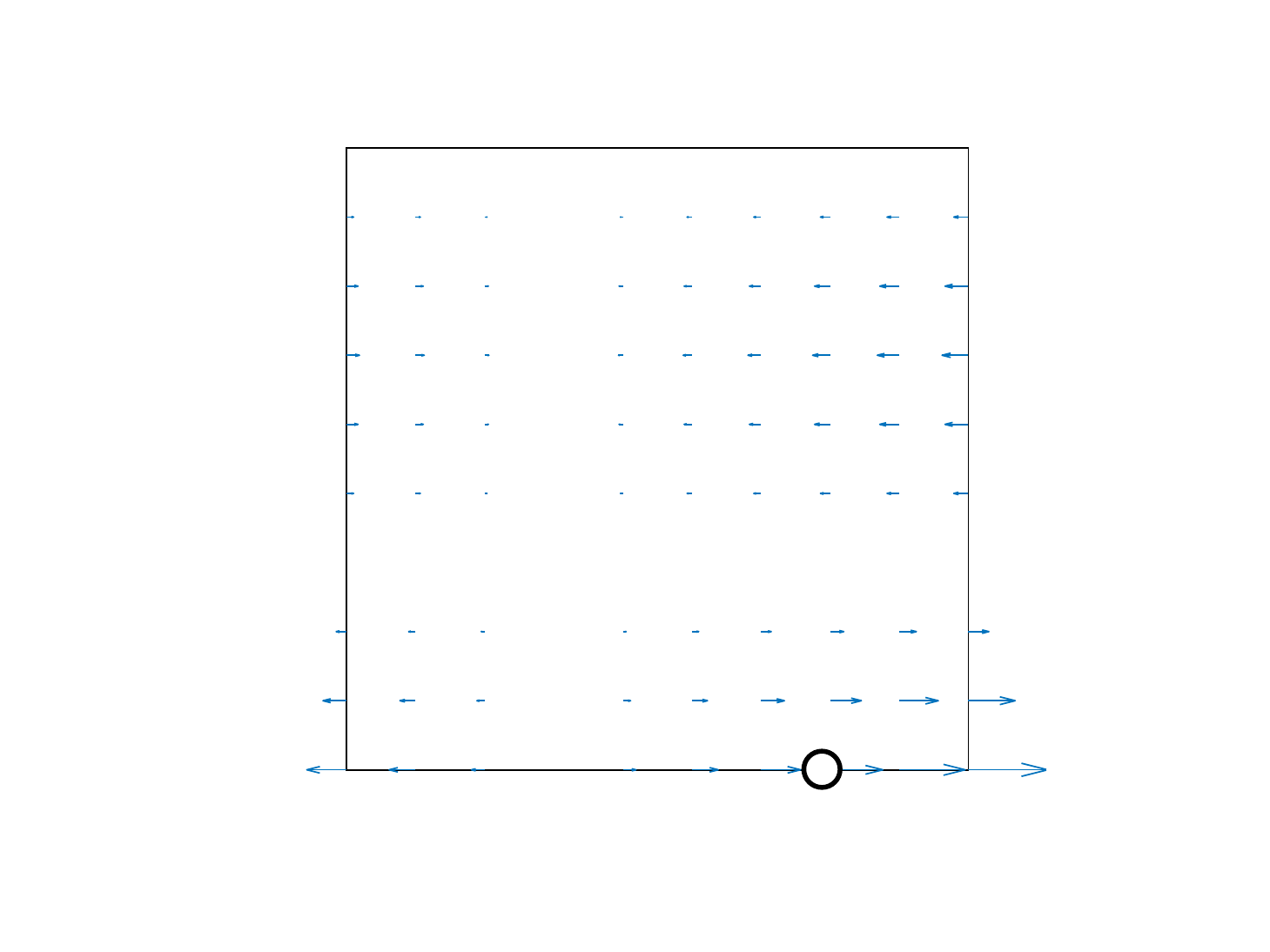}
        \caption{$\phi^2_{e_1}$ }
        \label{fig-nd2_hex_e2}
    \end{subfigure}
        \begin{subfigure}[t]{0.24\textwidth}
    \includegraphics[trim={0.80cm 0.80cm 0.80cm 0.80cm},clip,width=\textwidth]{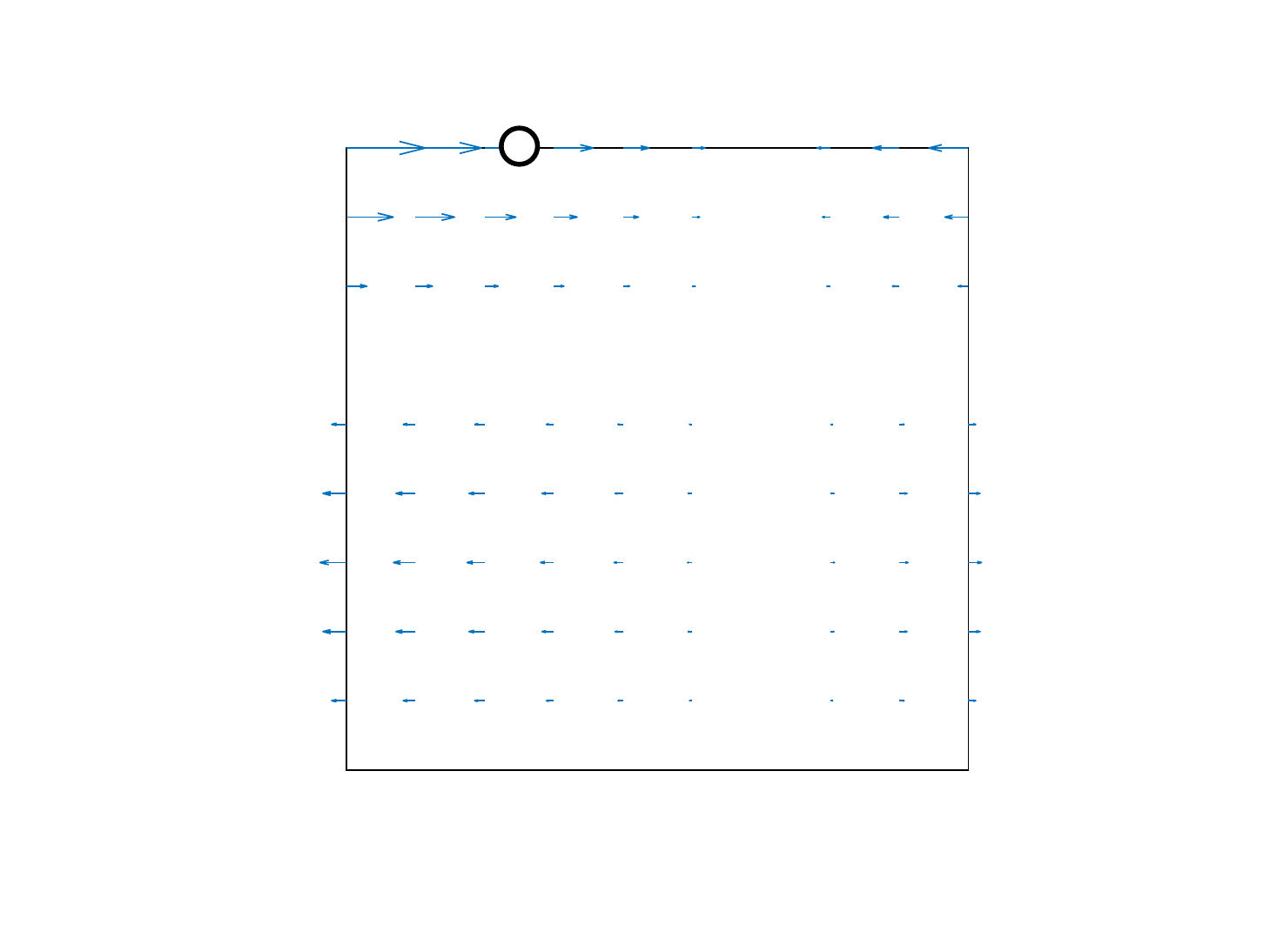}
        \caption{$\phi^1_{e_2}$}
        \label{fig-nd2_hex_e3}
    \end{subfigure}
            \begin{subfigure}[t]{0.24\textwidth}
      \includegraphics[trim={0.80cm 0.80cm 0.80cm 0.80cm},clip,width=\textwidth]{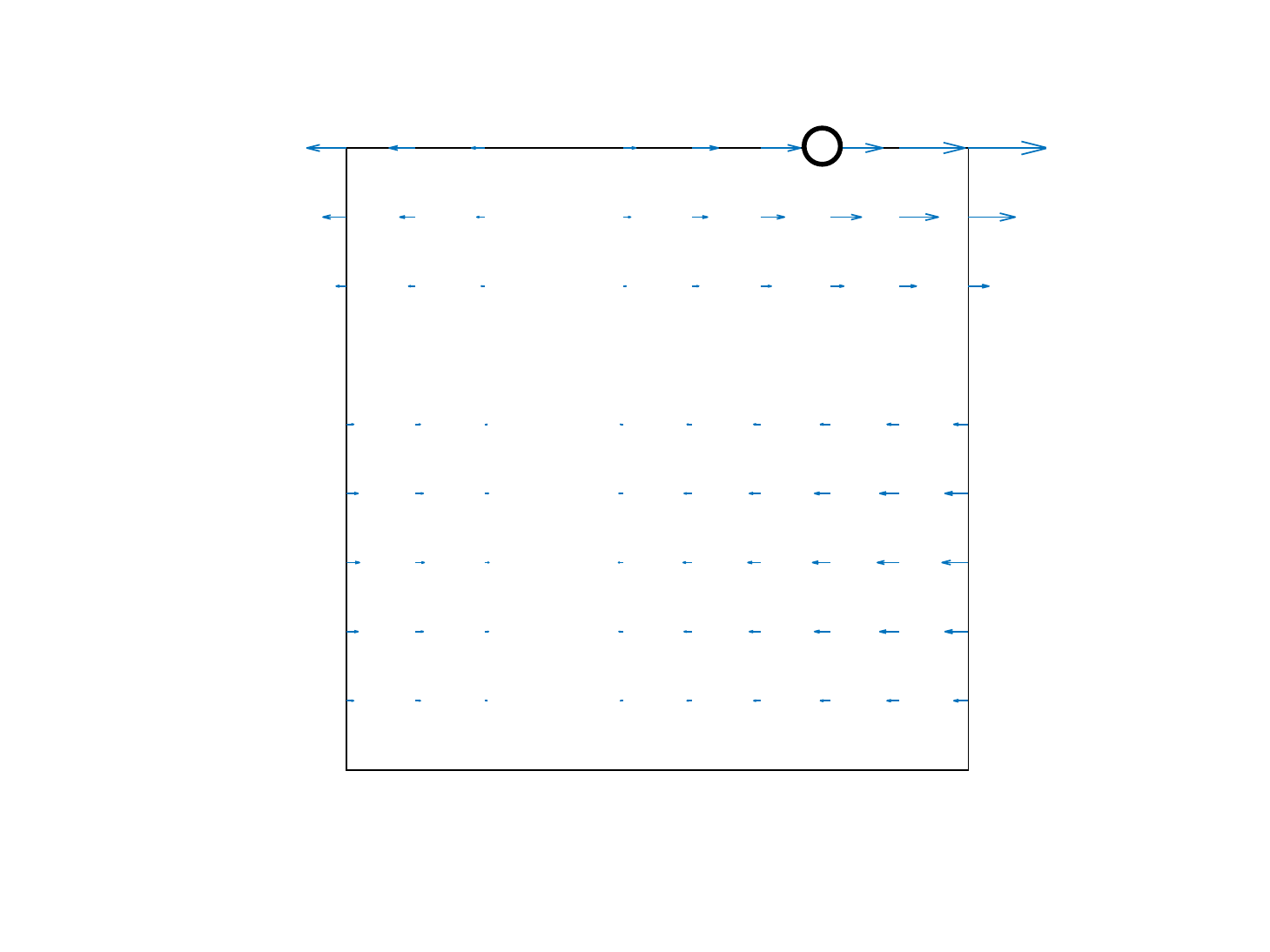}
        \caption{$\phi^1_{e_2}$}
        \label{fig-nd2_hex_e4}
    \end{subfigure}

         \begin{subfigure}[t]{0.24\textwidth}
    \includegraphics[trim={0.80cm 0.80cm 0.80cm 0.80cm},clip,width=\textwidth]{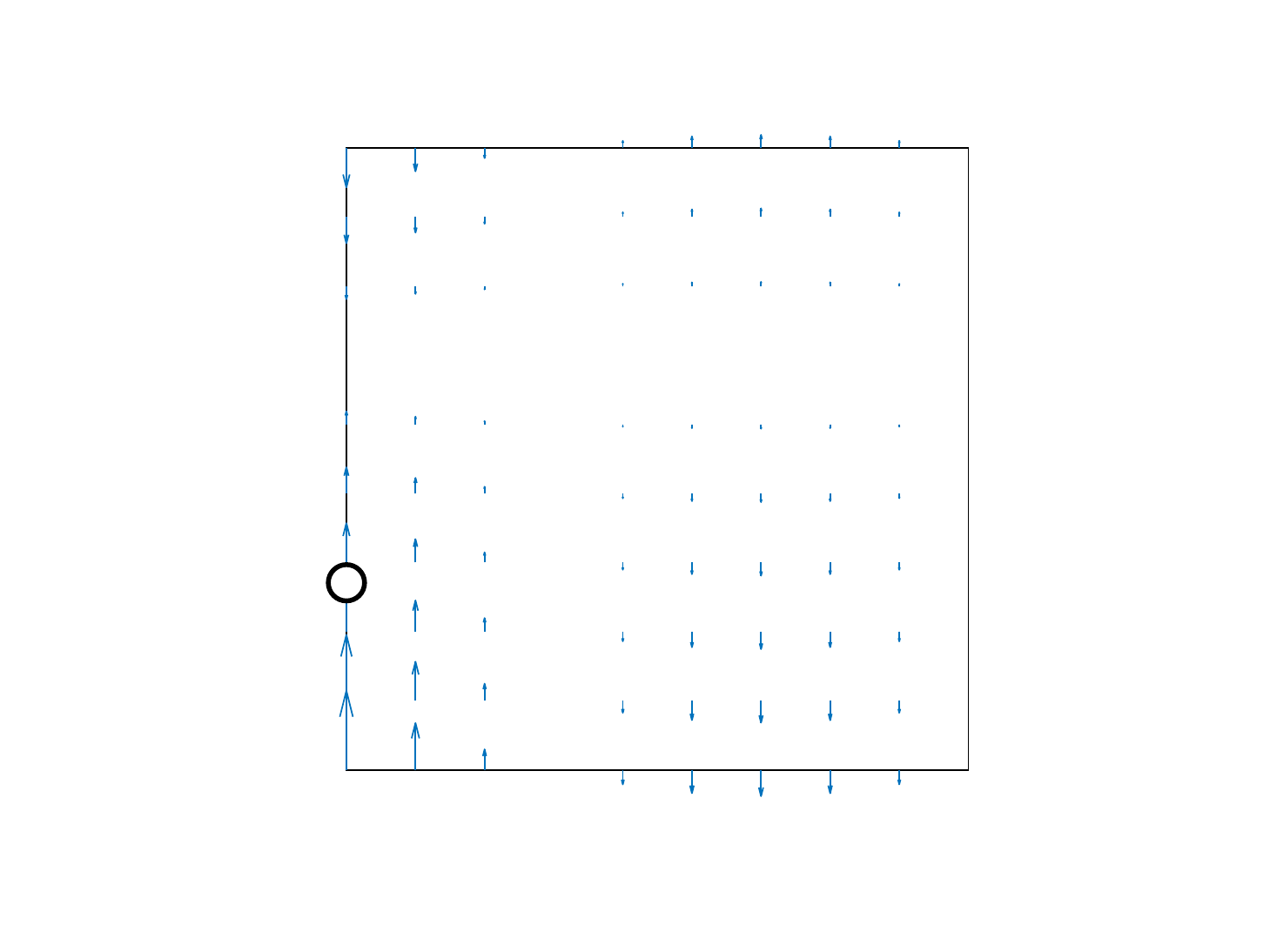}
        \caption{$\phi^1_{e_3}$ }
        \label{fig-nd2_hex_e5}
    \end{subfigure}
    \begin{subfigure}[t]{0.24\textwidth}
       \includegraphics[trim={0.80cm 0.80cm 0.80cm 0.80cm},clip,width=\textwidth]{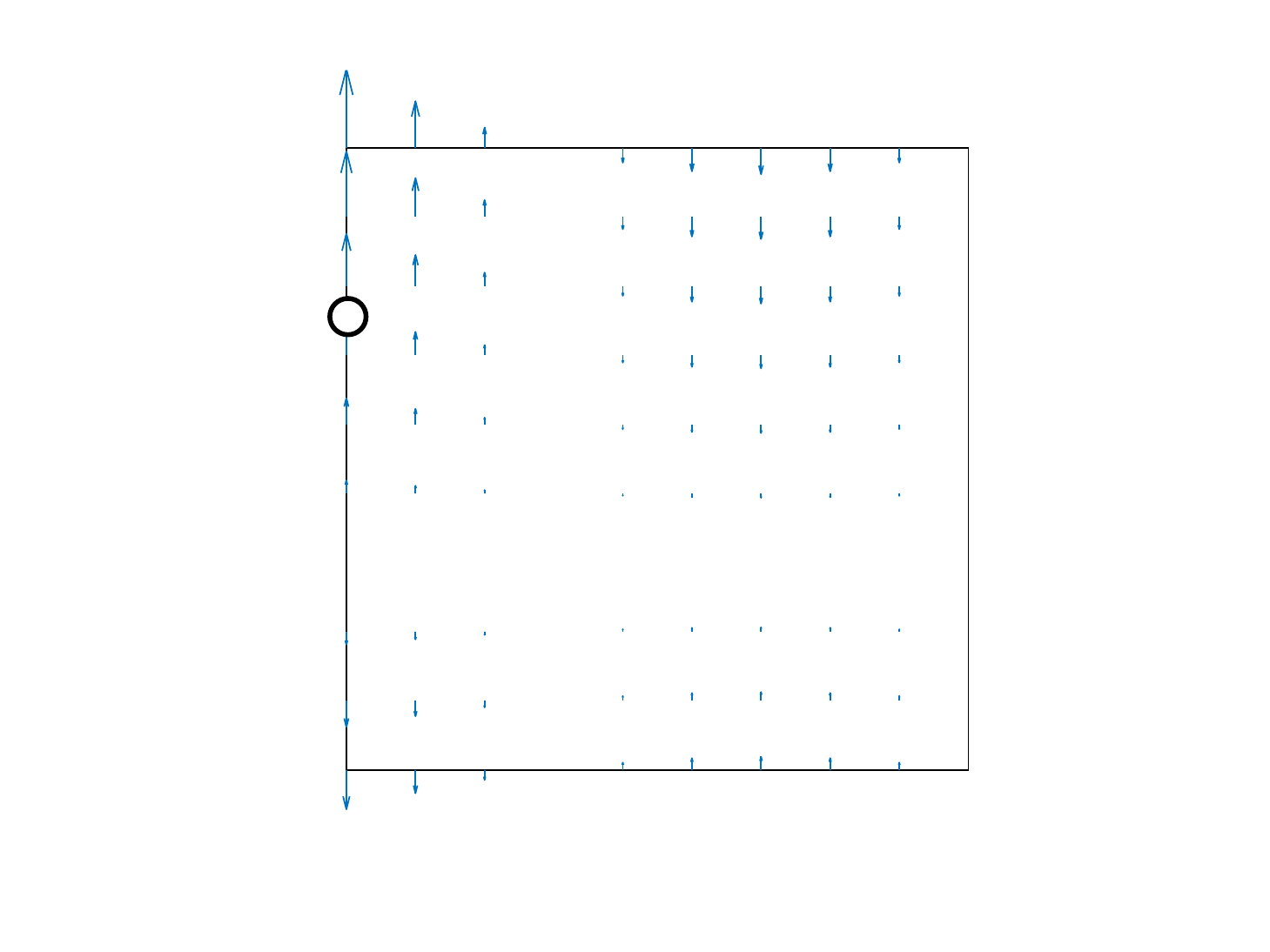}
        \caption{$\phi^2_{e_3}$ }
        \label{fig-nd2_hex_e6}
    \end{subfigure}
     \begin{subfigure}[t]{0.24\textwidth}
    \includegraphics[trim={0.80cm 0.80cm 0.80cm 0.80cm},clip,width=\textwidth]{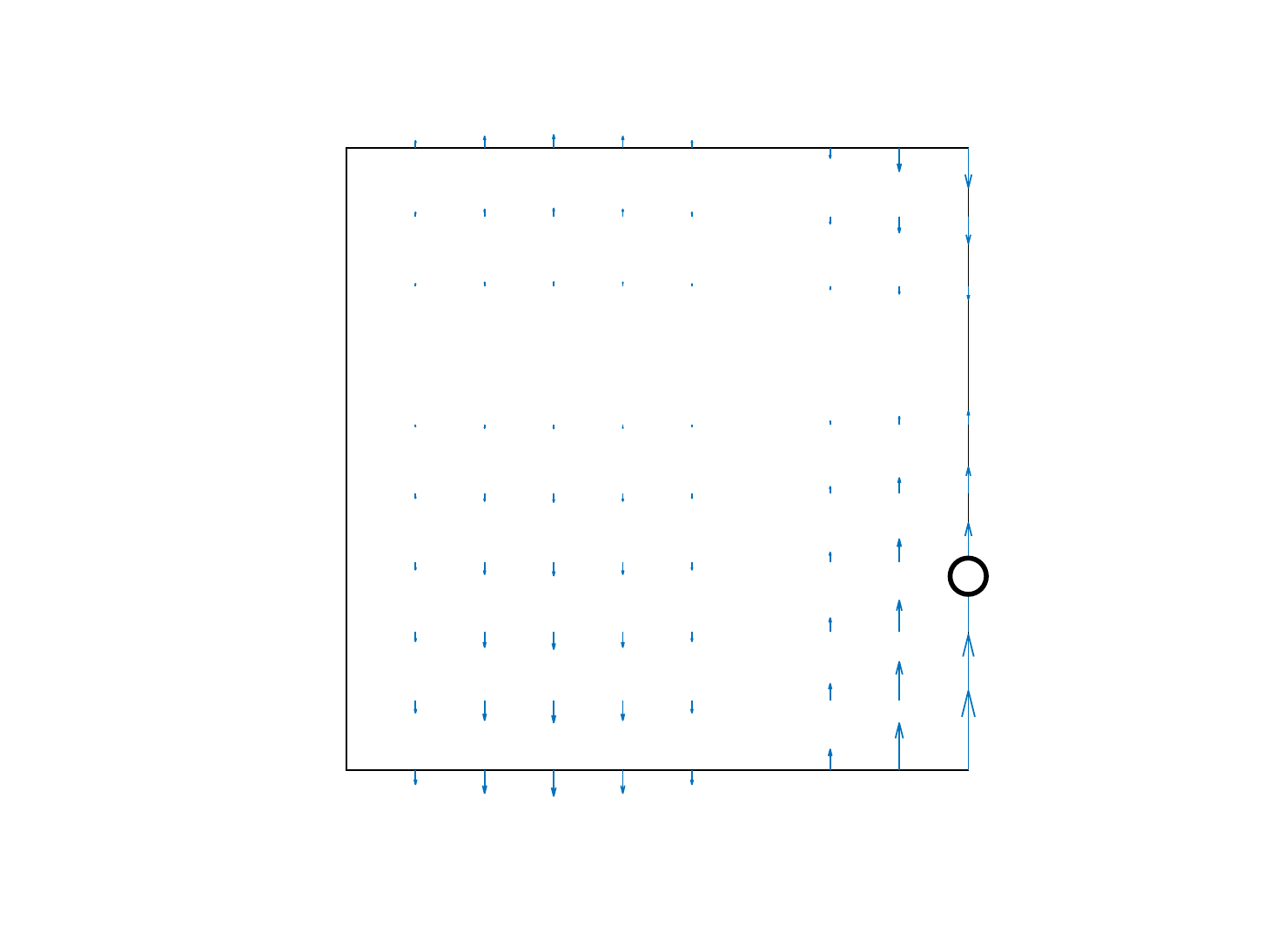}
        \caption{$\phi^1_{e_4}$ }
        \label{fig-nd2_hex_e7}
    \end{subfigure}
    \begin{subfigure}[t]{0.24\textwidth}
       \includegraphics[trim={0.80cm 0.80cm 0.80cm 0.80cm},clip,width=\textwidth]{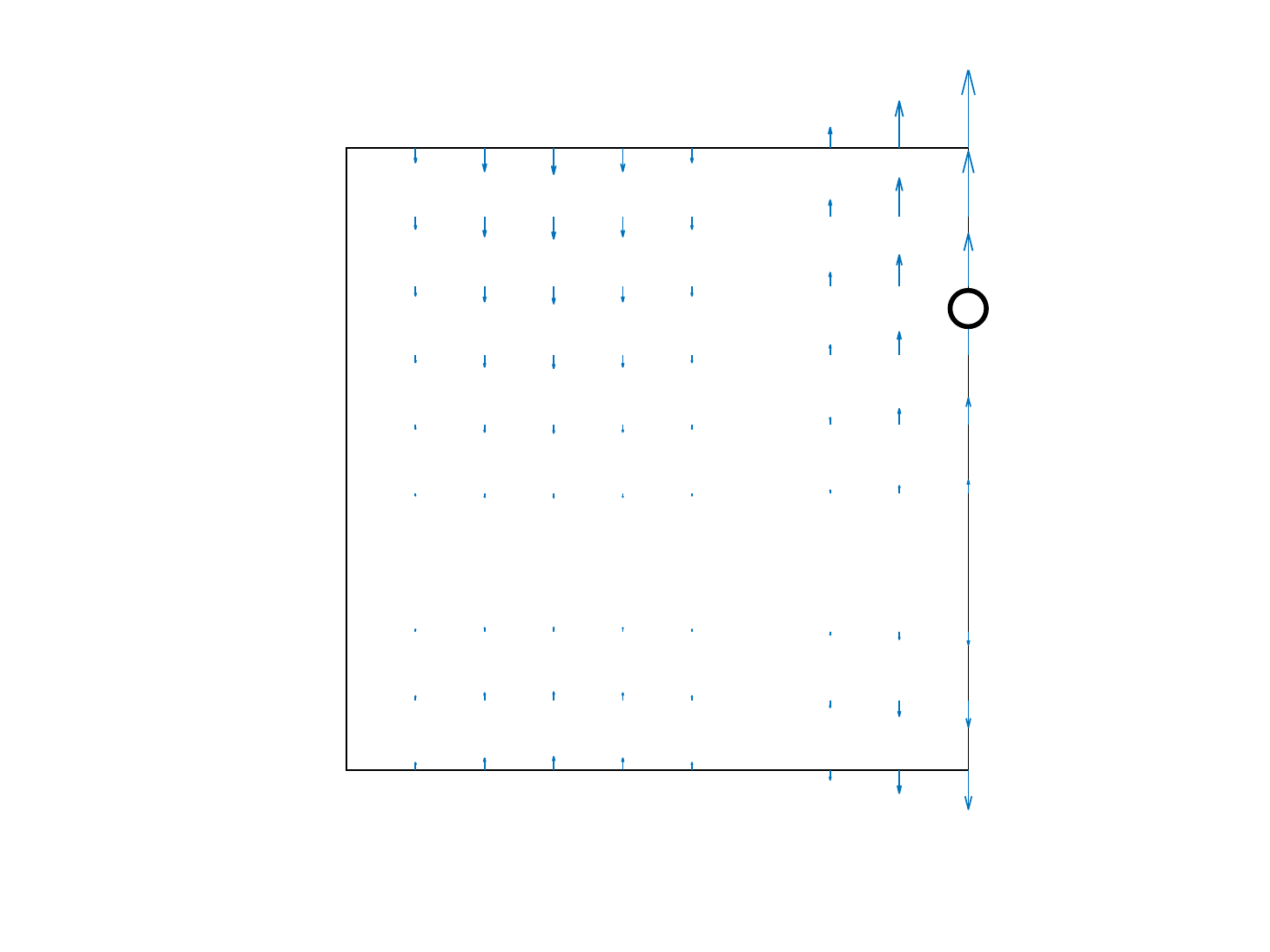}
        \caption{$\phi^2_{e_4}$ }
        \label{fig-nd2_hex_e8}
    \end{subfigure}

        \begin{subfigure}[t]{0.24\textwidth}
        \includegraphics[trim={0.80cm 0.80cm 0.80cm 0.80cm},clip,width=\textwidth]{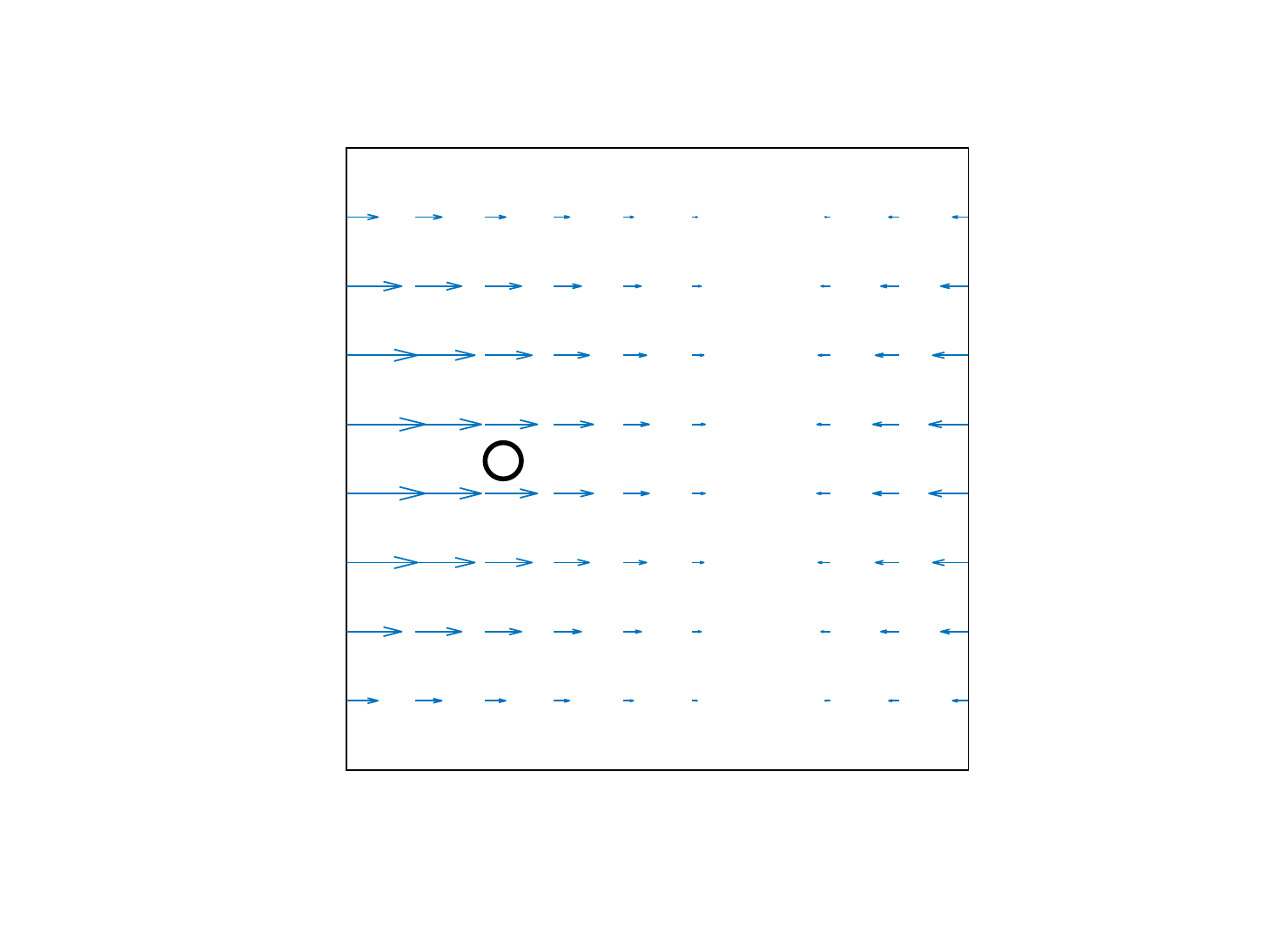}
        \caption{$\phi^1_{K}$}
        \label{fig-nd2_hex_f1}
    \end{subfigure}
            \begin{subfigure}[t]{0.24\textwidth}
        \includegraphics[trim={0.80cm 0.80cm 0.80cm 0.80cm},clip,width=\textwidth]{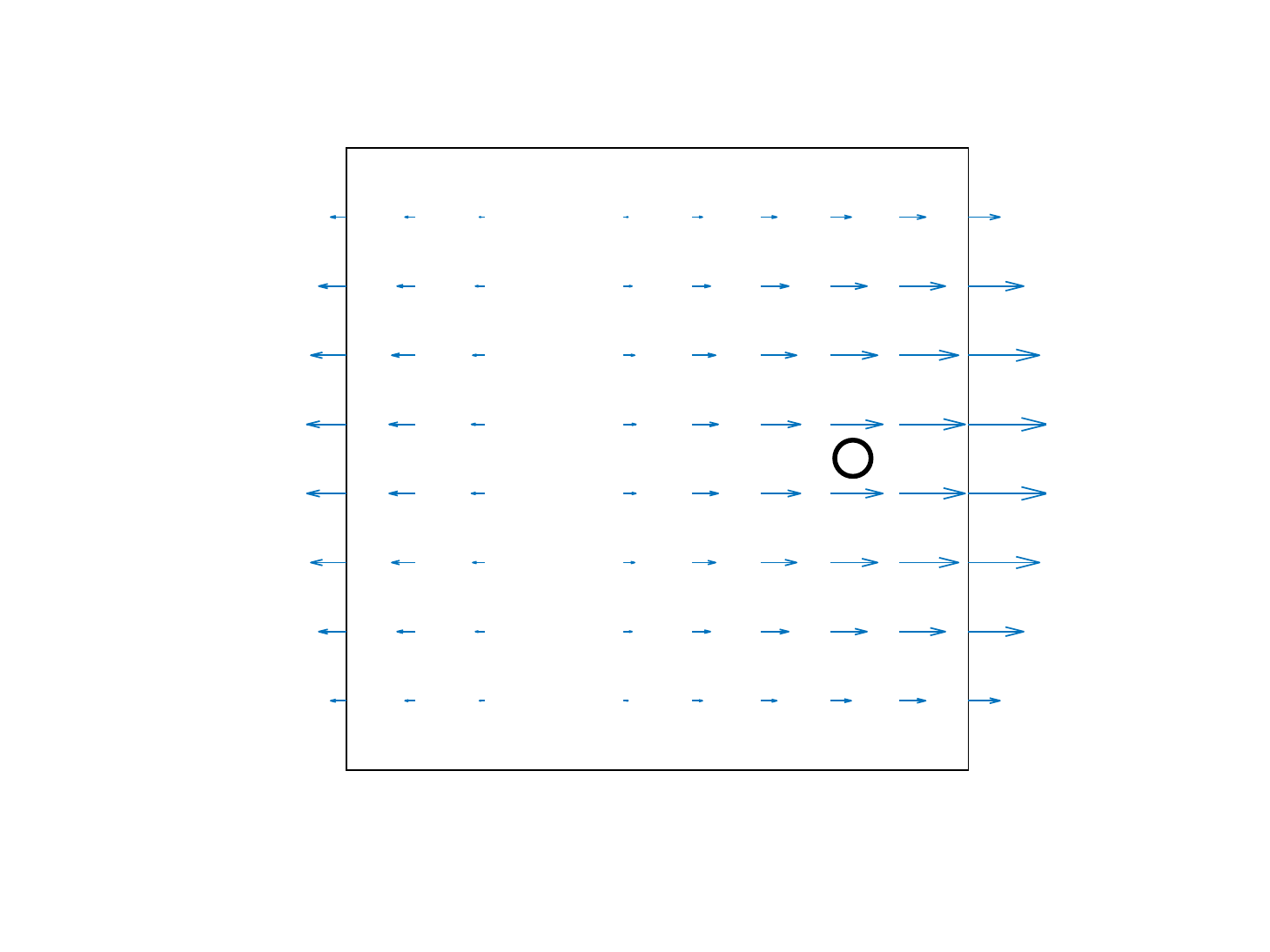}
        \caption{$\phi^2_{K}$}
        \label{fig-nd2_hex_f2}
    \end{subfigure}
     \begin{subfigure}[t]{0.24\textwidth}
        \includegraphics[trim={0.80cm 0.80cm 0.80cm 0.80cm},clip,width=\textwidth]{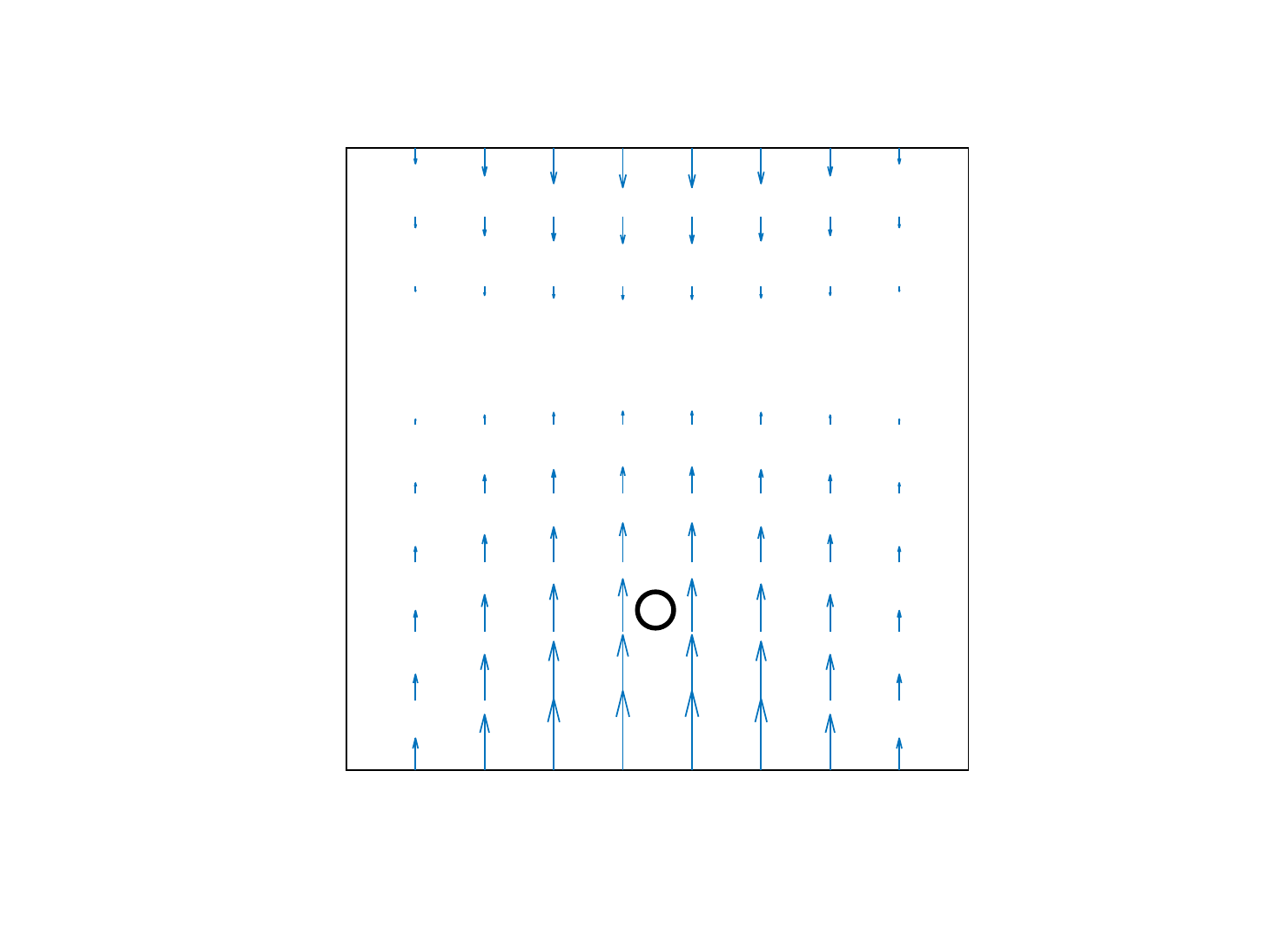}
        \caption{$\phi^3_{K}$}
        \label{fig-nd2_hex_f3}
    \end{subfigure}
     \begin{subfigure}[t]{0.24\textwidth}
        \includegraphics[trim={0.80cm 0.80cm 0.80cm 0.80cm},clip,width=\textwidth]{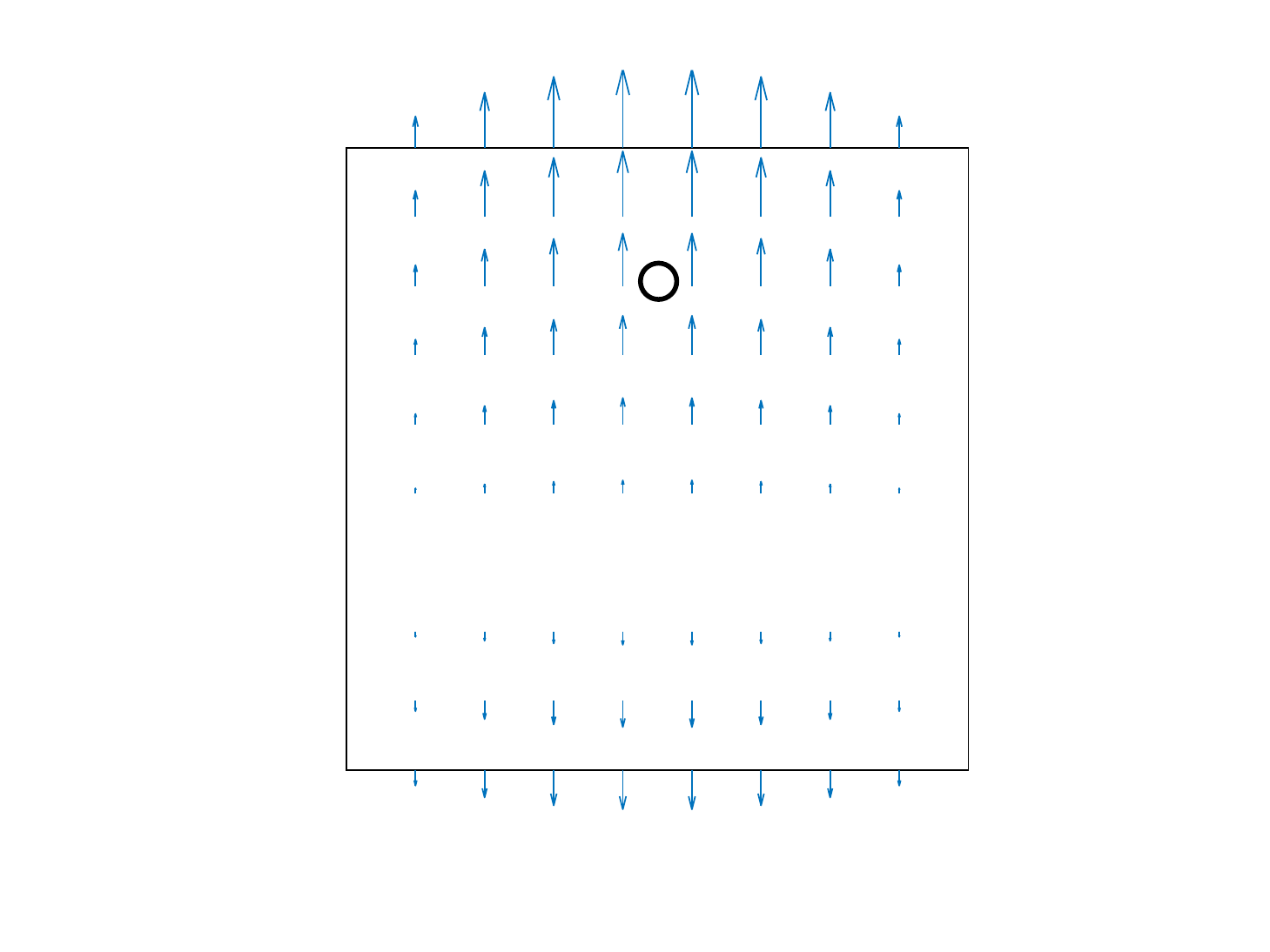}
        \caption{$\phi^4_{K}$}
        \label{fig-nd2_hex_f4}
    \end{subfigure}
    \caption{Vector-field plots of the shape functions in the second order hexahedral edge element. The indices of edges follow  the geometrical entities indices for the {reference} hexahedron in \fig{fig-3D_hex_vefs} restricted to the plane $z=0$. Auxiliary circles denote the geometrical entity where the moment is defined.}\label{fig-NED_hex_s2}
\end{figure}

\begin{figure}[t!]
    \centering
    \begin{subfigure}[t]{0.32\textwidth}
        \includegraphics[trim={1cm 1cm 1cm 1cm},clip,width=\textwidth]{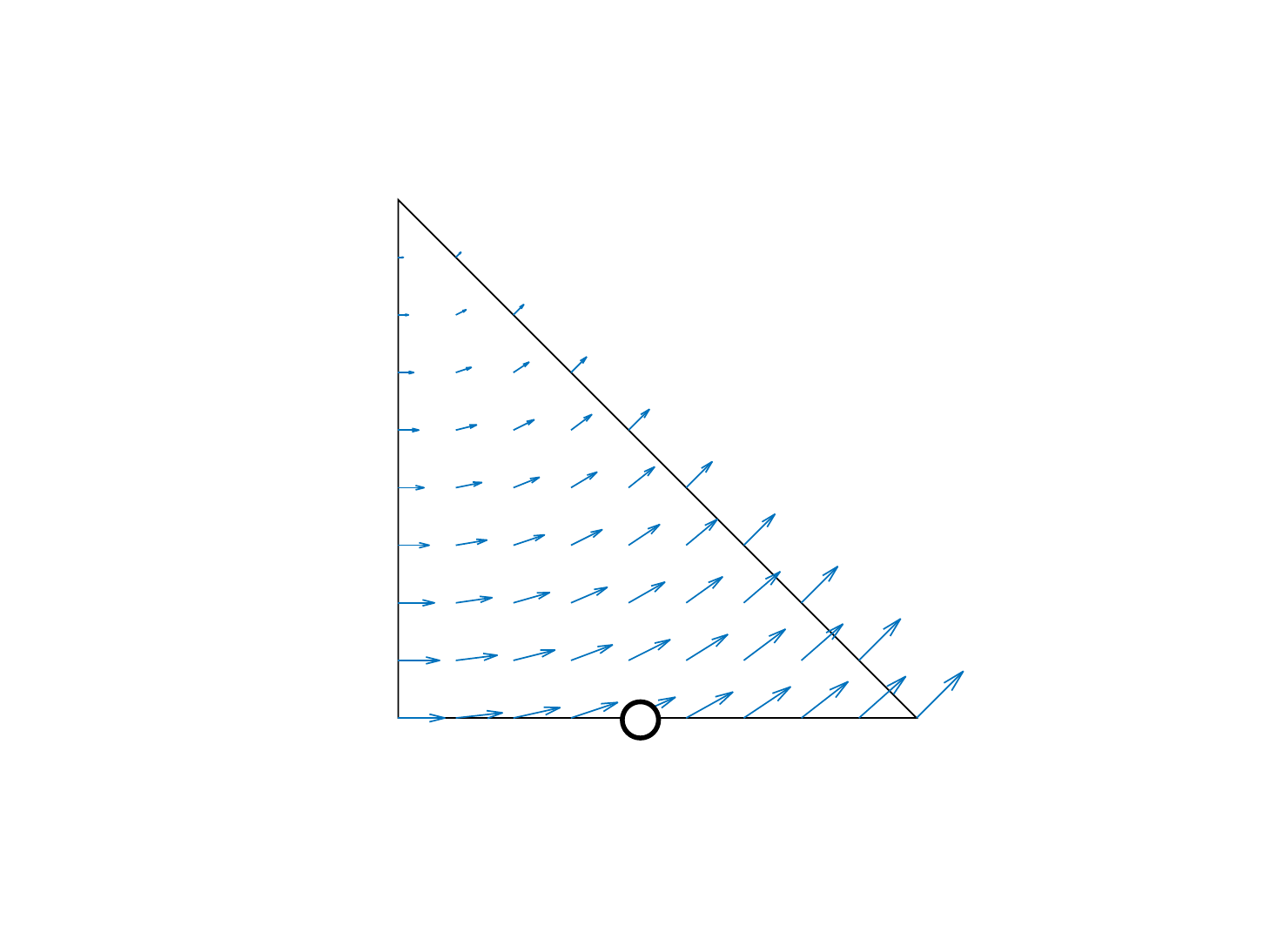}
        \caption{$\phi_{e_1}$ }
        \label{fig-nd1_tet_e1}
    \end{subfigure}
    \begin{subfigure}[t]{0.32\textwidth}
        \includegraphics[trim={1cm 1cm 1cm 1cm},clip,width=\textwidth]{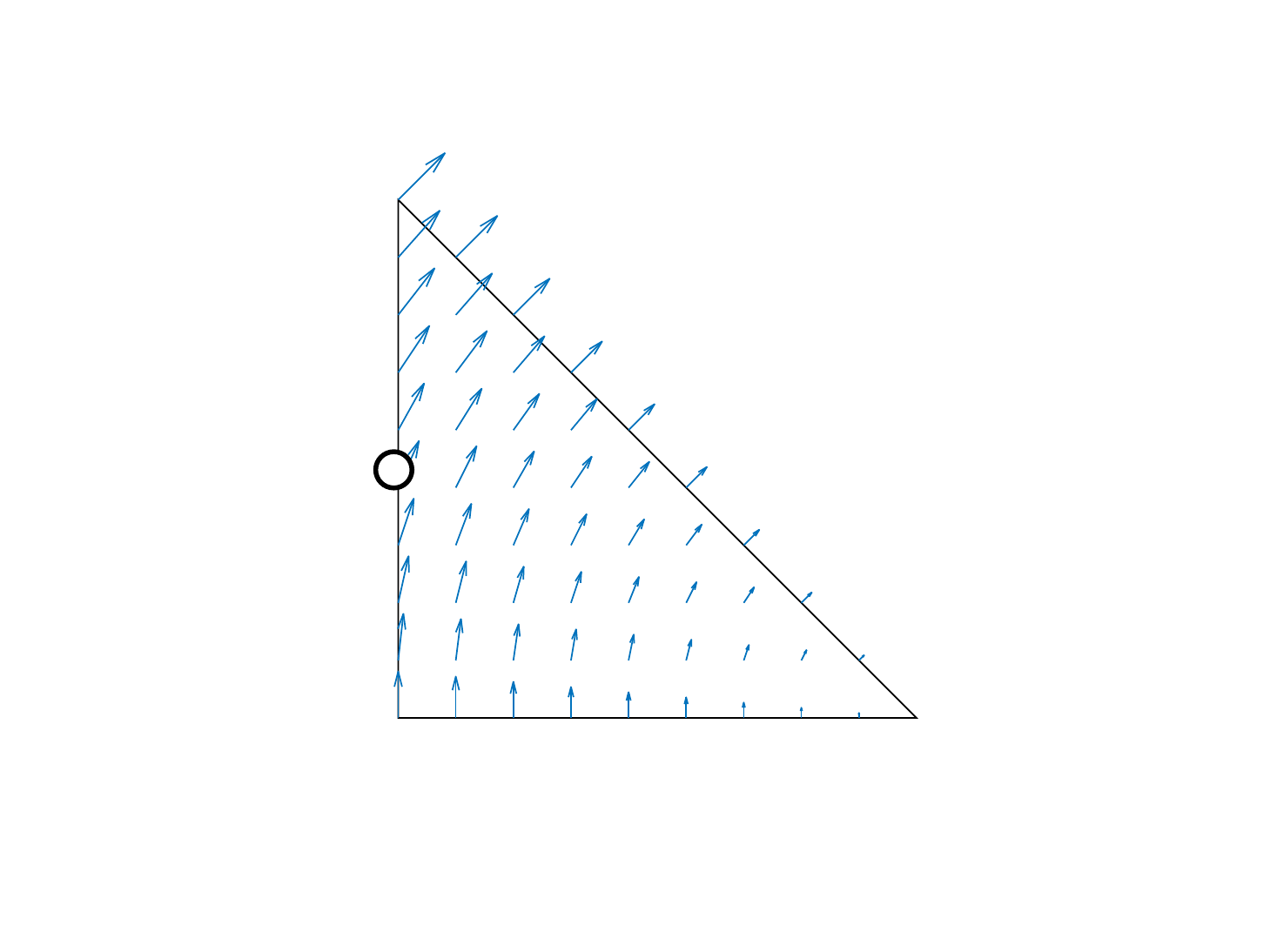}
        \caption{$\phi_{e_2}$ }
        \label{fig-nd1_tet_e2}
    \end{subfigure}
        \begin{subfigure}[t]{0.32\textwidth}
        \includegraphics[trim={1cm 1cm 1cm 1cm},clip,width=\textwidth]{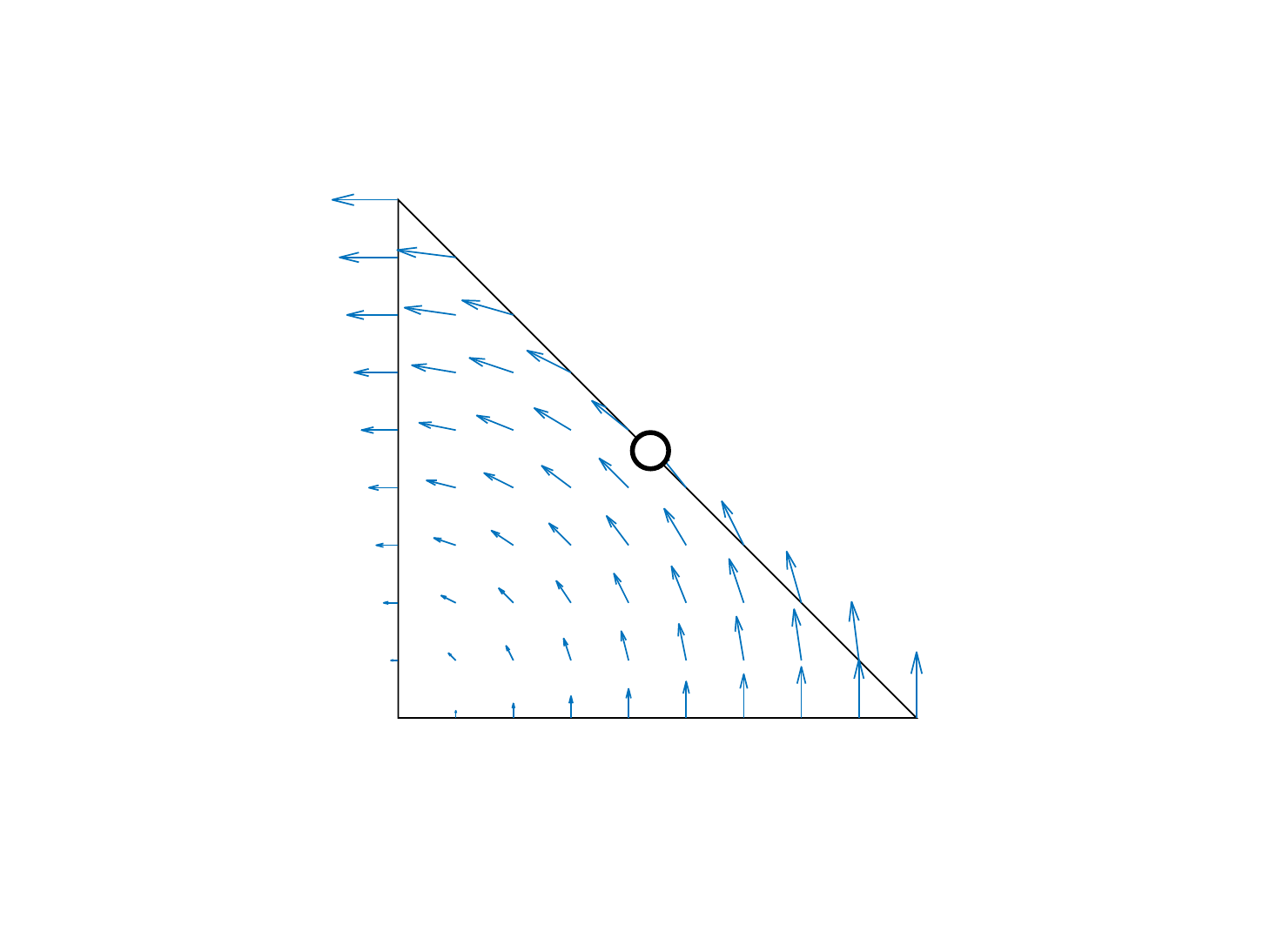}
        \caption{$\phi_{e_3}$}
        \label{fig-nd1_tet_e3}
    \end{subfigure}
    \caption{Vector-field plots of the shape functions in the lowest order tetrahedral element. The indices of edges follow  the geometrical entities indices for the {reference} tetrahedron in \fig{fig-3D_tet_vefs} restricted to the plane $z=0$. Auxiliary circles denote the geometrical entity where the moment is defined.}\label{fig-NED_s1}
\end{figure}

\begin{figure}[t!]
    \centering
    \begin{subfigure}[t]{0.24\textwidth}
   \includegraphics[trim={\crop_s \crop_s \crop_s \crop_s},clip,width=\textwidth]{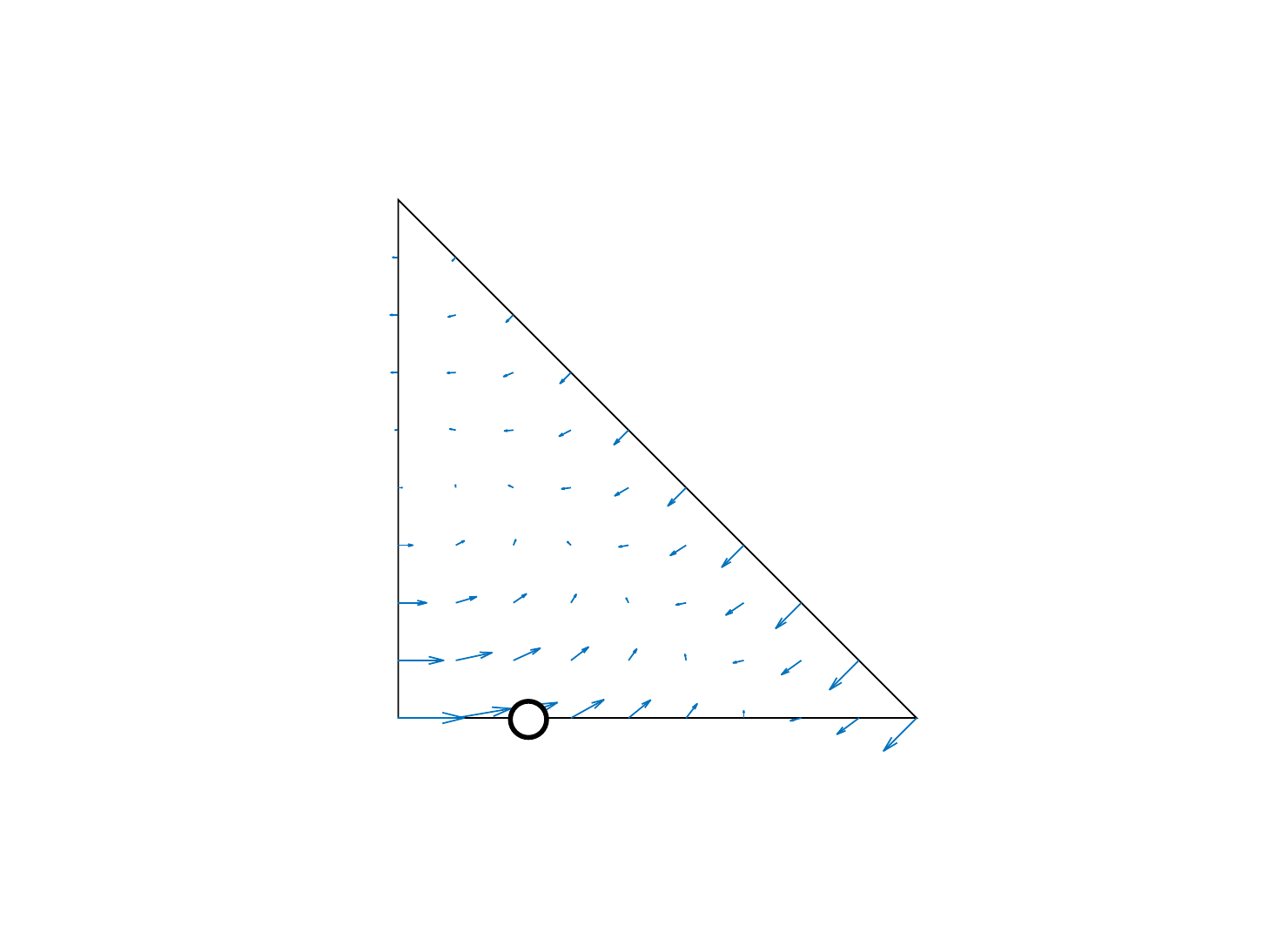}
        \caption{$\phi^1_{e_1}$ }
        \label{fig-nd2_tet_e1}
    \end{subfigure}
    \begin{subfigure}[t]{0.24\textwidth}
   \includegraphics[trim={\crop_s \crop_s \crop_s \crop_s},clip,width=\textwidth]{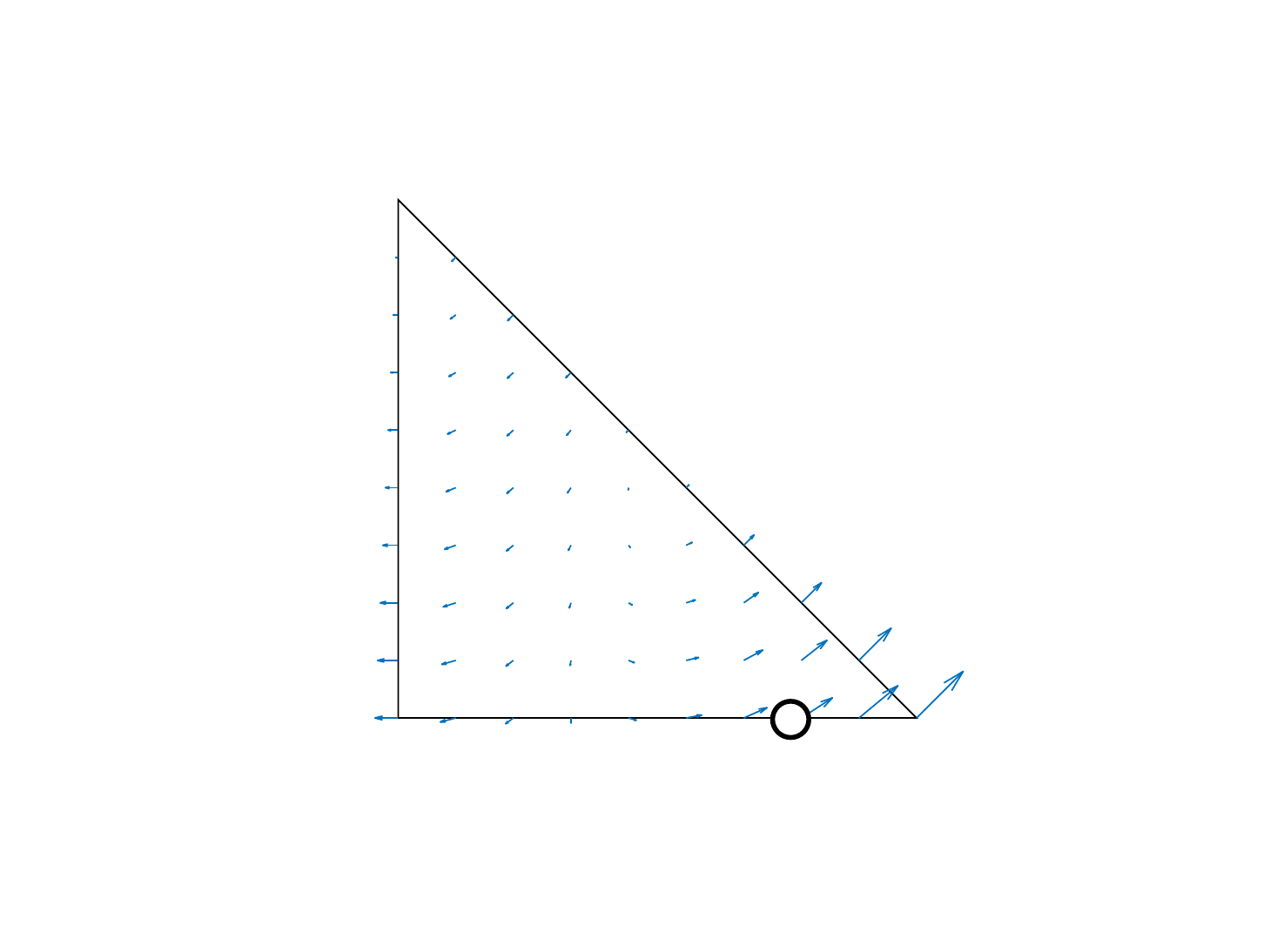}
        \caption{$\phi^2_{e_1}$ }
        \label{fig-nd2_tet_e2}
    \end{subfigure}
        \begin{subfigure}[t]{0.24\textwidth}
    \includegraphics[trim={\crop_s \crop_s \crop_s \crop_s},clip,width=\textwidth]{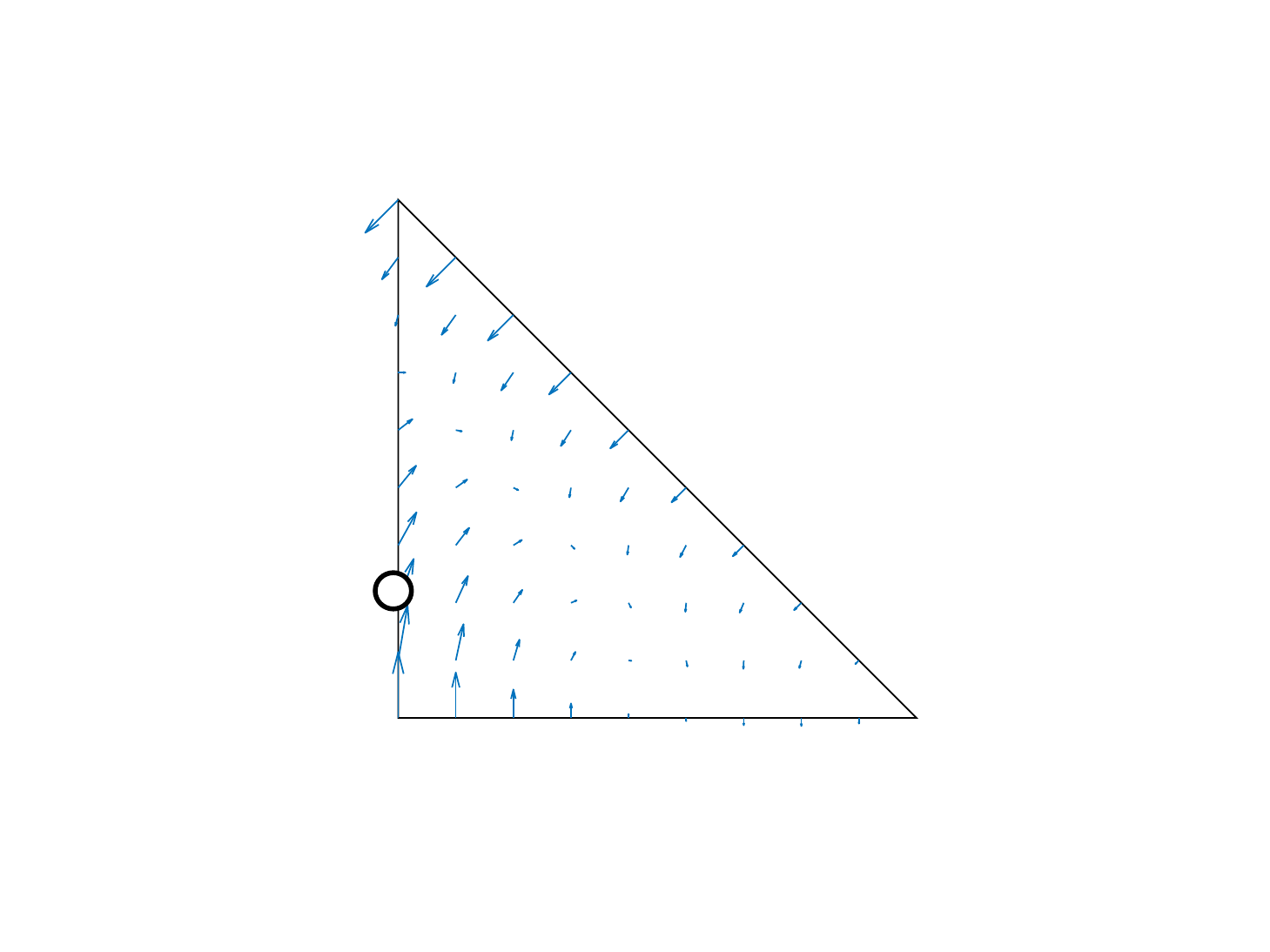}
        \caption{$\phi^1_{e_2}$}
        \label{fig-nd2_tet_e3}
    \end{subfigure}
            \begin{subfigure}[t]{0.24\textwidth}
      \includegraphics[trim={\crop_s \crop_s \crop_s \crop_s},clip,width=\textwidth]{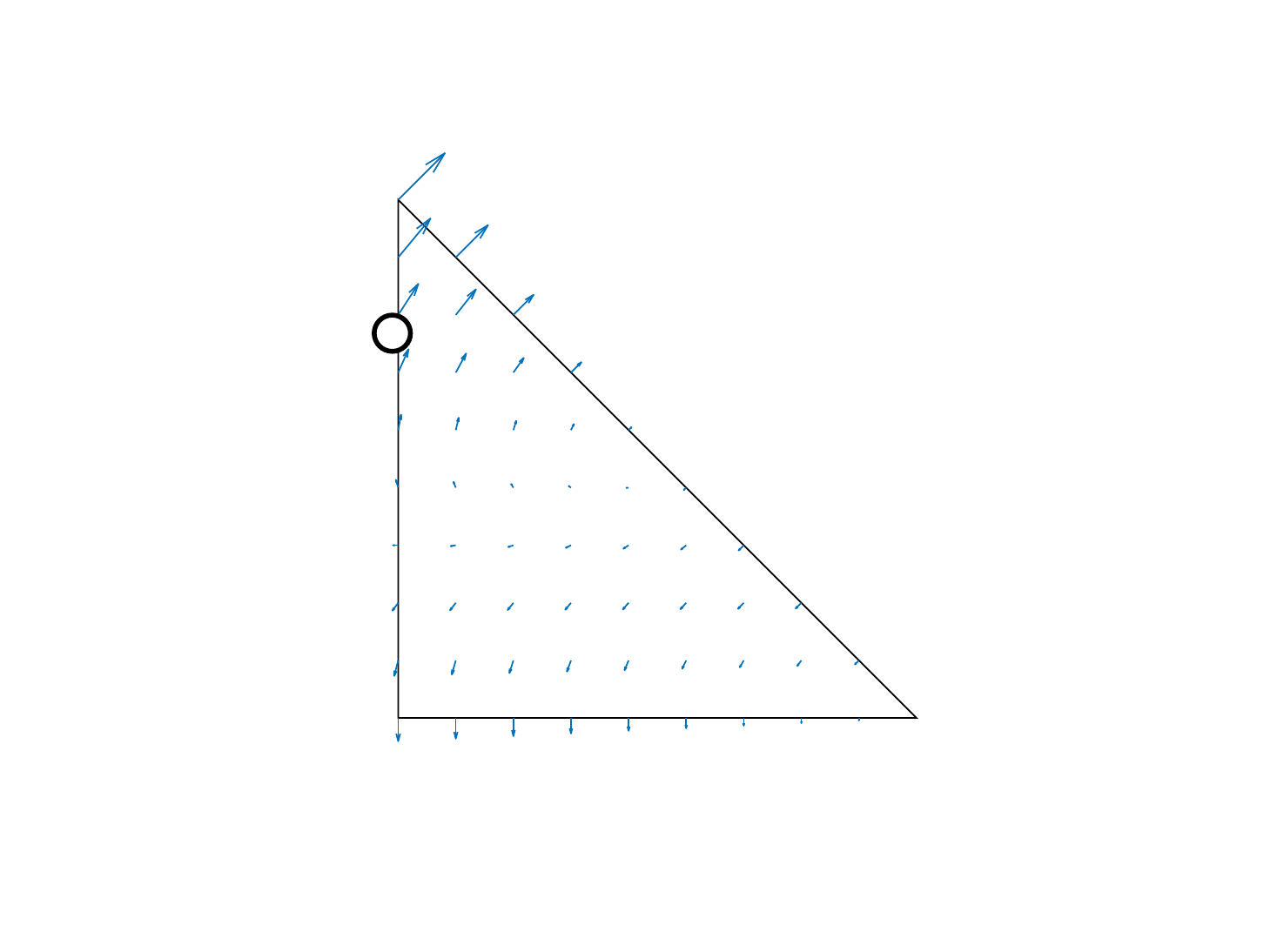}
        \caption{$\phi^2_{e_2}$}
        \label{fig-nd2_tet_e4}
    \end{subfigure}

         \begin{subfigure}[t]{0.24\textwidth}
    \includegraphics[trim={\crop_s \crop_s \crop_s \crop_s},clip,width=\textwidth]{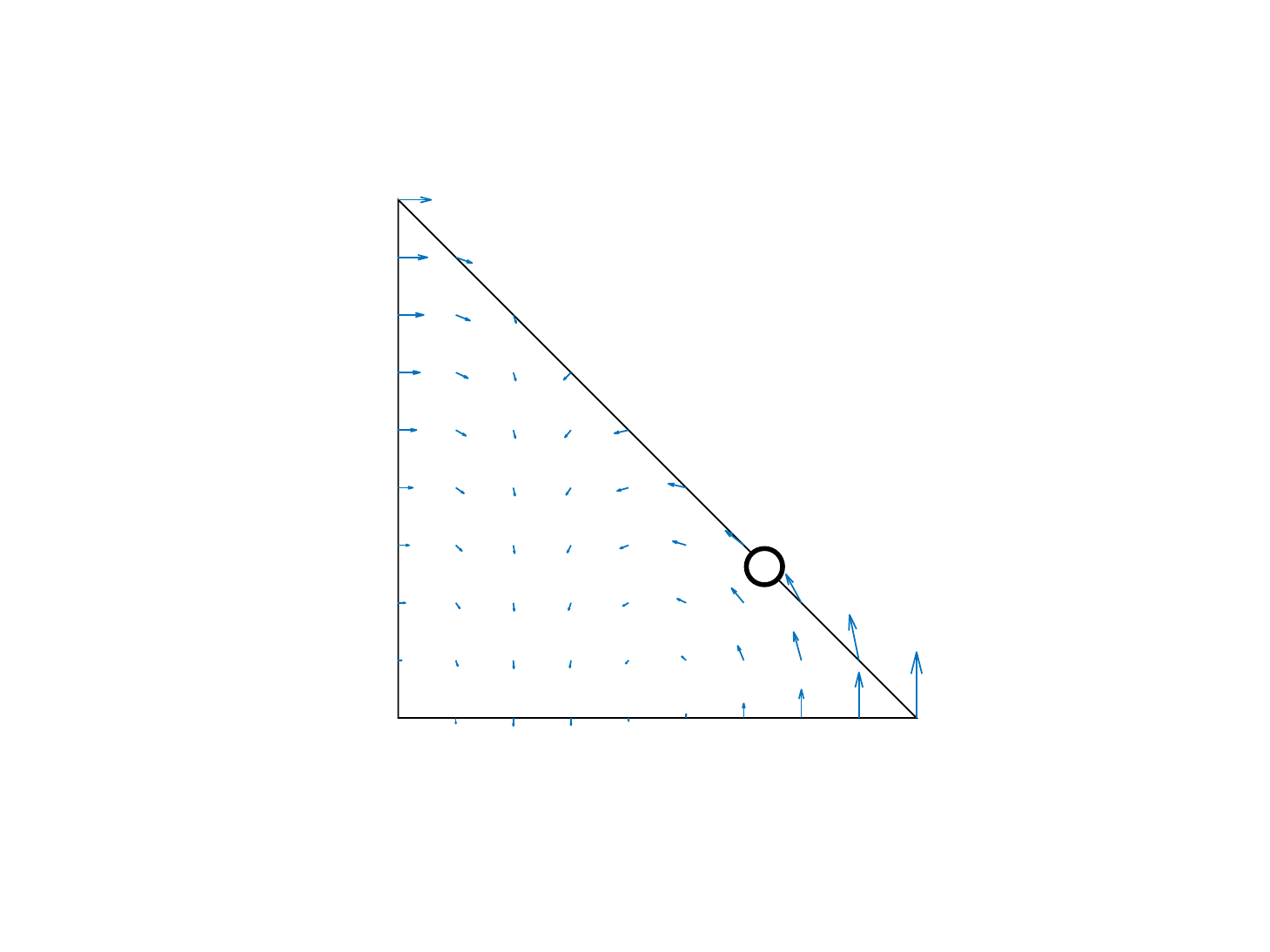}
        \caption{$\phi^1_{e_3}$ }
        \label{fig-nd1_tet_e5}
    \end{subfigure}
    \begin{subfigure}[t]{0.24\textwidth}
       \includegraphics[trim={\crop_s \crop_s \crop_s \crop_s},clip,width=\textwidth]{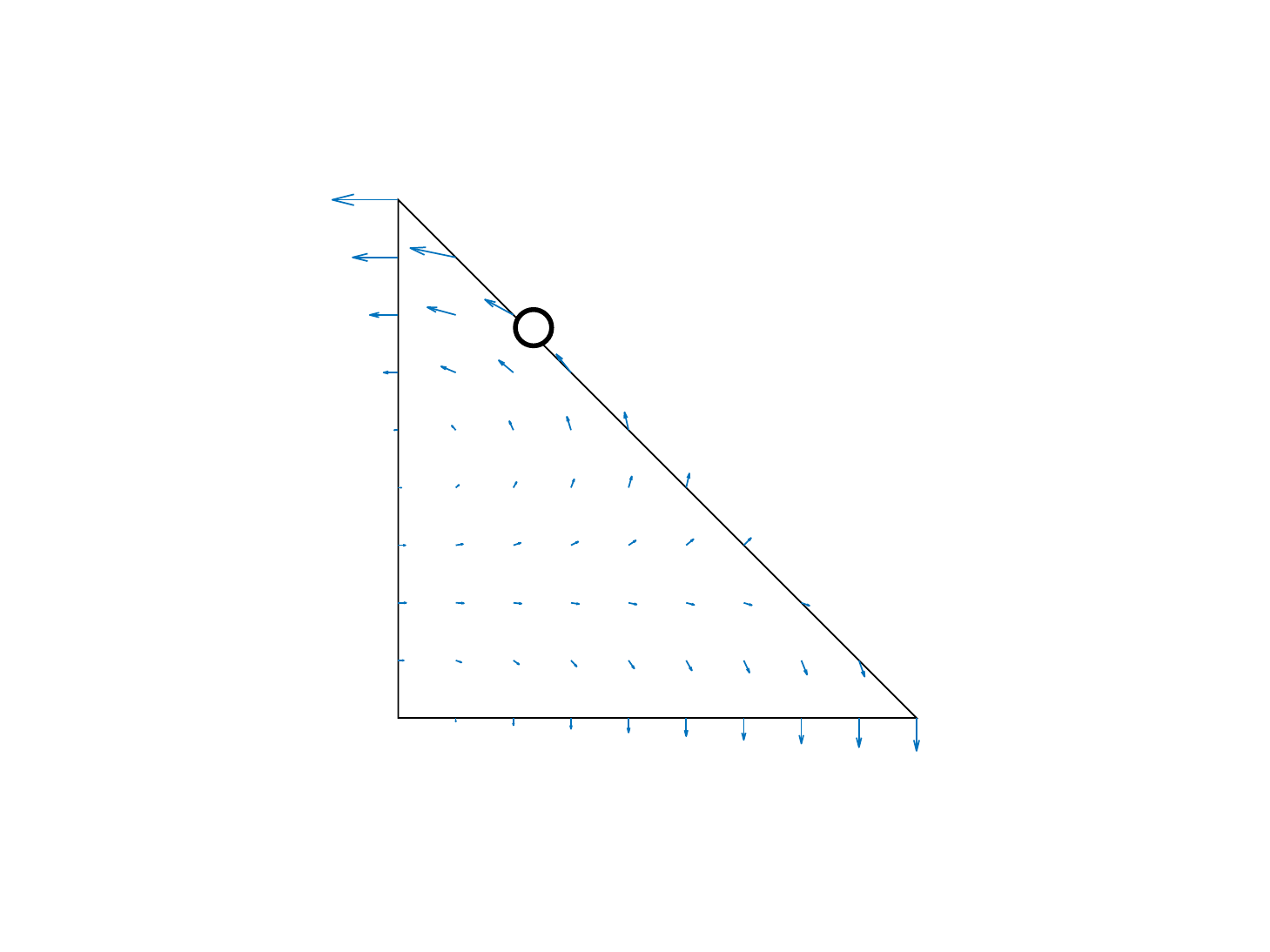}
        \caption{$\phi^2_{e_3}$ }
        \label{fig-nd1_tet_e6}
    \end{subfigure}
        \begin{subfigure}[t]{0.24\textwidth}
        \includegraphics[trim={\crop_s \crop_s \crop_s \crop_s},clip,width=\textwidth]{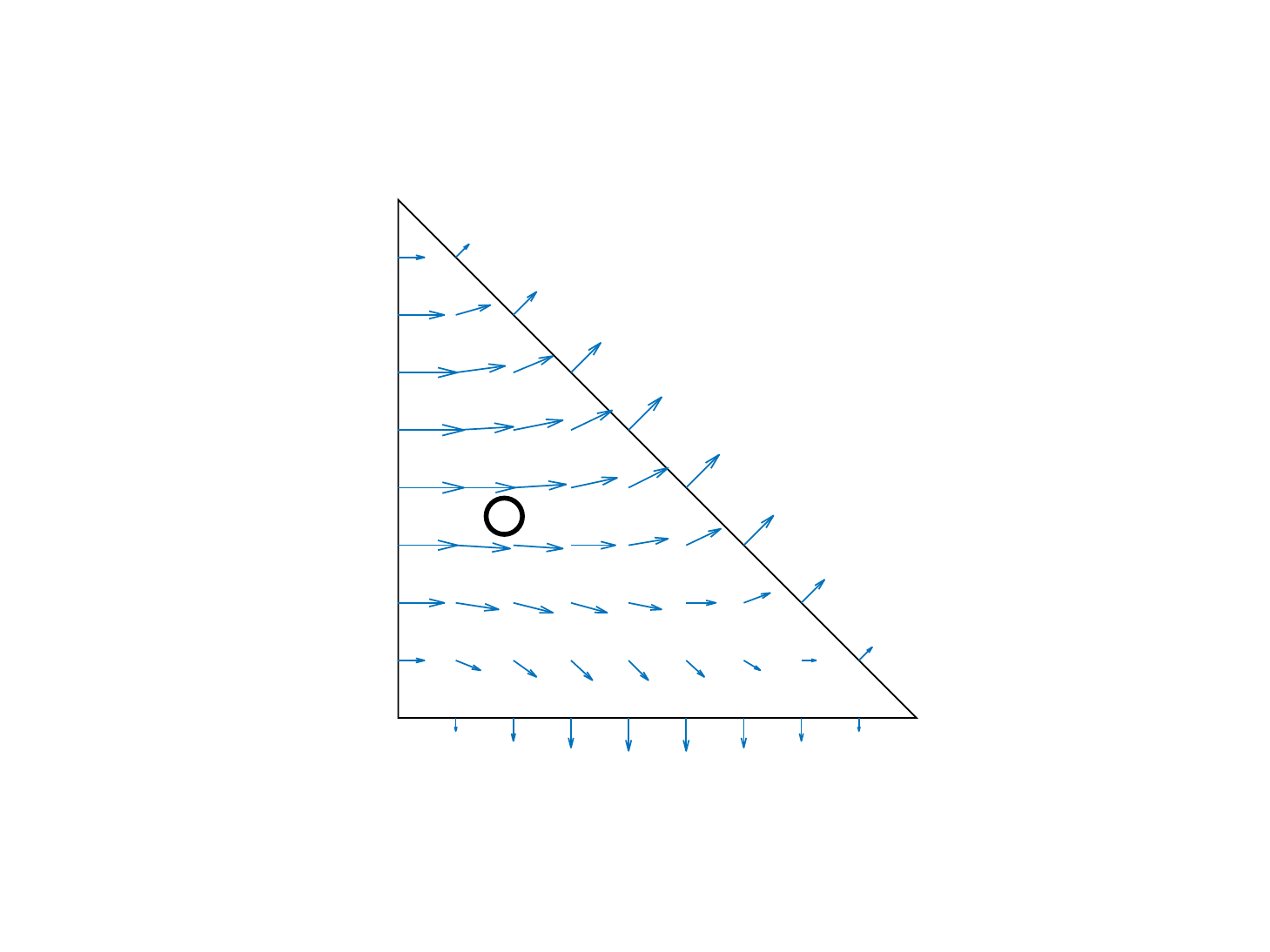}
        \caption{$\phi^1_{K}$}
        \label{fig-nd1_tet_e7}
    \end{subfigure}
            \begin{subfigure}[t]{0.24\textwidth}
        \includegraphics[trim={\crop_s \crop_s \crop_s \crop_s},clip,width=\textwidth]{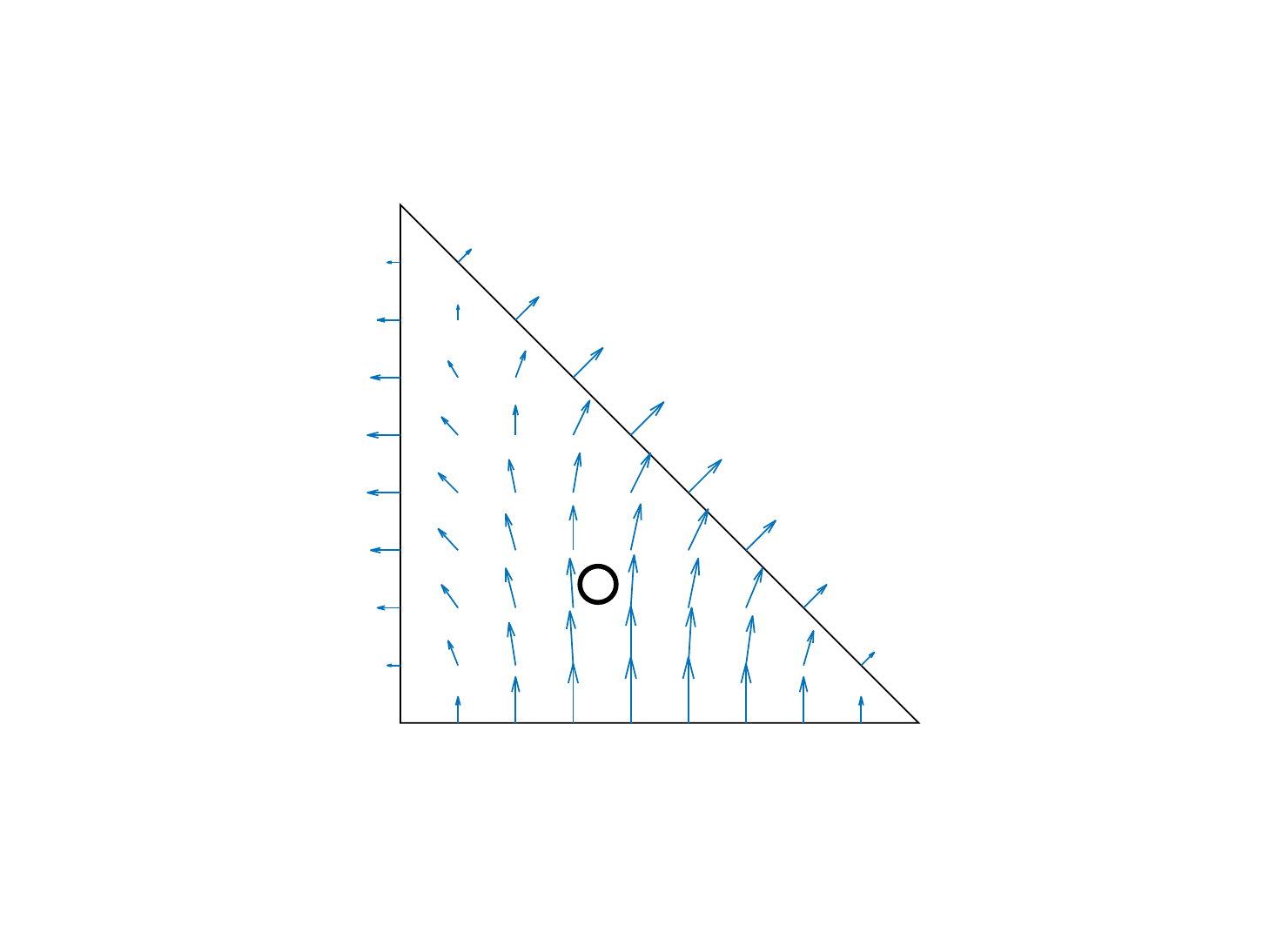}
        \caption{$\phi^2_{K}$}
        \label{fig-nd1_tet_e8}
    \end{subfigure}
    \caption{Vector-field plots of the shape functions in the second order edge element. The indices of edges follow  the geometrical entities indices for the {reference} tetrahedron in \fig{fig-3D_tet_vefs} restricted to the plane $z=0$. Auxiliary circles denote the geometrical entity where the moment is defined.}\label{fig-NED_s2}
\end{figure}

\section{Global \ac{fe} spaces and conformity}\label{sec-assembly}
A \ac{fe} space is $H$(curl)-conforming if the tangential components at the interface between elements are continuous, i.e., they do not have to satisfy normal continuity over element faces. The discrete global \ac{fe} space where the magnetic field solution $\uu_h$ lies is defined as
\begin{align}
\mathcal{ND}_k(\Omega) = \{ \vv_h \in \Hcurl \hbox{ such that } \vv_h |_K \in \esp_k(K) \, \forall \, K \in \mathcal{T}_h \},
\end{align}
where $\esp_k$ has been defined for the hexahedral and tetrahedral edge \acp{fe} in \sect{subsec-polspaces}.

\subsection{Moments in the physical space and the Piola map}\label{subsec-phymoms}
In this section we define a set of \acp{dof} (moments) in the physical space that allow one to interpolate analytical functions onto the edge \ac{fe} space, e.g., to enforce Dirichlet data. One can check that the continuity of those \acp{dof}  which lay at the boundary between cells provides the desired continuity of the tangential component. On top of that, using the so-called covariant Piola mapping ${\funmap}_\geophy(\hat{\vv}) \doteq \hat{\funmap}_\geophy (\hat{\vv}) \circ \geomap^{-1}_\geophy$, where
\begin{align}\label{eq-Piola_map}
\hat{\funmap}_K(\hat{\vv}) \doteq \jacobian^{-T} \hat{\vv},
\end{align}
it can be checked that a \ac{dof} in the reference space of a function $\hat{\vv}$ on $\georef$ is equal to the corresponding \ac{dof} in the physical space for
${\funmap}_\geophy(\hat{\vv})$ on $\geophy$ (see~\cite[Ch.\ 5]{monk_finite_2003} for details). Thus, the Piola mapping, which preserves tangential traces of vector fields~\cite{rognes_efficient_2009}, is the key to achieve a curl-conforming global space by using a \ac{fe} space relying on a reference \ac{fe} definition. 

\subsubsection{Hexahedra}
Let us now define the moments for a curl-conforming \ac{fe} defined on a general hexahedron.
Given an integer $k\geq 1$, the set of functionals that forms the basis in two dimensions reads:
\begin{subequations}
\begin{align}
\sigma_e(\uu_h)&:=  \int_{e} ( \uu_h \cdot \boldsymbol{\tau} ) q  \quad \forall q \in \mathcal{P}_{k-1}(e), \quad \forall e \in \geophy, \label{eq-rhexnedge_2D} \\
\begin{split}
\sigma_\geophy(\uu_h)&:=  \int_{\geophy}  \uu_h \cdot \boldsymbol{q}, \quad \forall \boldsymbol{q} \quad \text{obtained by mapping} \quad \boldsymbol{q} = \left(\frac{1}{{\rm det}(\boldsymbol{J}_K)}\right){\boldsymbol{J}_K}\hat{\boldsymbol{q}}, \\ & \hspace{3cm} \hat{\boldsymbol{q}} \in \Q_{k-1,k-2}({\georef}) \times  \Q_{k-2,k-1}({\georef}), \label{eq-rhexninner_2D}\end{split}
\end{align}
\end{subequations}
where $\boldsymbol{\tau}$ is the unit vector along the edge. In 3D, the set of functionals reads
\begin{subequations}
\begin{align}
\sigma_{e}(\uu_h):= & \int_{e} ( \uu_h \cdot \boldsymbol{\tau} ) q  \qquad \forall q \in \mathcal{P}_{k-1}(e), \quad \forall e \in \geophy \label{eq-rhexnedge} \\
\begin{split} \label{eq-rhexneface}
\sigma_{f}(\uu_h):= & \int_{\mathcal{F}} ( \uu_h \times \n )\cdot \boldsymbol{q} \qquad \forall \boldsymbol{q} \quad \text{obtained by mapping} \quad \boldsymbol{q}=\boldsymbol{J}_K^{-T}({\boldsymbol{J}_{\hat{f}}} \hat{\boldsymbol{q}}),\\
& \hspace{3cm} \hat{\boldsymbol{q}} \in \Q_{k-2,k-1}(\hat{\mathcal{F}})\times \Q_{k-1,k-2}(\hat{\mathcal{{F}}}), \quad \forall \mathcal{F} \in \geophy \end{split}\\
\begin{split}
\sigma_\geophy(\uu_h):= & \int_{{\geophy}}  \uu_h \cdot \boldsymbol{q} \qquad \forall \boldsymbol{q} \quad \text{obtained by mapping} \quad \boldsymbol{q} = \left(\frac{1}{{\rm det}(\boldsymbol{J}_K)}\right){\boldsymbol{J}_K} \hat{\boldsymbol{q}},  \\
& \hat{\boldsymbol{q}} \in \Q_{k-1,k-2,k-2}({\georef})\times \Q_{k-2,k-1,k-2}({\georef}) \times  \Q_{k-2,k-2,k-1}({\georef}) \label{eq-rhexninner}
\end{split}
\end{align}
\end{subequations}
where $\n$ is the unit normal to the face. The definition of the face moments in \eq{eq-rhexneface} requires to transfer either the 3D vector $(\uu_h \times \boldsymbol{n})$ to the face $\mathcal{F}$ or the vector $\boldsymbol{\hat{q}}$, contained in a reference face, to the 3D physical cell. We choose this latter option, which implies the transformation of the vector $\boldsymbol{\hat{q}}$ to the reference cell through the face Jacobian $\boldsymbol{J}_{\hat{f}}= \frac{\boldsymbol{\partial \Phi_{\hat{K}}}}{\partial \boldsymbol{x}_{\hat{f}}}$.

\subsubsection{Tetrahedra}

On a general tetrahedron $K$, we define the moments for a curl-conforming \ac{fe} as follows. Given an integer $k\geq 1$, the set of functionals that forms the basis in two dimensions reads

\begin{subequations}
\begin{align}
\sigma_{e}(\uu_h):= & \int_{e} ( \uu_h \cdot \boldsymbol{\tau} ) q  \quad \forall q \in \mathcal{P}_{k-1}(e), \quad \forall e \in \geophy \label{eq-rtetnedge_2D} \\
\begin{split}
\sigma_\geophy(\uu_h):= &  \int_{\geophy} \uu_h \cdot \boldsymbol{q} \quad \forall \boldsymbol{{q}} \quad \text{obtained by mapping}  \\
& \hspace{3cm}  \boldsymbol{q} = \left(\frac{1}{{\rm det}(\boldsymbol{J}_K)}\right){\boldsymbol{J}_K} \hat{\boldsymbol{q}}, \quad \hat{\boldsymbol{q}}\in[\mathcal{P}_{k-2}(\georef)]^2. \label{eq-rtetninner_2D}
\end{split}
\end{align}
\end{subequations}
Clearly, we have $3k$ edge \acp{dof} and $k(k-1)$ inner \acp{dof}, which lead to a total number of $n_{{\Sigma}}=k(k+2)$ \acp{dof}. In three dimensions the set ${\Sigma}$ is defined as
\begin{subequations}\label{eq-rtetmoms}
\begin{align}
\sigma_{e}(\uu_h):= & \int_{e} ( \uu_h\cdot \boldsymbol{\tau} ) q \qquad \forall q \in \P_{k-1}(e), \quad \forall e \in \geophy \label{eq-rtetnedge} \\
\sigma_f(\uu_h):= & \frac{1}{\| \mathcal{F} \|} \int_{\mathcal{{F}}} \uu_h \cdot \boldsymbol{q} \qquad \forall \boldsymbol{q} = {\boldsymbol{J}_K} \hat{\boldsymbol{q}} \quad \text{ s.t. } \boldsymbol{\hat{q}}\cdot\hat{\n}=0, \label{eq-rtetneface} \\
& \hspace{3cm} \hat{\boldsymbol{q}}\in[\mathcal{P}_{k-2}(\hat{f})]^3 \quad \forall f \in \geophy \notag \\
\begin{split}
\sigma_\geophy(\uu_h):= & \int_{\geophy}  \uu_h \cdot \boldsymbol{q} \qquad \forall \boldsymbol{q} \quad \text{obtained by mapping} \label{eq-rtetninner} \\
& \hspace{3cm} \boldsymbol{q} = \left(\frac{1}{{\rm det}(\boldsymbol{J}_K)}\right){\boldsymbol{J}_K} \hat{\boldsymbol{q}},\quad \hat{\boldsymbol{q}}\in[\mathcal{P}_{k-3}(\georef)]^3.
\end{split}
\end{align}
\end{subequations}

\subsection{Nedelec interpolator}
The space of edge \ac{fe} functions can be represented as the range of an interpolation operator $\pi^h$ that is well defined for sufficiently smooth functions $\uu \in \Hcurl$ by
\begin{eqnarray}\label{eq:interpolant}
{\pi}^h(\uu) := \sum_a u^a(\uu) \phi^a
\end{eqnarray}
where $u^a(\uu)=\sigma^a(\uu)$ are the evaluation of the moments for the function $\uu$ described for the hexahedra and tetrahedra cases in \sect{subsec-phymoms} for all $e\in\mathcal{E}$, $F\in\mathcal{F}$ and $\geophy \in \mathcal{T}_h$. Note that Dirichlet boundary conditions can be strongly imposed in the resulting system (usual implementation in \ac{fe} codes) by evaluating the moments corresponding to edges and/or faces on the Dirichlet boundary given the analytical expression of the tangential trace.


\subsection{Global \acp{dof}} \label{sec:global_space}
In order to guarantee global inter-element continuity with Piola-mapped elements, special care has to be taken with regard to the orientation of edges and faces at the cell level.
Let us consider a global numbering for the vertices in a mesh and a local numbering at the cell level. Given an edge/face, sorting its vertices with respect to the local (resp. global) index of their vertices, one determines the local (resp. global) orientation of the edge/face. A mesh in which the local and global orientation of all its edges and faces coincide is called an \emph{oriented mesh}. For oriented meshes, common edges or faces for two adjacent tetrahedra will always agree in its orientation, thus represent the same global \ac{dof}. Tetrahedral meshes are oriented with a simple local renumbering~\cite{rognes_efficient_2009}, and we restrict ourselves to hexahedral octree meshes that are oriented by construction.

Local \acp{dof} are uniquely determined by the cell in which they are defined, the geometrical entity within the reference cell that owns them and the local numbering of \acp{dof} within the geometrical entity (see Sect. \ref{subsubsec-hex-ref-moms}). The local numbering only depends on the orientation of the geometrical entity through the ordering of its vertices.

Global \acp{dof} are defined as an equivalence class over the union of the local \acp{dof} for all cells. Two local \acp{dof} are the same global \ac{dof} if and only if they belong to the same geometrical object and the same local numbering within the geometrical object in their respective cells. The previous equivalence class leads to a curl-conforming \ac{fe} space if local orientations of geometrical objects coincide with a global orientation, i.e., on oriented meshes only.

For the sake of illustration, see the simple mesh depicted in \fig{fig-oriented_mesh}, composed by two tetrahedra defined by the global vertices $v_1, v_2, v_3, v_4$ and $v_5$. The two elements share a common face, defined by 3 vertices that have different local indices for all $v_i \in K_i$ and $v_j \in K_j$. The ordering convention, local indexing according to ascending global indices, ensures that both triangles agree on the direction of the common edges and the common face. Note that the Jacobian of the transformation may become negative with the orientation procedure. We only need to make sure that the absolute value of the Jacobian is taken whenever the measure of the change of basis is applied.
\begin{figure}[t!]
    \centering
        \includegraphics[width=0.8\textwidth]{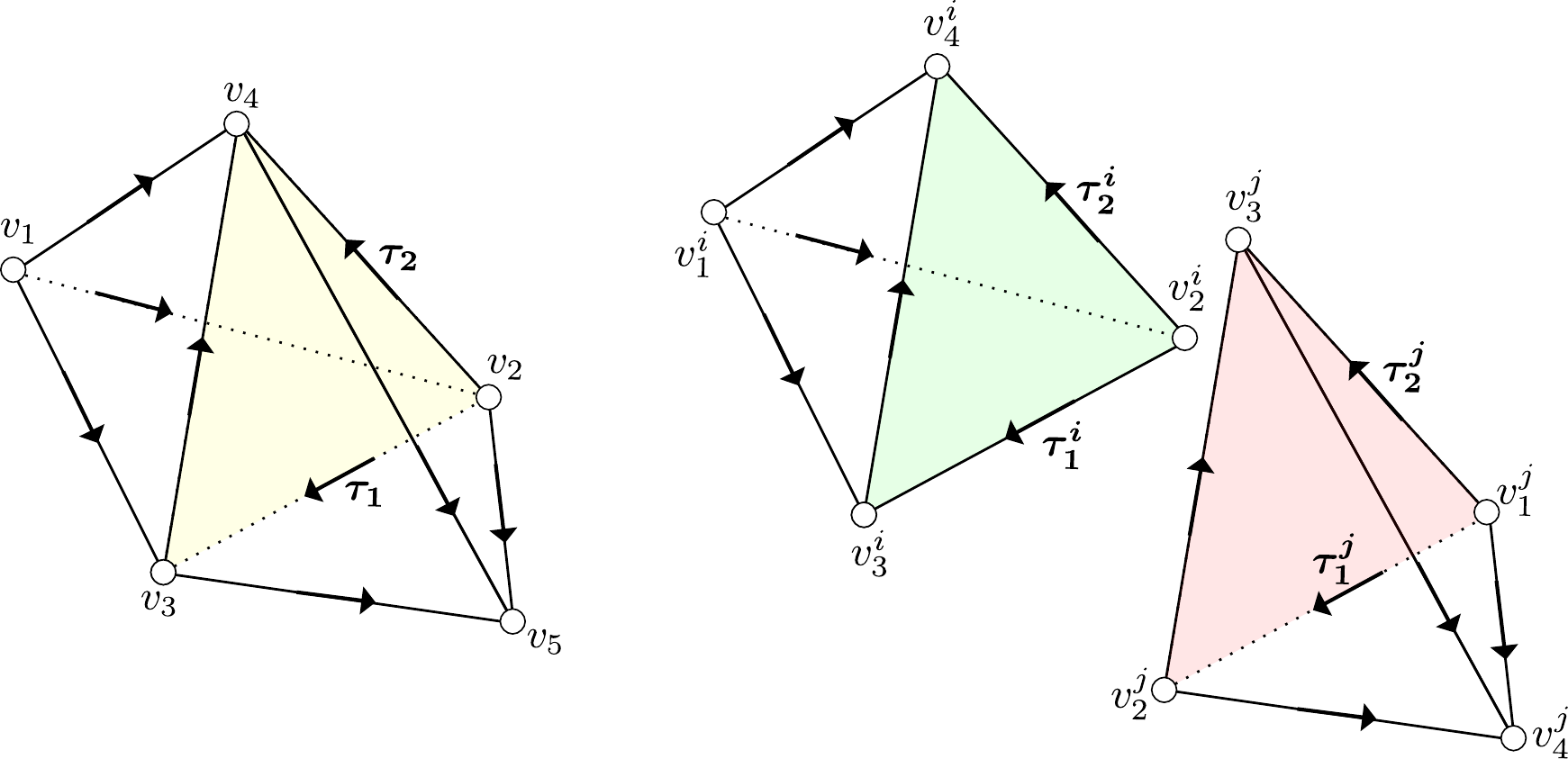}
    \caption{Oriented mesh. Global information (left) and its local, oriented, counterpart (right) for two elements $K_i$ and $K_j$. Common edges for two adjacent tetrahedra will always agree in its direction, as well as chosen tangent vectors $\boldsymbol{\tau_1}$ and $\boldsymbol{\tau_2}$ to the common face.}\label{fig-oriented_mesh}
\end{figure}
The situation is much more involved for general hexahedral meshes. In any case, our implementation of hexahedral meshes relies on octree meshes, where consistency is ensured. It is easy to check that an octree mesh in which a cell inherits the orientation of its parent is \emph{oriented} by construction.

\section{Edge \acp{fe} in $h$-adaptive meshes}\label{sec-h_adaptive}
In this section, we expose an implementation procedure for \emph{non-conforming} meshes and edge \acp{fe} of arbitrary order. In the following, we restrict ourselves to hexahedral meshes, even though the extension to tetrahedral meshes is straightforward.
 Our procedure can also be readily extended to any \ac{fe} space that relies on a \emph{pre-basis} of Lagrangian polynomial spaces plus a change of basis, e.g., Raviart-Thomas \acp{fe}.

\subsection{Hierarchical \ac{amr} on octree-based meshes}\label{subsec-AMR}

The \ac{amr} generation relies on hierarchically refined octree-based hexahedral meshes~\cite{Tu_2005} and the p4est~\cite{BursteddeWilcoxGhattas11} library is used for such purpose. \mo{In this method, one must enforce constraints to ensure the conformity of the global \ac{fe} space, which amounts to compute constraints between coarser and refined geometrical entities shared by two cells with different level of refinement \cite{demkowicz_book}}. For Lagrangian elements, restriction operators in geometrical sub-entities are generally obtained by evaluating the shape functions associated with the coarse side of the face at the interpolation points of the shape functions on the refined side of the face~\cite{Bangerth_2009}. The idea is conceptually equivalent for edge elements but considerably more complex to implement because \acp{dof} do not represent nodal values but moments on top of geometrical entities. In this work, we follow an approach based on the relation between a Lagrangian \emph{pre-basis}, where moments are trivial to evaluate, and the edge basis, so we can avoid the evaluation of edge moments on the refined cell to compute constraints.

Let $\mathcal{T}^{'}_h$ be a conforming partition of $\Omega$ into a set of hexahedra (quadrilateral in 2D) geometrical cells $\geophy$. $\mathcal{T}^{'}_h$ can, e.g., be as simple as a single quadrilateral or hexahedron. Starting from $\mathcal{T}^{'}_h$, hierarchical \ac{amr} is a multi-step process in which at each step, some cells of the input mesh are marked for refinement. A cell marked for refinement is partitioned into four (2D) or eight (3D) children cells by bisecting all cell edges $\mathcal{E}_\geophy$. Let us denote by $\mathcal{T}_h$ the resulting partition of $\Omega$. $\mathcal{T}_h$ can be thought as a collection of quads (2D) or octrees (3D) where the cells of $\mathcal{T}^{'}_h$ are the roots of these trees, and children cells branch off their parent cells. The leaf cells in this hierarchy form the mesh in the usual meaning, i.e., $\mathcal{T}_h$. Thus, for every cell $K \in \mathcal{T}_h$ we can define $\ell(K)$ as the level of $K$ in the aforementioned hierarchy, where $\ell(K)=0$ for the root cells, and $\ell(K)=\ell({\rm parent}(K))+1$ for any other cell. Clearly, the cells in $\mathcal{T}_h$ can be at different levels of refinement. Thus, these meshes are {\em non-conforming}. In order to complete the definition of $\mathcal{T}_h$, let us introduce the concept of \emph{hanging} geometrical entities. For every cell $\geophy \in \mathcal{T}_h$, consider its set of vertices $\mathcal{N}_\geophy$, edges $\mathcal{E}_\geophy$ and faces $\mathcal{F}_\geophy$. We can represent the set of geometrical entities that have lower dimension than the cell by $\mathcal{G}_\geophy=\mathcal{N}_\geophy \cup \mathcal{E}_\geophy \cup \mathcal{F}_\geophy$. Its global counterpart is defined as $\mathcal{G}_{\mathcal{T}}=\cup_{\geophy\in\mathcal{T}_h} \mathcal{G}_\geophy$. For every geometrical entity $s\in\mathcal{G}_{\mathcal{T}}$, we can represent by $\mathcal{T}_s$ the set of cells $\geophy \in \mathcal{T}_h$ such that $s \in \mathcal{G}_\geophy$. {Additionally, $\tilde{\mathcal{T}}_s$ is defined as the set of cells $\geophy \in \mathcal{T}_h$ such that $s \subsetneq s'$ for some $s'\subset \mathcal{G}_\geophy$. Roughly speaking, one set contains all the cells where the geometrical entity is found while the other set contains all the cells where the geometrical entity is a strict subset of a coarser geometrical entity}. A  geometrical entity $s\in\mathcal{G}_\mathcal{T}$ is \emph{hanging} (or \emph{improper}) if and only if {there exists at least one neighbouring cell in $\tilde{\mathcal{T}_s}$.}
We represent the set of hanging geometrical entities with $\mathcal{G}_\mathcal{T}^{\rm hg}$.
{The definition of the proposed \ac{amr} approach is completed by imposing the so-called 2:1 balance restriction, which is used in mesh adaptive methods by a majority of authors as a reasonable trade-off between performance gain and complexity of implementation~\cite{Tu_2005, BursteddeWilcoxGhattas11, demkowicz_book}.}
\begin{definition} \label{def-21_balance}
A $d$-tree ($d=\{2,3\}$) is 2:1 $k$-balanced if and only if, for any cell $K\in \mathcal{T}_h$ there is no $s \in \mathcal{G}_\geophy$ of dimension $m\in[k,d)$ having non-empty intersection with the closed domain of another finer cell $K'\in \tilde{\mathcal{T}_h}$ such that $\ell(K') - \ell(K) > 1$.
\end{definition}
{For edge \acp{fe} it is enough to consider 1-balance since vertices do not contain any associated \ac{dof}.}
For the sake of clarity, in Figs.~\ref{fig-hanging_node} and \ref{fig-2lev_hang_nodes}, allowed \emph{hanging} geometrical entities are depicted in red, whereas in Fig.~\ref{fig-2lev_hang_nodes}, not allowed ones are shown in blue. Clearly, the latter mesh is the result of a refinement process that does not accomplish the 2:1 balance, thus not permitted in our \ac{amr} approach. Note that in order to enforce the 2:1 balance in the situation depicted in Fig. \ref{fig-2lev_hang_nodes}, one would need to apply additional refinement to some cells with lower values for $\ell(K)$ until the restriction (\ref{def-21_balance}) is satisfied.

\begin{figure}[t!]
    \centering
    \begin{subfigure}[t]{0.20\textwidth}
        \includegraphics[width=\textwidth]{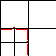}
        \caption{\emph{Hanging} entities marked in red.}
        \label{fig-hanging_node}
    \end{subfigure}
    \hspace{1cm}
    \begin{subfigure}[t]{0.20\textwidth}
        \includegraphics[width=\textwidth]{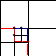}
        \caption{Not permitted \emph{hanging} entities marked in blue.}
        \label{fig-2lev_hang_nodes}
    \end{subfigure}
       \hspace{1cm}
        \begin{subfigure}[t]{0.20\textwidth}
        \includegraphics[width=\textwidth]{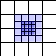}
        \caption{Centered refinement pattern. $\ell\in\{2,3,4\}$ from clearest to darkest.}
        \label{fig-refmesh}
    \end{subfigure}

    \caption{$h$-adaptive refined (single octree) {\em non-conforming} meshes.}
    \label{fig-nc_meshes}
\end{figure}

\subsection{Conformity of the global \ac{fe} space}
In order to preserve the conformity of the \ac{fe} space $\mathcal{ND}_k(\Omega)$, we cannot allow an arbitrary value for \acp{dof} placed on top of \emph{hanging} geometrical entities, which will be denoted by \emph{hanging} \acp{dof}. Our approach is to eliminate the \emph{hanging} \acp{dof} of the global system by defining a set of constraints such that curl-conformity is preserved. We propose an algorithm that computes the edge \ac{fe} constraints by relying on the ones of Lagrangian \acp{fe} and the change of basis in Sect. \ref{subsec-ned_shapes}. This way, one can reuse existing ingredients in a nodal-based \ac{amr} code and work already required to define the edge \ac{fe} shape functions. 

The computation of constraints requires some preliminary work at the geometrical level, i.e., independent of the \ac{fe} space being used. Let us first compute the set of {hanging} geometrical entities $\mathcal{G}_\mathcal{T}^{\rm hg}$. For every $g \in \mathcal{G}_\mathcal{T}^{\rm hg}$, let us compute the coarser geometrical entity $G(g) \in \mathcal{G} \setminus \mathcal{G}_\mathcal{T}^{\rm hg}$ that contains it. Let us assign to every $g \in \mathcal{G}_\mathcal{T}$ one cell such that $K(g) \in \mathcal{T}_g$. With this information, for every $g \in \mathcal{G}_\mathcal{T}^{\rm hg}$, we can extract the fine cell $K_h \equiv K(g)$ and the coarse cell $K_{2h} \equiv K(G(g))$ (see \fig{fig-cells_a}).

The coarse cell $K_{2h}$ can be refined (once) to meet the level of refinement of $K_h$. (The coarse cell is only refined for the computation of constraints so the original mesh is not affected.) Let us consider its children cells $\{ K_h^s \}_{s=1}^{2^d}$ obtained after isotropic refinement by the procedure exposed in \sect{subsec-AMR}, see \fig{fig-cells_b}. We represent the patch of subcells with $\tilde{K}_h = \bigcup_{i=1}^{2^d} K_{h}^s$. We can determine a subcell index $s(g)$ such that $g \in \mathcal{G}_{K^{s(g)}_h}$; $s(g)$ is not unique in general, see \fig{fig-cells_b}.

\begin{figure}[t!]
    \centering
    \begin{subfigure}[t]{0.3\textwidth}
        \includegraphics[width=\textwidth]{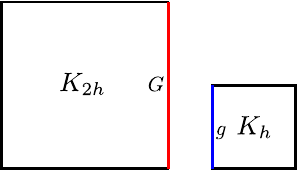}
        \caption{Fine cell $K_h \equiv K(g)$ and coarse cell $K_{2h} \equiv K(G(g))$.}
        \label{fig-cells_a}
    \end{subfigure}
    \hspace{1cm}
    \begin{subfigure}[t]{0.30\textwidth}
        \includegraphics[width=\textwidth]{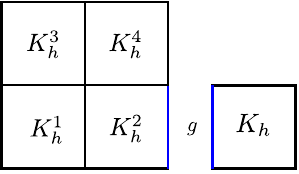}
        \caption{Coarse cell refined into children cells $\{ K_h^s \}_{s=1}^4$. In this example, $s(g)=2$. }
        \label{fig-cells_b}
    \end{subfigure}
    \caption{The coarse cell $K_{2h}$ is refined in order to meet the highest level of refinement. Then, the index $s(g)$ can be determined. }
    \label{fig-cells_scheme}
\end{figure}

\subsubsection{Constraints for Lagrangian \acp{fe}}
In this section, we compute the constraints for Lagrangian \ac{fe} spaces. Using the geometrical information above, given a hanging geometrical entity $g$, let us consider the Lagrangian \ac{fe} spaces $\lagesp2hk$, $\lagespthk$ and $\lagesphk$ (see \sect{sec-assembly} for the definitions, Figs. \ref{fig-or_fespace} and \ref{fig-Roperator} for an illustration). The objective is to compute the constraints over \acp{dof} on $g \in \mathcal{G}_\mathcal{T}^{\rm hg}$ that enforce global continuity. This continuity is enforced by evaluating the coarse cell $K_{2h}$ shape functions on the fine cell $K_h$ nodes (\acp{dof}) on $g$. Since both cells share the \ac{fe} order, such set of constraints leads to full continuity across $g$~\cite[Ch.\ 3]{Solin2003}.  Let us explain the steps followed to compute these constraints.

Clearly, $\lagesp2hk \subset \lagespthk$, so we can apply the Lagrangian interpolant to inject $\uu^{2h} \in \lagesp2hk$ into $\lagespthk$. For this purpose, consider the set of original Lagrangian basis functions $\{ \varphi^{j} \}_{j=1}^{n_k}$ spanning $\lagesp2hk$, which has cardinality $n_k=\prod_i^d (k_i +1)$ (see \sect{subsec-polspaces} for details). Further, consider the set of moments $\{\sigma_i\}_{i=1}^{\tilde n_{k}}$ uniquely defining a solution in $\lagespthk$, which has cardinality $\tilde n_{k} = \prod_i^d (2k_i+1)$ (see \sect{subsec-lag_polys}). Given the vector of \ac{dof} values of $\uu^{2h}$, the restriction operator $\boldsymbol{R}: \lagesp2hk' \rightarrow \lagespthk'$
  \begin{align}\label{eq-Rl}
R_{ij} \doteq  \sigma_i( \varphi^j ), \qquad i = 1, \ldots, \tilde n_k, \quad j = 1, \ldots, n_k,
\end{align}
provides the \acp{dof} of the interpolated function.
 Further, note that given a \emph{hanging} \ac{dof} $u^{h,i}$ on top of $g$, it is easy to check that $R_{ij}=0$ if \ac{dof} $u^{2h,j}$ is not on $G(g)$. 

  For every subcell $K_h^s$ one can readily determine the map $w_s(\cdot)$ such that, given the local cell index $i\in\{1,\ldots,n_k \}$ for moments defined on $\lagesphk'$ ($s\in \{1,\ldots,2^d\}$), returns a global moment index $i'\in \{ 1,\ldots,\tilde n_{k}\}$ in the space $\lagespthk'$. It leads to the subcell restriction operator (of \ac{dof} values) $\R^s : \lagesp2hk' \rightarrow \lagesphk'$ defined as
  \begin{align}\label{eq-Rned}
    R_{ij}^s \doteq R_{w_s(i) j}  \qquad i = 1, \ldots, n_k, \quad j = 1, \ldots, n_k.
\end{align}
 A representation of the action of $\R$ and $\R^s$ can be seen in \fig{fig-Roperator}. Note that they are independent of the level of refinement and the cell in the physical space; they are computed only once at the reference cell.

We can  identify every \ac{dof} $i \in \lagesphh' |_g$ with a \ac{dof} $i' \in \mathcal{L}_{\boldsymbol{k}}(K_h^{s(g)})' |_g$ as for conforming Lagrangian meshes; two local \ac{dof} values must be identical if their corresponding nodes are located at the same position (see \fig{fig-eqclass}). Finally, \emph{the \ac{dof} value $i \in \lagesphh'|_g$ is constrained by the \ac{dof} values of the coarse cell through the row $i'$ of $\boldsymbol{R}^{s(g)}$}.

\begin{figure}[t!]
    \centering
    \begin{subfigure}[t]{0.25\textwidth}
        \includegraphics[width=\textwidth]{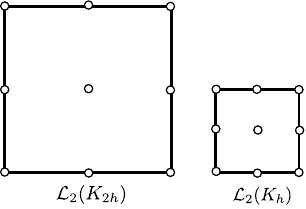}
        \caption{FE spaces representation.}
        \label{fig-or_fespace}
    \end{subfigure}
    \hspace{0.5cm}
    \begin{subfigure}[t]{0.32\textwidth}
        \includegraphics[width=\textwidth]{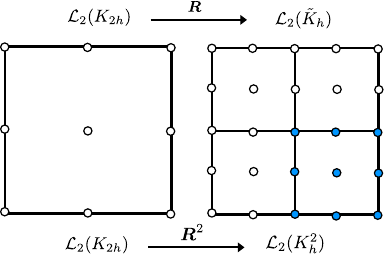}
        \caption{A function in the coarse cell injected in the refined space  or its restriction to a subcell (denoted in blue).}
        \label{fig-Roperator}
    \end{subfigure}
       \hspace{0.5cm}
    \begin{subfigure}[t]{0.25\textwidth}
        \includegraphics[width=\textwidth]{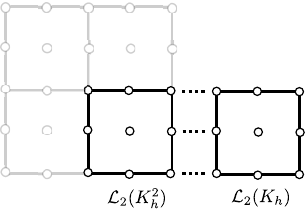}
        \caption{\acp{dof} identification as conforming meshes.}
        \label{fig-eqclass}
    \end{subfigure}
    \caption{Scheme for enforcing continuity at common entities for two cells with different level of refinement with second order Lagrangian \acp{fe}.}
    \label{fig-dofs_scheme}
\end{figure}


\subsubsection{Constraints for edge \acp{fe}} In order to compute the constraints for edge \acp{fe} one could follow a similar procedure as the one above, see \fig{fig-edge_dofs_scheme} for an illustration. Nevertheless, we propose a different approach for the implementation of restriction operator in edge \acp{fe} (see \fig{fig-Redge}), which allows us to reuse the restriction operator defined for nodal \acp{fe} and avoids the evaluation of edge moments in subcells. Besides, our approach for computing the edge \ac{fe} constraints is applicable to \emph{any} \ac{fe} space that relies on a {pre-basis} of Lagrangian polynomials plus a change of basis (e.g., Raviart-Thomas \acp{fe}) without \emph{any} additional implementation effort.

The idea is to build the restriction operator for edge \acp{fe} as a composition of operators. We recall the change of basis matrix $\boldsymbol{Q}: \lagespk(K) \rightarrow  \nedespk(K)$ between Lagrangian and edge \ac{fe} basis functions (see \sect{subsec-ned_shapes}).
It leads to the adjoint operator $\boldsymbol{Q}^T: \nedespk(K)' \rightarrow \lagespk(K)'$  
and its inverse $\boldsymbol{Q}^{-T}: \lagespk(K)' \rightarrow \nedespk(K)'$.
As a result, given a subcell $K_h^s$, we can define the restriction operator $ \hat \R^s \doteq \boldsymbol{Q}^{-T} \R^s \boldsymbol{Q}^T :  \nedesp2hk' \rightarrow \nedesphk'$ that takes the edge \ac{fe} \ac{dof} values in the coarse cell $K_{2h}$ and provides the ones of the interpolated function (see \eqref{eq:interpolant}) in $K_h^s$. Again, $\hat \R^s$ can be computed once at the reference cell. An illustration of the sequence of spaces is shown in \fig{fig-Ropedge}.

Given a hanging edge/face $g \in \mathcal{G}_\mathcal{T}^{\rm hg}$ (vertices do not have associated \acp{dof} for edge \acp{fe}), the constraints of its \acp{dof} are computed as follows. We can identify every \ac{dof} $i \in \nedesphh' |_g$ with a \ac{dof} $i' \in \mathcal{ND}_{{k}}(K_h^{s(g)})' |_g$ as for conforming meshes (see the equivalence class in Sect. \ref{sec:global_space} and \fig{fig-edge_dofeq}). As a result, \emph{the \ac{dof} value $i \in \nedesphh'|_g$ is constrained by the \ac{dof} values of the coarse cell through the row $i'$ of $\hat \R^{s(g)}$.}

\begin{figure}[t!]
    \centering
    \begin{subfigure}[t]{0.26\textwidth}
        \includegraphics[width=\textwidth]{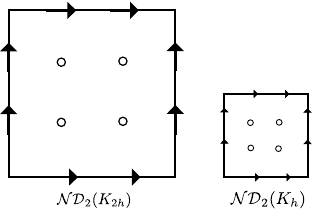}
        \caption{FE spaces representation.}
        \label{fig-edge_spaces}
    \end{subfigure}
    \hspace{0.5cm}
    \begin{subfigure}[t]{0.35\textwidth}
        \includegraphics[width=\textwidth]{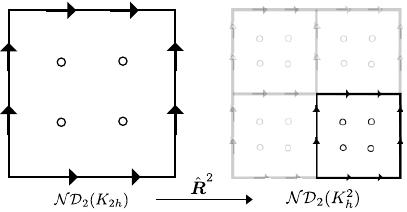}
        \caption{Restriction of a function injected in the refined space to a subcell.}
        \label{fig-Redge}
    \end{subfigure}
       \hspace{0.5cm}
    \begin{subfigure}[t]{0.26\textwidth}
        \includegraphics[width=\textwidth]{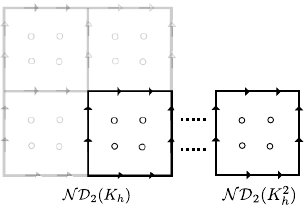}
        \caption{\acp{dof} identification as conforming meshes.}
        \label{fig-edge_dofeq}
    \end{subfigure}
    \caption{Scheme for enforcing continuity at common entities for two cells with different level of refinement with second order edge \acp{fe}.}
    \label{fig-edge_dofs_scheme}
\end{figure}

\begin{figure}[t!]
    \centering
   \includegraphics[width=0.7\textwidth]{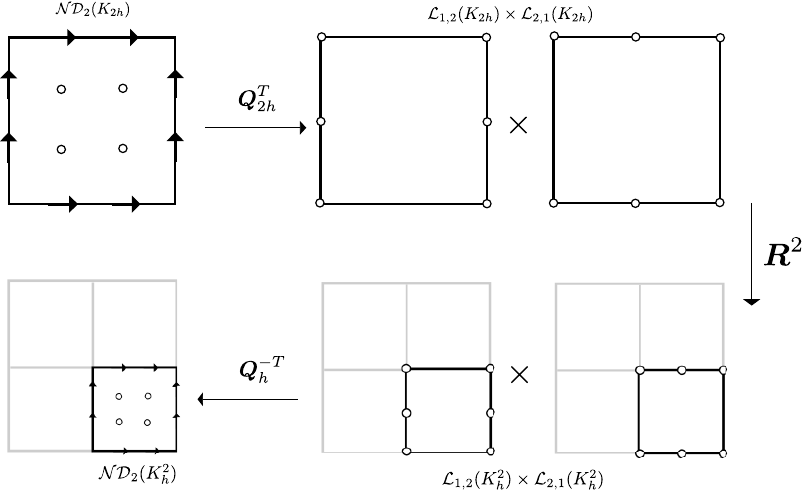}
    \caption{Sequence of spaces in order to compute the restriction operator entries for $K^2_{h}$ with second order edge \acp{fe}. Representation of edge and inner \acp{dof} by arrows and nodes, resp. }
    \label{fig-Ropedge}
\end{figure}

\subsubsection{Global assembly for non-conforming meshes}
The definition of the constraints is used in the assembly of the linear system as follows. First of all, let us write the problem \eq{eq-variational_form} in algebraic form. The solution $\uu \in \mathcal{ND}_k(\Omega)$ is expanded by $\{ \phi^a \}_{a=1}^{n}$. The elemental matrices are defined as $\Mm_{ij} = \int_K \beta \phi^j \cdot \phi^i$ and $\Km_{ij} = \int_K (\alpha \bsnabla \times \phi^j) \cdot (\bsnabla \times \phi^i)$, whereas the elemental right-hand side is the discrete vector $\fm_i = \int_K \fm_K \cdot \phi^i$ for $i,j \in \{1, \ldots, n\}$. The usual assembly for every $K\in\mathcal{T}_h$ is performed to obtain the global matrix and array, hence the algebraic form reads $\Am \uu = \f$, where $\Am=\Km + \Mm$. Let us now distinguish between the set of \emph{conforming} \acp{dof} $\uu_c$ and the set of \emph{non-conforming} \acp{dof} $\uu_{nc}$, whose cardinalities are $n_c
+n_{nc}=n$. One can write the constraints that $\uu_{nc}$ must satisfy in compact form, $\uu_{nc} = \boldsymbol{G} \uu_c $, where the entries of $\boldsymbol{G}$ (of dimension $n_{nc}\times n_c$) are given by the described procedure for obtaining the \emph{constraining} factors between the \emph{constrained} and the \emph{constraining} part. Then, the solution $\uu_h$, expressed in terms of both types of \acp{dof} can be written as
\begin{eqnarray}\label{eq-constrained_sol}
  \uu_h = \begin{bmatrix} \uu_c \\ \uu_{nc} \end{bmatrix} =
  \begin{bmatrix} \boldsymbol{I} \\ \boldsymbol{G} \end{bmatrix} \uu_c = \boldsymbol{P} \uu_c,
\end{eqnarray}
and we can write the constrained problem only in terms of the \emph{conforming} \acp{dof} $\uu_c$ as
\begin{eqnarray}
\bar{\boldsymbol{\mathcal{A}}} \uu_c &=& \bar{\boldsymbol{f}},
\end{eqnarray}
where $\bar{\boldsymbol{\mathcal{A}}} = \boldsymbol{P}^T \boldsymbol{\mathcal{A}} \boldsymbol{P}$ and $\bar{\boldsymbol{f}}=\boldsymbol{P}^T \boldsymbol{f}$. Our implementation directly builds the constrained operator $\bar{\boldsymbol{\mathcal{A}}}$ by \emph{locally} applying the constraints to eliminate $\uu_{nc}$ \acp{dof} using the constraints given by $\boldsymbol{G}$ in the assembly process (see~\cite{lssc_hadap_2018} for further details). Once the solution for \emph{conforming} \acp{dof} is obtained, \emph{hanging} \ac{dof} values are recovered by \eq{eq-constrained_sol} as a postprocess.

\section{Numerical experiments}\label{sec-results}
In this section, we test the implementation of edge elements and $H$(curl)-conforming spaces in \FEMPAR~\cite{fempar-web-page}, a general purpose, parallel scientific software for the \ac{fe} simulation of complex multiphysics problems governed by \acp{pde} written in Fortran200X following object-oriented principles. It supports several computing and programming environments, such as, e.g., multi-threading via OpenMP for moderate scale problems and hybrid MPI/OpenMP for HPC clusters. See~\cite{badia-fempar} for a thorough coverage of the software architecture of \FEMPAR, which is distributed as open source software under the terms of the GNU GPLv3 license. Regarding the content of this document, \FEMPAR supports arbitrary order edge \acp{fe} on both hexahedra and tetrahedra, on either structured and unstructured conforming meshes, and also mesh generation by adaptation using hierarchically refined octree-based meshes. In order to test our implementation, we will compare both theoretical and experimental convergence rates. Generally, log-log plots of the computed error in the considered norms ($L_2$-norm or $H$(curl)-norm) against different values of $h$ or number of \acp{dof} will be shown.
Let us first cite some \emph{a priori} error estimates for edge \acp{fe} of the first kind.   In the $H$(curl)-norm, we find the following optimal estimate, presented in~\cite{alonso_1999}:
\begin{theorem} If $\mathcal{T}_h$ is a regular family of triangulations on $\Omega$ for $h>0$, then there exists a constant $C>0$ such that
\begin{eqnarray}
\norm{\vv - {\pi}_h^k\vv}_{H({\rm curl})} \leq C h^{{\rm min}\{r,k\}} \norm{\vv}_{H^r({\rm curl})}, \label{eq-err_Hc}
\end{eqnarray}
for all $\vv \in H^r({curl})$, where $r>\frac{1}{2}$ determines the regularity of the function $\vv$, which is valid for $H$(curl)-conforming elements. If the function $\vv$ is smooth enough to have bounded derivatives such that $s>k$, then the estimate states that superlinear convergence can be achieved as we increase the polynomial order $k$.
\end{theorem}

For the $L_2(\Omega)$ approximation, the following convergence rate can be expected~\cite{nedelec_mixed_1980}.
\begin{theorem} If $\mathcal{T}_h$ is a regular family of triangulations on $\Omega$ for $h > 0$, then there exists a constant $C>0$ such that
\begin{eqnarray}
\norm{\vv - {\pi}_h^k\vv}_{L_2(\Omega)} \leq C h^{k} \norm{\vv}_{H^k(\Omega)}. \label{eq-err_L2}
\end{eqnarray}
\end{theorem}

%
%
%

In general, we choose analytical functions that do not belong to the \ac{fe} space. We show convergence plots with respect to the mesh size for uniform mesh refinement or the number of \acp{dof} for $h$-adaptive mesh refinement. In all cases, we will solve the reference problem \eq{eq-dcurl_p} with homogeneous, scalar parameters $\alpha=\beta=1$. Unless otherwise stated, the results are computed in the unit box domain $\Omega:=[0,1]^d$. We utilize the method of \emph{manufactured} solutions, i.e., to plug an analytical function $\uu^*$ in the exact form of the problem and obtain the corresponding source term that verifies the equation. Then, one can solve the problem for the unknown $\uu$ so that the computed solution must converge to the exact solution with mesh refinement.

\subsection{Uniform mesh refinement}
In this section, the experimental rate of convergence will be numerically included in the legend as the value for the slope computed with the two last available data points in each plot. We will make use of the following analytical function and source term for all the 2D cases presented in this section:
\begin{eqnarray*}
\uu^* = \begin{bmatrix}
\cos(\pi x_1 ) \cos(\pi x_2) \\
\sin(\pi x_1 ) \sin(\pi x_2)
\end{bmatrix},  & &
\f = (2\pi^2 + 1) \uu^*,
\end{eqnarray*}
whereas in 3D cases the analytical function and corresponding source term are given by
\begin{eqnarray*}
\uu^* = \begin{bmatrix}
\cos(\pi x_1) \cos(\pi x_2) \\
\sin(\pi x_2) \sin(\pi x_3) \\
\cos(\pi x_1) \cos(\pi x_3)
\end{bmatrix}, & &
\f = (\pi^2 +1)\uu^* + \pi^2 \begin{bmatrix}
\sin(\pi x_1) \sin(\pi x_3) \\
\sin(\pi x_1) \sin(\pi x_2) \\
\cos(\pi x_2) \cos( \pi x_3)
\end{bmatrix}.
\end{eqnarray*}
Dirichlet boundary conditions are strongly imposed over the entire boundary, i.e., $\partial \Omega_D := \partial \Omega$, where we enforce the tangent component of the analytical function $\uu^*_\tau$.

\begin{figure}[t!]
    \centering
    \begin{subfigure}[t]{0.45\textwidth}
        \includegraphics[width=\textwidth]{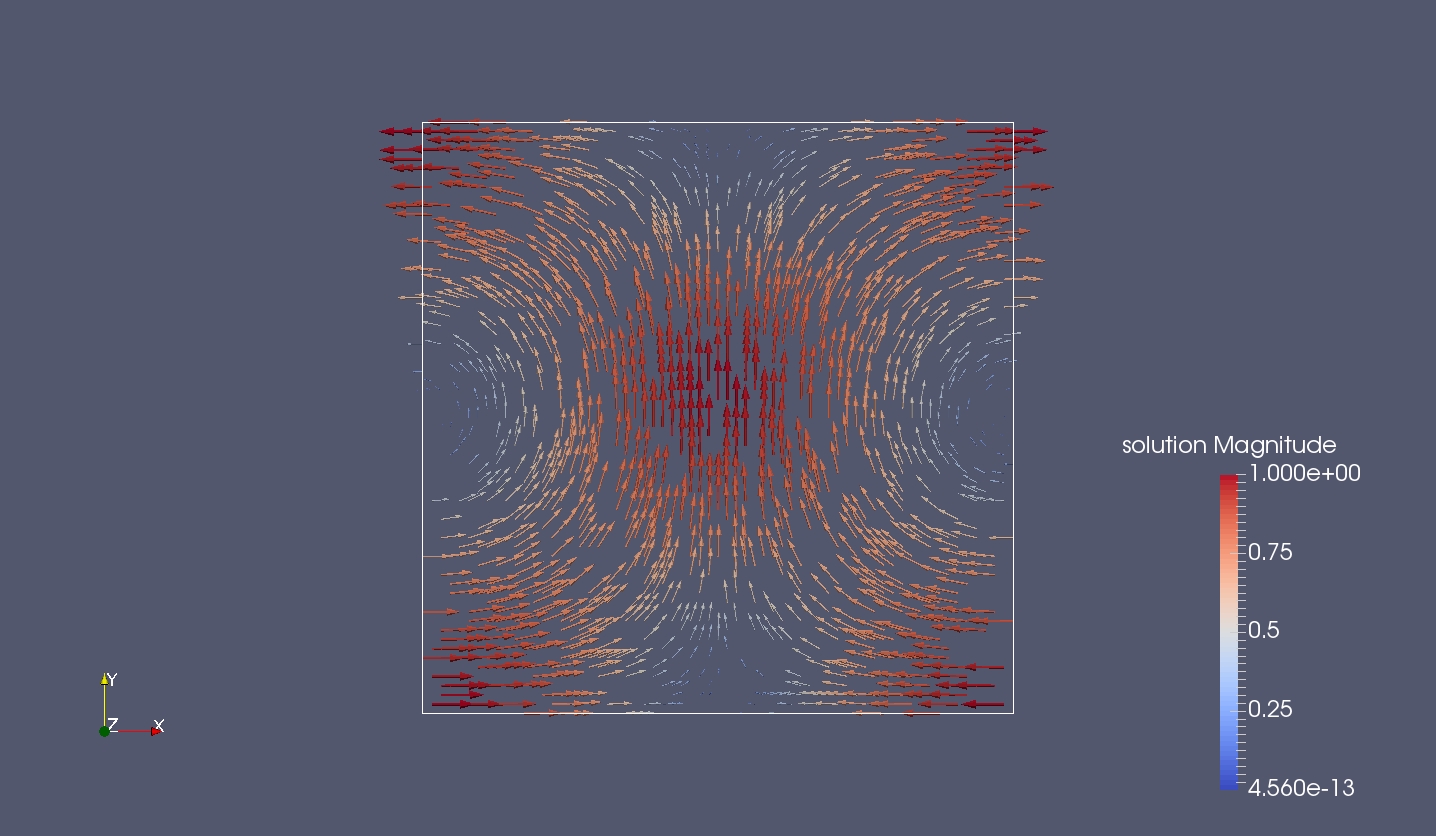}
        \caption{ Analytical function $\uu^*$ in the 2D case. }
        \label{fig-2D_sol}
    \end{subfigure}
    \begin{subfigure}[t]{0.45\textwidth}
        \includegraphics[width=\textwidth]{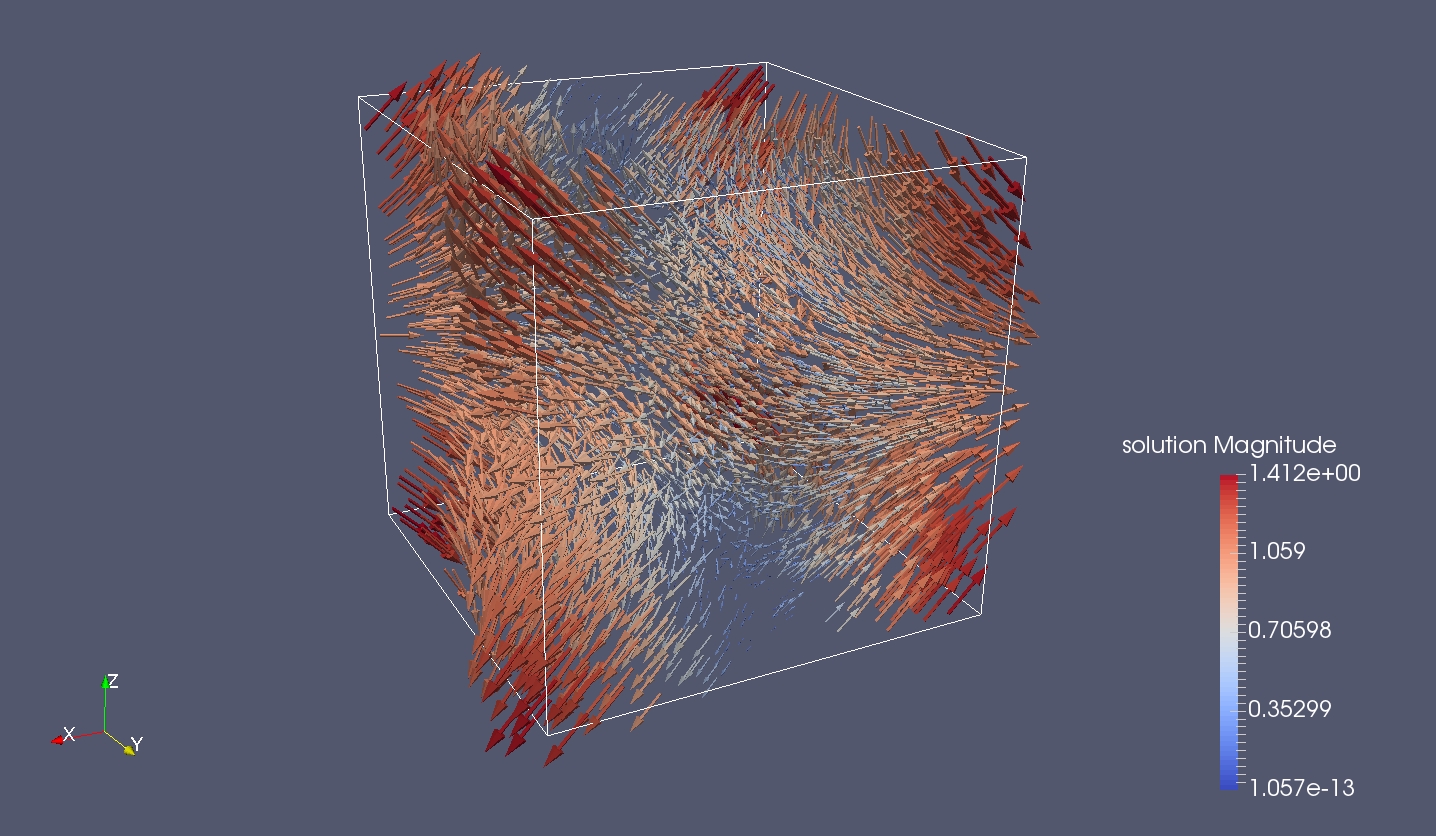}
        \caption{ Analytical function $\uu^*$ in the 3D case. }
        \label{fig-3D_sol}
    \end{subfigure}
    \caption{Analytical functions.}\label{fig-anasol}
\end{figure}

\subsubsection{Hexahedral meshes}

Structured simulations are performed with a structured mesh on the unit box domain with the same number of elements $n_K$ in each direction. Figs.~\ref{fig-err_2D_hex} and \ref{fig-err_3D_hex} show the convergence rates with the order of the element $k$ for 2D and 3D cases, resp. In all cases, the expected convergence ratio (see (\ref{eq-err_Hc}) and (\ref{eq-err_L2})) is achieved, presented up to $k=6$ in 2D and $k=4$ in 3D. 

\begin{figure}[t!]
    \centering
    \begin{subfigure}[t]{0.45\textwidth}
        \includegraphics[width=\textwidth]{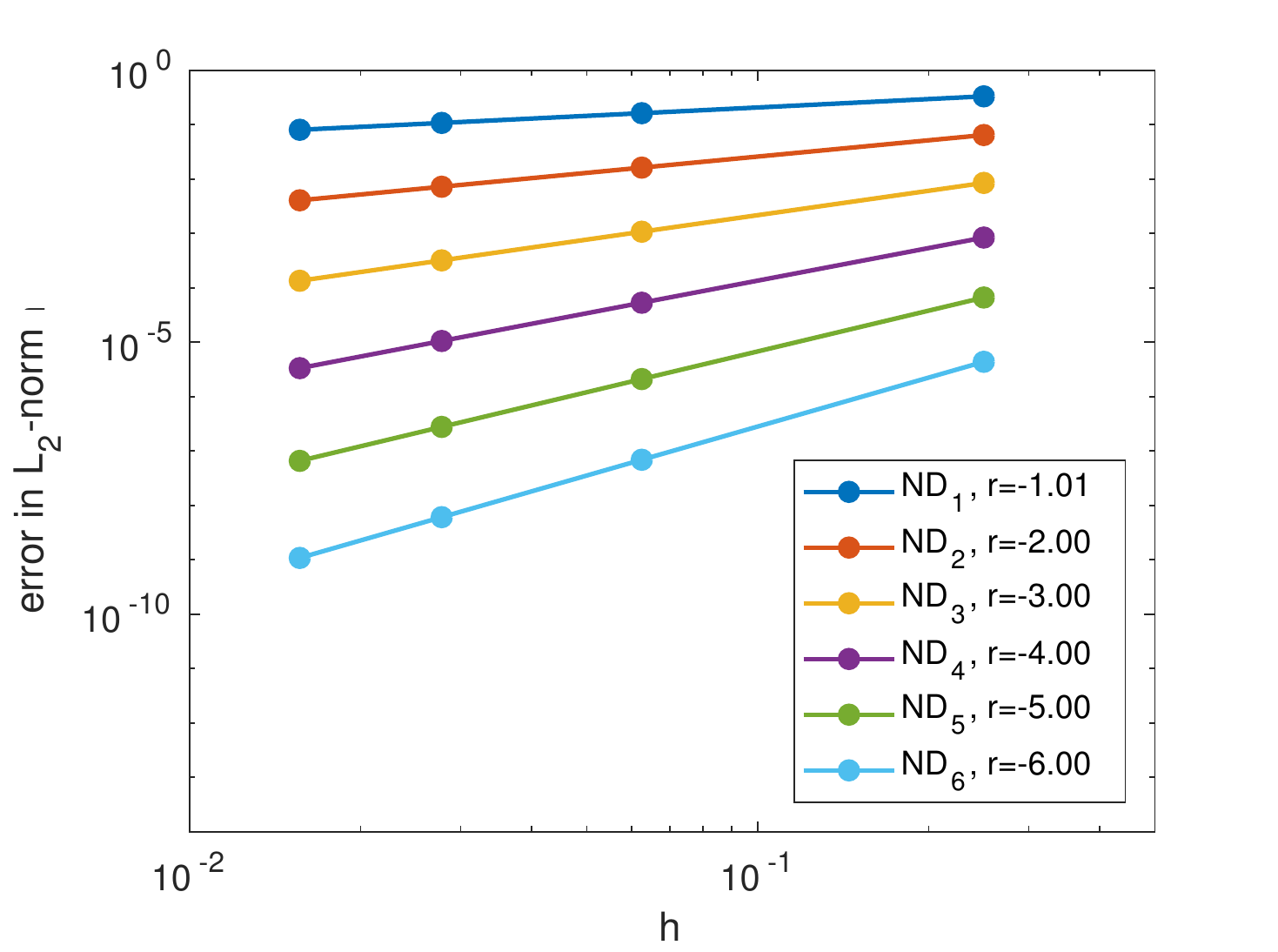}
        \caption{ $L_2$-norm of the error. }
        \label{fig-l2error_hex_2D}
    \end{subfigure}
    \begin{subfigure}[t]{0.45\textwidth}
        \includegraphics[width=\textwidth]{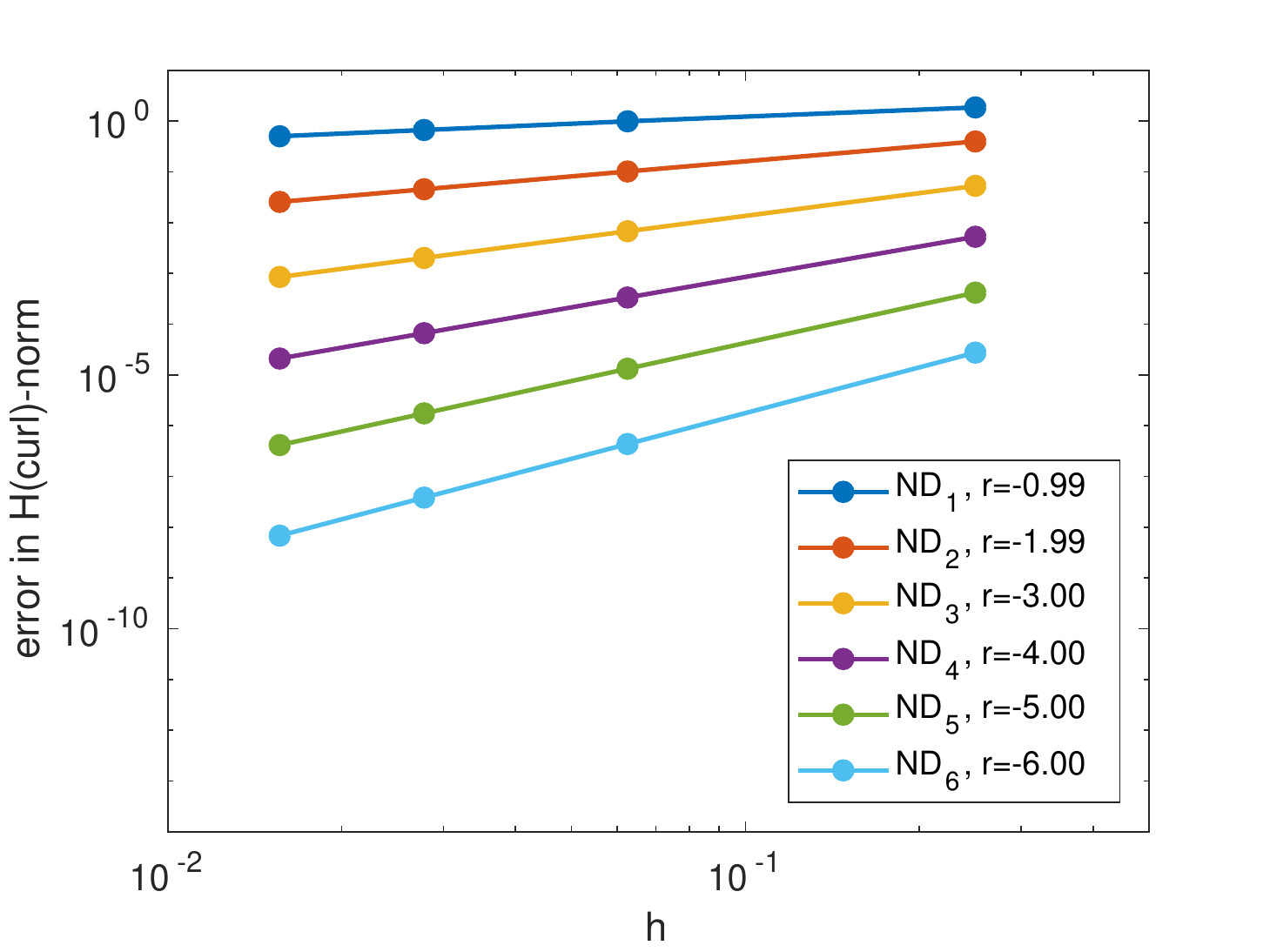}
        \caption{$H$(curl)-norm of the error. }
        \label{fig-Hcurl_error_hex_2D}
    \end{subfigure}
    \caption{Error norms for different orders 2D hexahedral edge \acp{fe}.}\label{fig-err_2D_hex}
\end{figure}

\begin{figure}[t!]
    \centering
    \begin{subfigure}[t]{0.45\textwidth}
        \includegraphics[width=\textwidth]{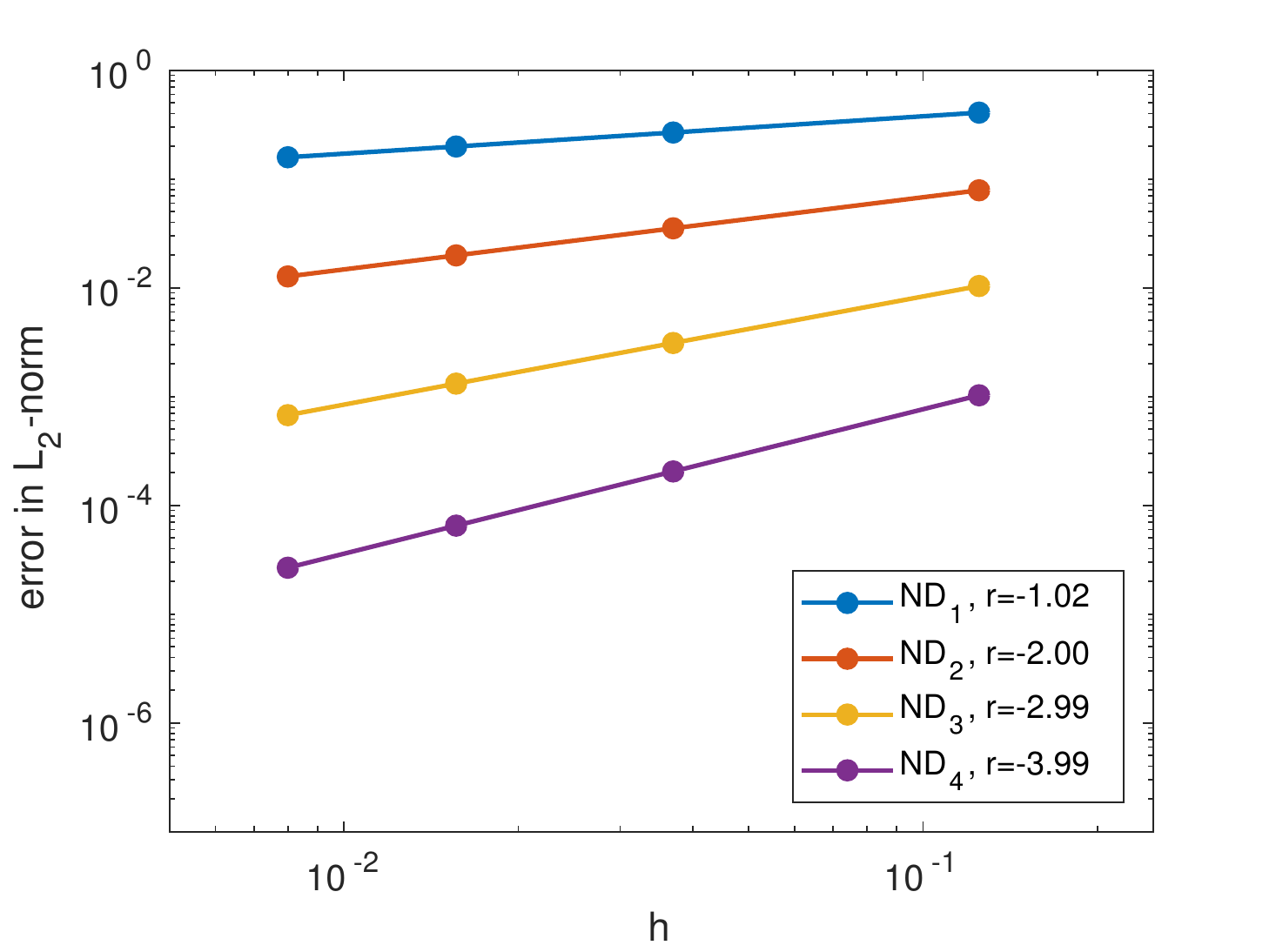}
        \caption{ $L_2$-norm of the error. }
        \label{fig-l2error_hex_3D}
    \end{subfigure}
    \begin{subfigure}[t]{0.45\textwidth}
        \includegraphics[width=\textwidth]{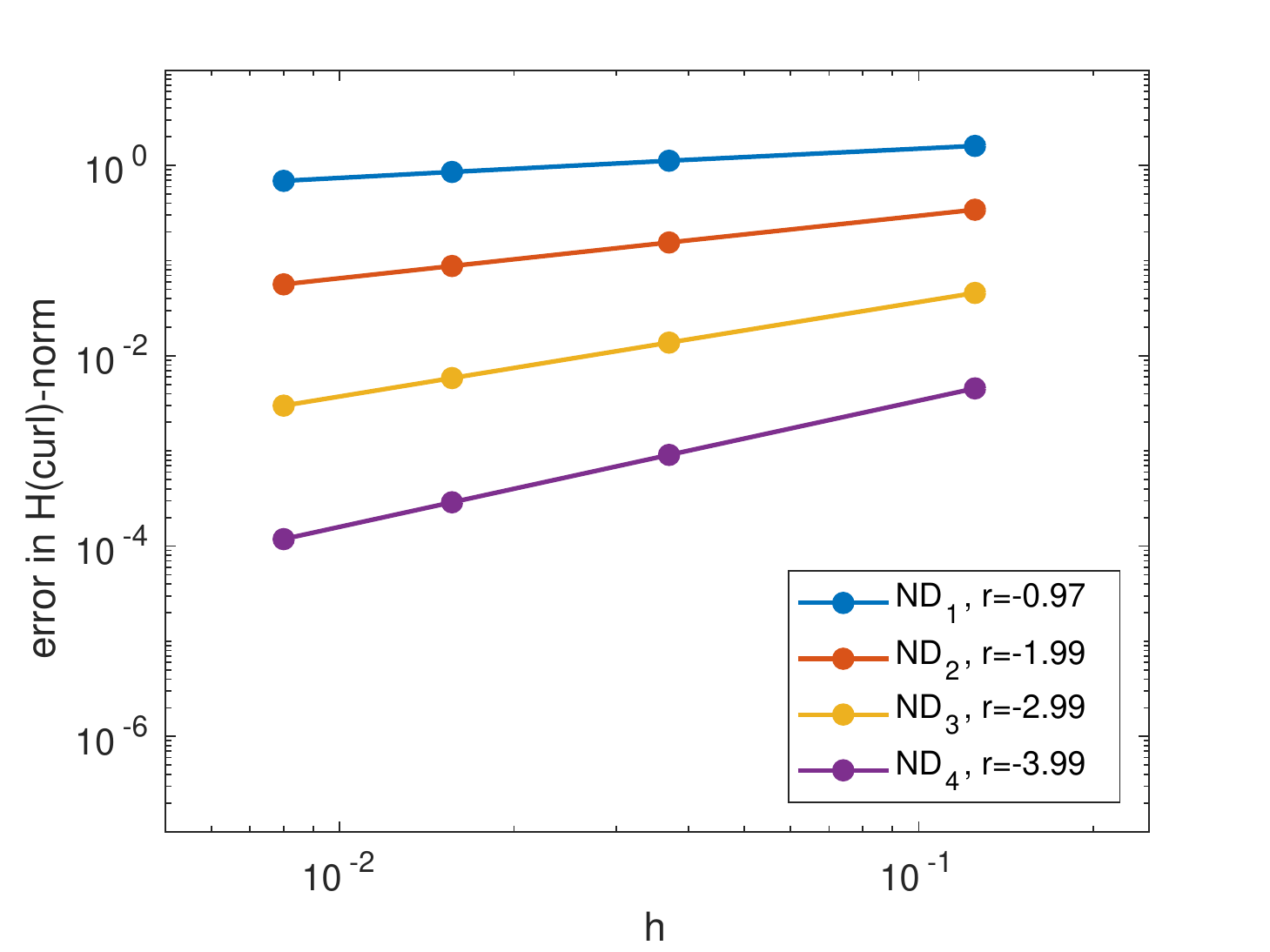}
        \caption{$H$(curl)-norm of the error. }
        \label{fig-Hcurl_error_hex_3D}
    \end{subfigure}
    \caption{Error norms for different orders 3D hexahedral edge \acp{fe}.}\label{fig-err_3D_hex}
\end{figure}

\subsubsection{Tetrahedral meshes}
Edge \acp{fe} are tested in this section with tetrahedral mesh partitions of the domain $\Omega$. Consider a family of tetrahedral meshes $\{\mathcal{T}_m \}_{m=1}^M$, obtained by structured, hexahedral meshes plus triangulation of hexahedral cells. Here the element size denotes the usual definition $h = \max_{K\in T_m} h_K$, being $h_K$ the diameter of the largest circumference or sphere containing $K$ for 2D and 3D, resp. Convergence results in Figs. \ref{fig-err_2D_tet} and \ref{fig-err_3D_tet} are computed with this family of meshes. In all cases, computed convergence ratios are consistent with the estimates in~(\ref{eq-err_Hc}) and (\ref{eq-err_L2}).
\begin{figure}[t!]
    \centering
    \begin{subfigure}[t]{0.45\textwidth}
        \includegraphics[width=\textwidth]{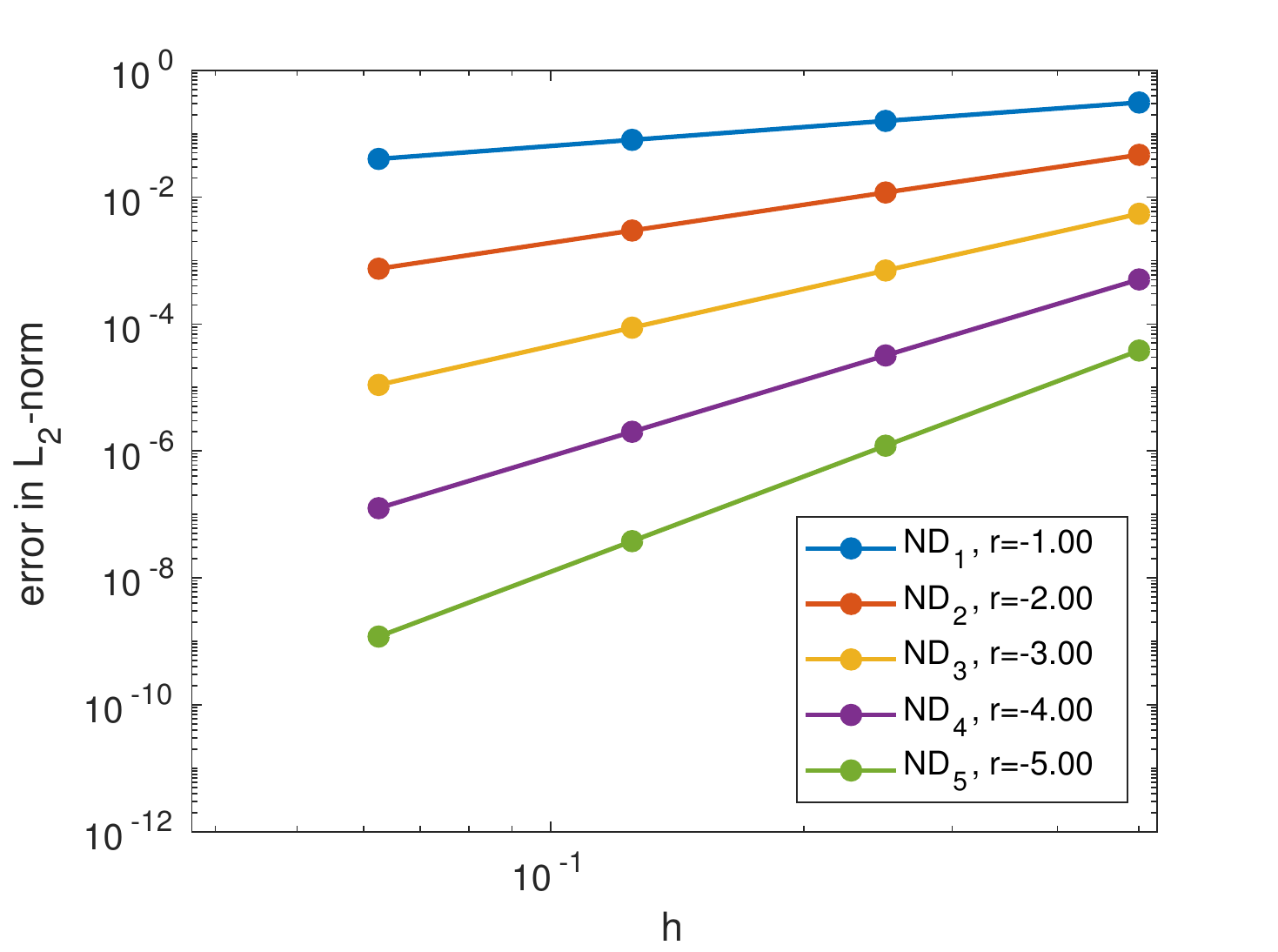}
        \caption{ $L_2$-norm of the error. }
        \label{fig-l2error_tet_2D}
    \end{subfigure}
    \begin{subfigure}[t]{0.45\textwidth}
        \includegraphics[width=\textwidth]{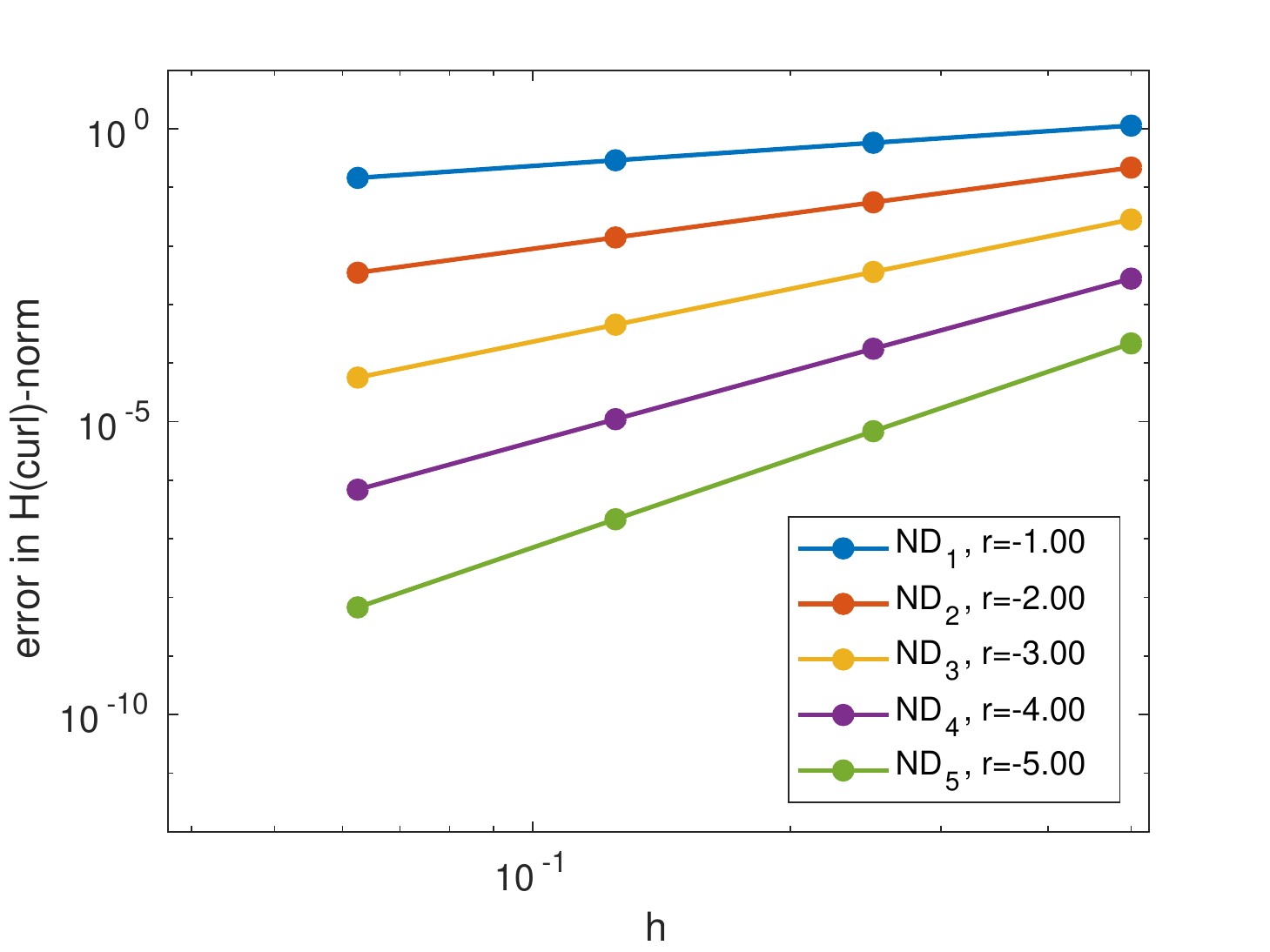}
        \caption{$H$(curl)-norm of the error. }
        \label{fig-Hcurl_error_tet_2D}
    \end{subfigure}
    \caption{Error norms for different orders 2D tetrahedral edge \acp{fe}.}\label{fig-err_2D_tet}
\end{figure}

\begin{figure}[t!]
    \centering
    \begin{subfigure}[t]{0.45\textwidth}
        \includegraphics[width=\textwidth]{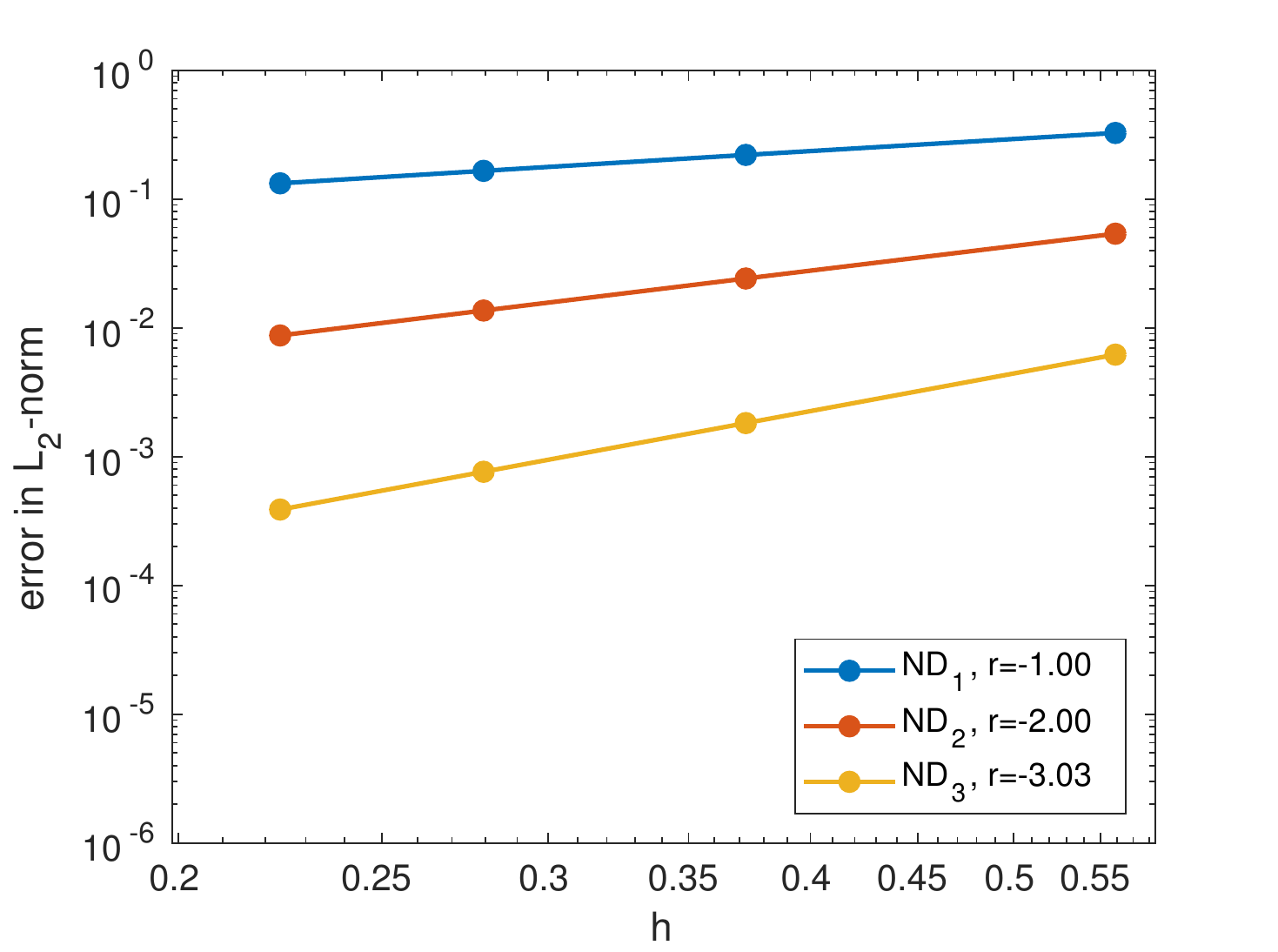}
        \caption{ $L_2$-norm of the error. }
        \label{fig-l2error_tet_3D}
    \end{subfigure}
    \begin{subfigure}[t]{0.45\textwidth}
        \includegraphics[width=\textwidth]{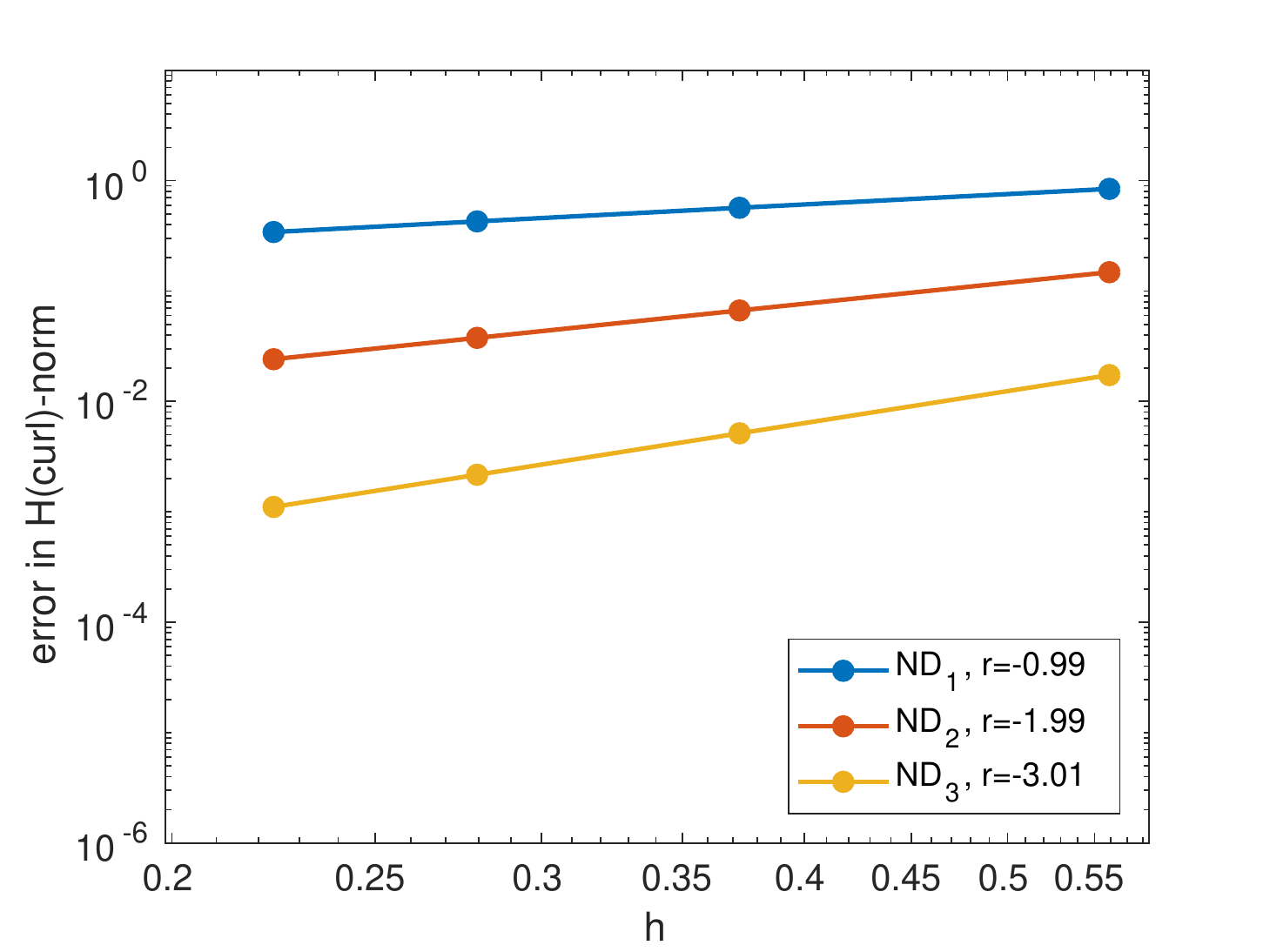}
        \caption{$H$(curl)-norm of the error. }
        \label{fig-Hcurl_error_tet_3D}
    \end{subfigure}
    \caption{Error norms for different orders 3D tetrahedral edge \acp{fe}.}\label{fig-err_3D_tet}
\end{figure}

\subsection{$h$-adaptive mesh refinement}
In this section, we present results for meshes obtained by adaptive refinement from an initial structured, hexahedral \emph{conforming} mesh. The refinement process  follows the usual steps: 1) solve the problem on a given mesh, 2) compute an estimation of the local error contribution at every cell using the solution computed at the previous step, 3) mark the cells with more error for refinement, and 4) refine the mesh and restart the process if the stopping criterion is not fulfilled. For comparison purposes, we will also consider a uniform mesh refinement. Then, log-log convergence plots for the ($L_2$ or $H$(curl)-) error against the number of free \acp{dof} involved in the simulation is presented for analytical solutions that contain a singularity in a re-entrant corner. \mo{We provide numerical results for problems where the analytical solution is known \emph{a priori}, thus we can make use of the true error. In any case, the implementation of the \emph{a posteriori} error estimators enumerated in Sect.~\ref{sec:intro} does not pose any extra implementation challenge to what it is already exposed in the paper.}
%
%

Let us consider a L-shaped domain, $\Omega=[-1,1]^2 \setminus ([0,1]\times [-1,0])$ with Dirichlet boundary conditions imposed over the entire domain boundary. The source term $\f$ is such that the solution in polar coordinates $(r,\theta)$ is
\begin{eqnarray}
\uu = \bsnabla \left( r^{\frac{2n}{3}} \sin \left( \frac{2n}{3} \theta \right) \right).
\end{eqnarray}
The chosen analytical function $\uu$ has a singular behaviour at the origin of coordinates for $n=1$, which prevents the function to be in $H^1(\Omega)$. Larger values of $n$ lead to smoother solutions. The regularity of the solution is well studied~\cite{nicaise_2001}, and theoretical convergence rates \eq{eq-err_Hc} are bounded by ${\rm min}(2n/3, k)$.

\begin{figure}[t!]
    \centering
    \begin{subfigure}[t]{0.45\textwidth}
        \includegraphics[width=\textwidth]{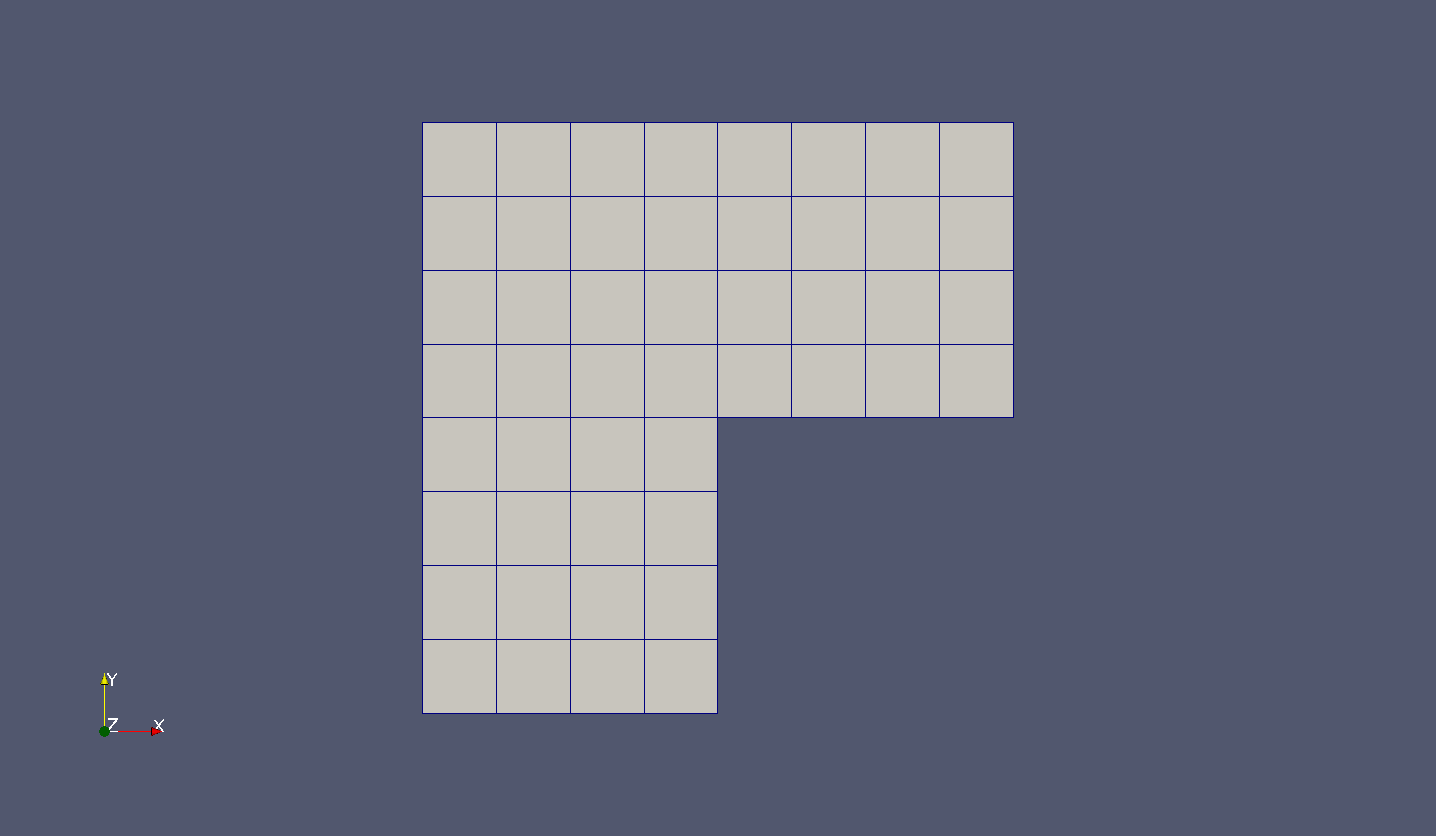}
        \caption{ L-shaped domain with initial mesh. }
        \label{fig-lshaped_2D}
    \end{subfigure}
    \begin{subfigure}[t]{0.45\textwidth}
        \includegraphics[width=\textwidth]{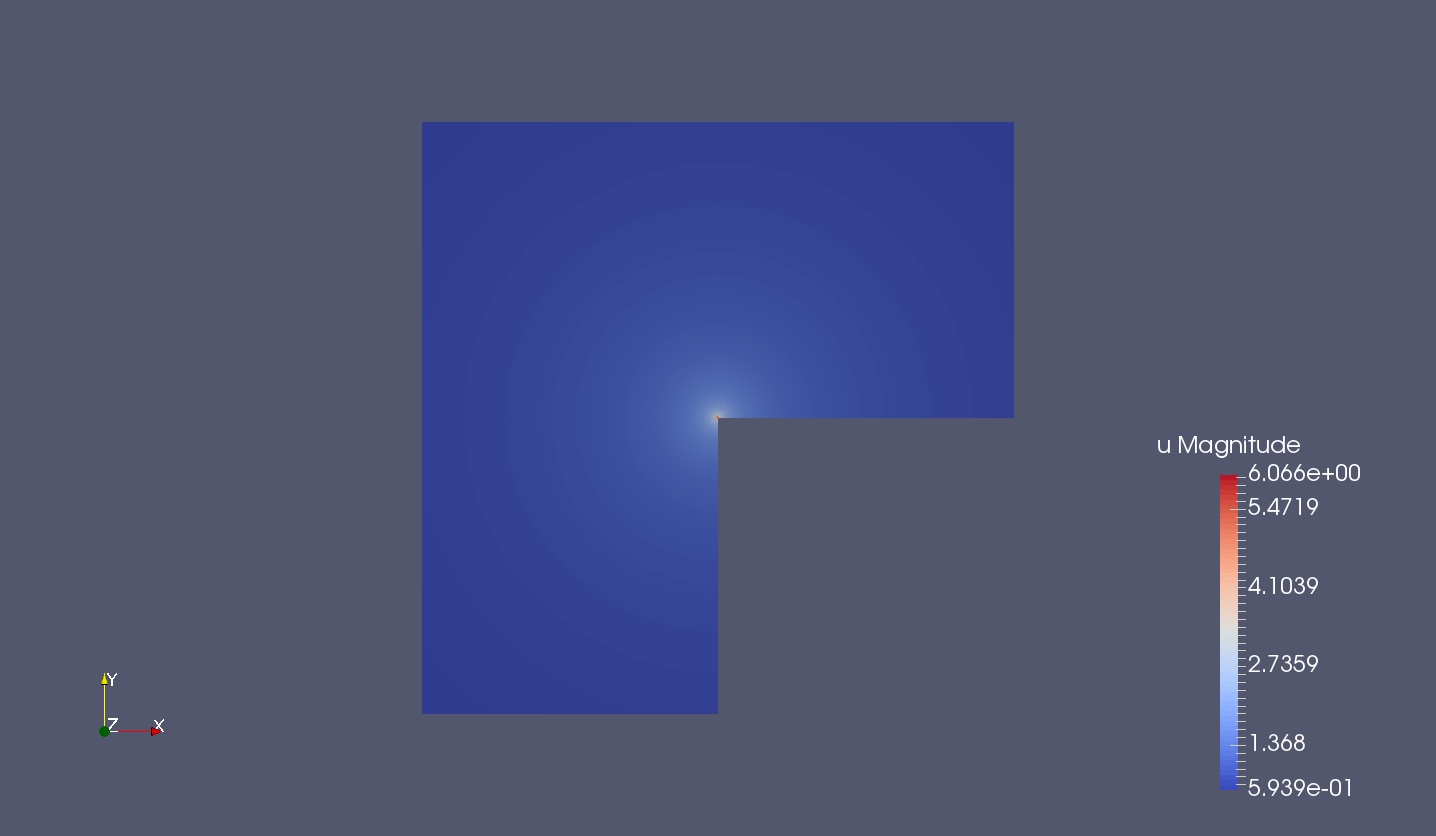}
        \caption{Analytical solution for the 2D L-shape domain with a singularity at the corner. }
        \label{fig-fichera2d_sol}
    \end{subfigure}

    \begin{subfigure}[t]{0.45\textwidth}
        \includegraphics[width=\textwidth]{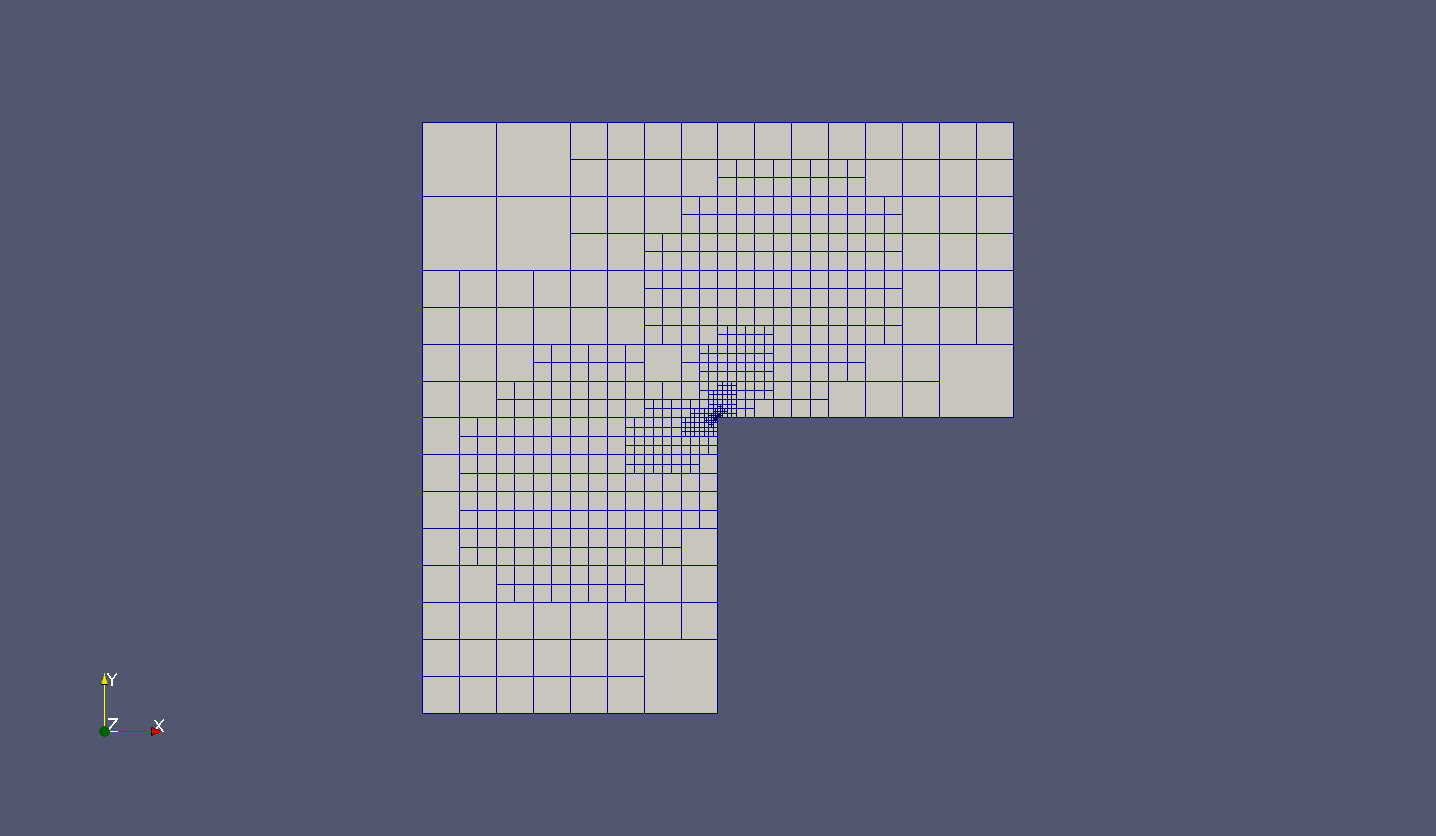}
        \caption{ L-shaped domain with refined mesh after 18 iterative mesh refinements. }
        \label{fig-lfinal_2D}
    \end{subfigure}
    \begin{subfigure}[t]{0.45\textwidth}
        \includegraphics[width=\textwidth]{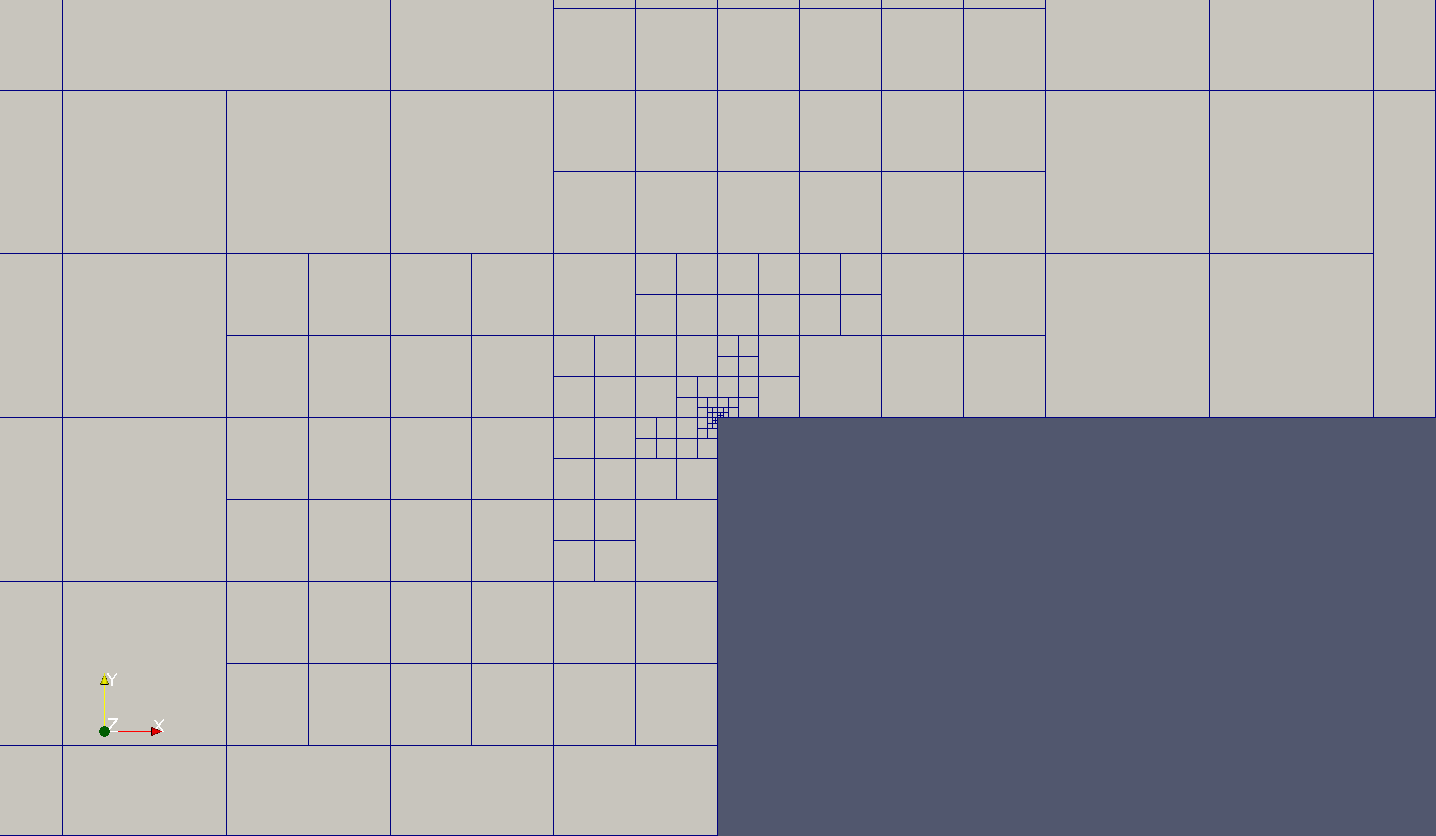}
        \caption{ Corner zoom of the refined mesh. }
        \label{fig-L_corner_zoom}
    \end{subfigure}
    \caption{fichera 2D problem. At each mesh refinement step the 5\% of cells with highest local cell $L_2$-error are refined.}\label{fig-fichera_2D}
\end{figure}

\begin{figure}[t!]
    \centering
    \begin{subfigure}[t]{0.45\textwidth}
        \includegraphics[width=\textwidth]{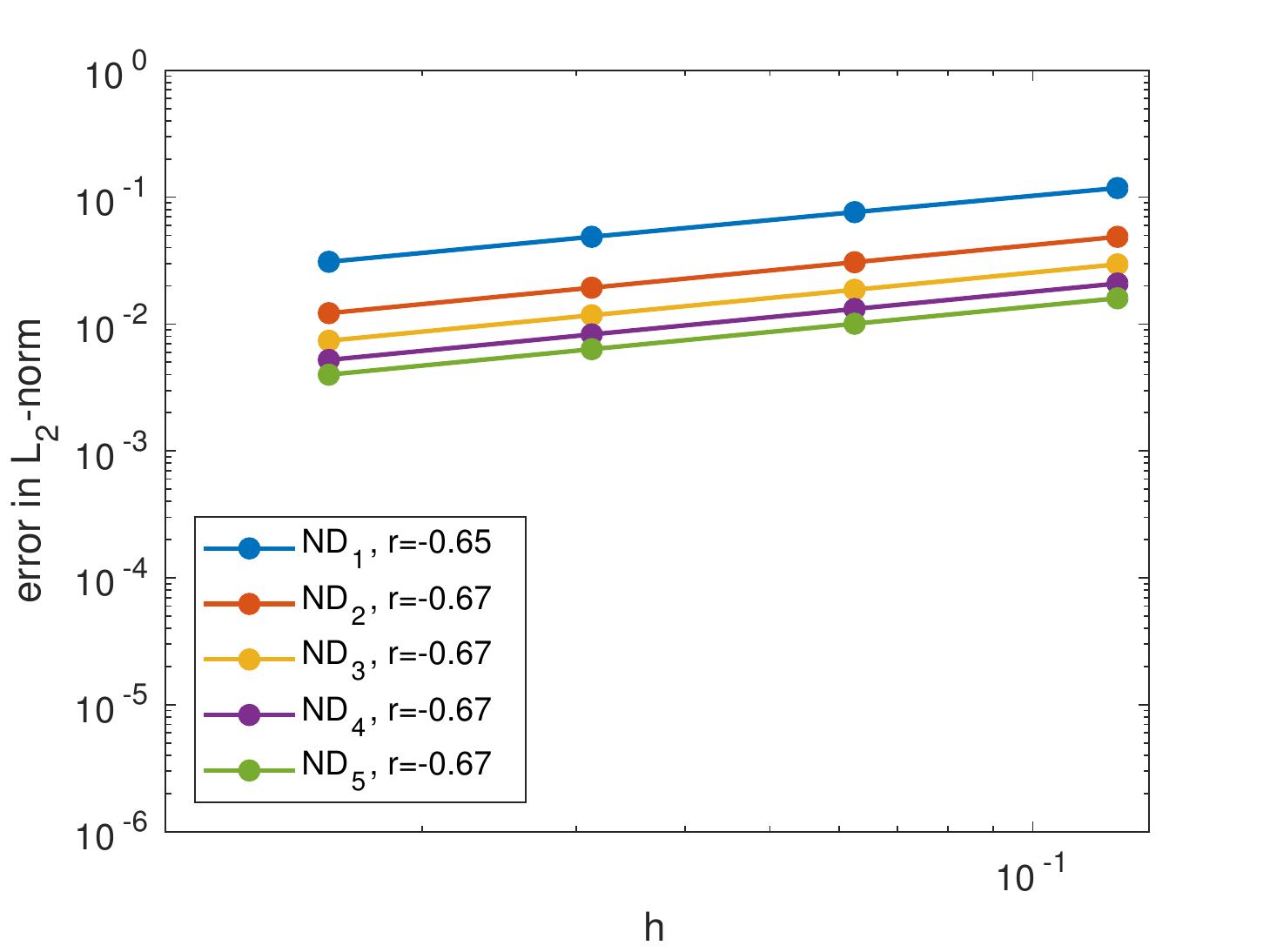}
        \caption{ $L_2$-norm of the error. }
        \label{fig-l2error_fich_2D_n1}
    \end{subfigure}
    \begin{subfigure}[t]{0.45\textwidth}
        \includegraphics[width=\textwidth]{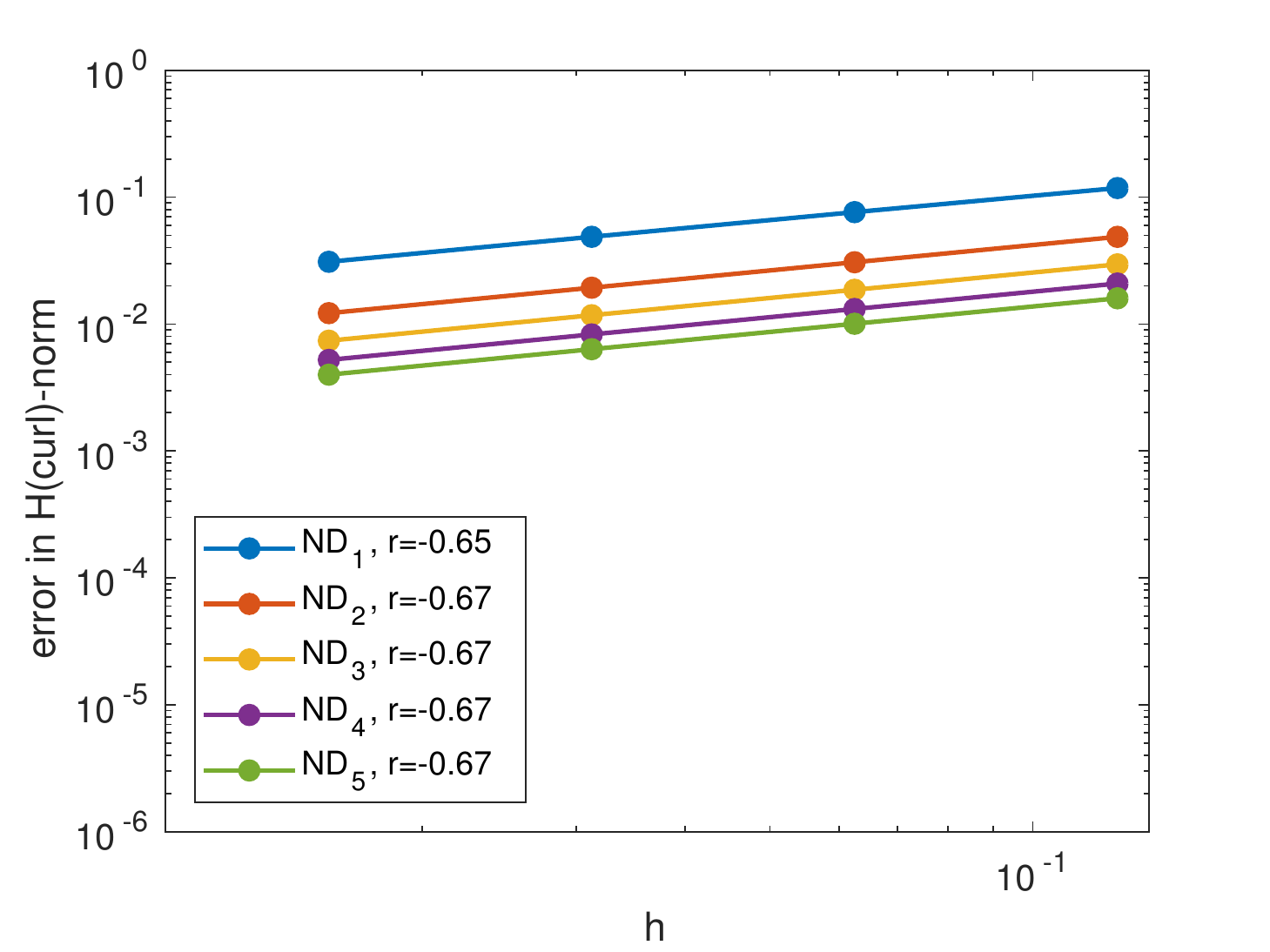}
        \caption{$H$(curl)-norm of the error. }
        \label{fig-Hcurl_error_fich_2D_n1}
    \end{subfigure}
    \caption{Error norms for the fichera 2D problem with uniform refinement for $n=1$.}\label{fig-err_fich_uni_2D}
\end{figure}

\begin{figure}[t!]
    \centering
    \begin{subfigure}[t]{0.45\textwidth}
        \includegraphics[width=\textwidth]{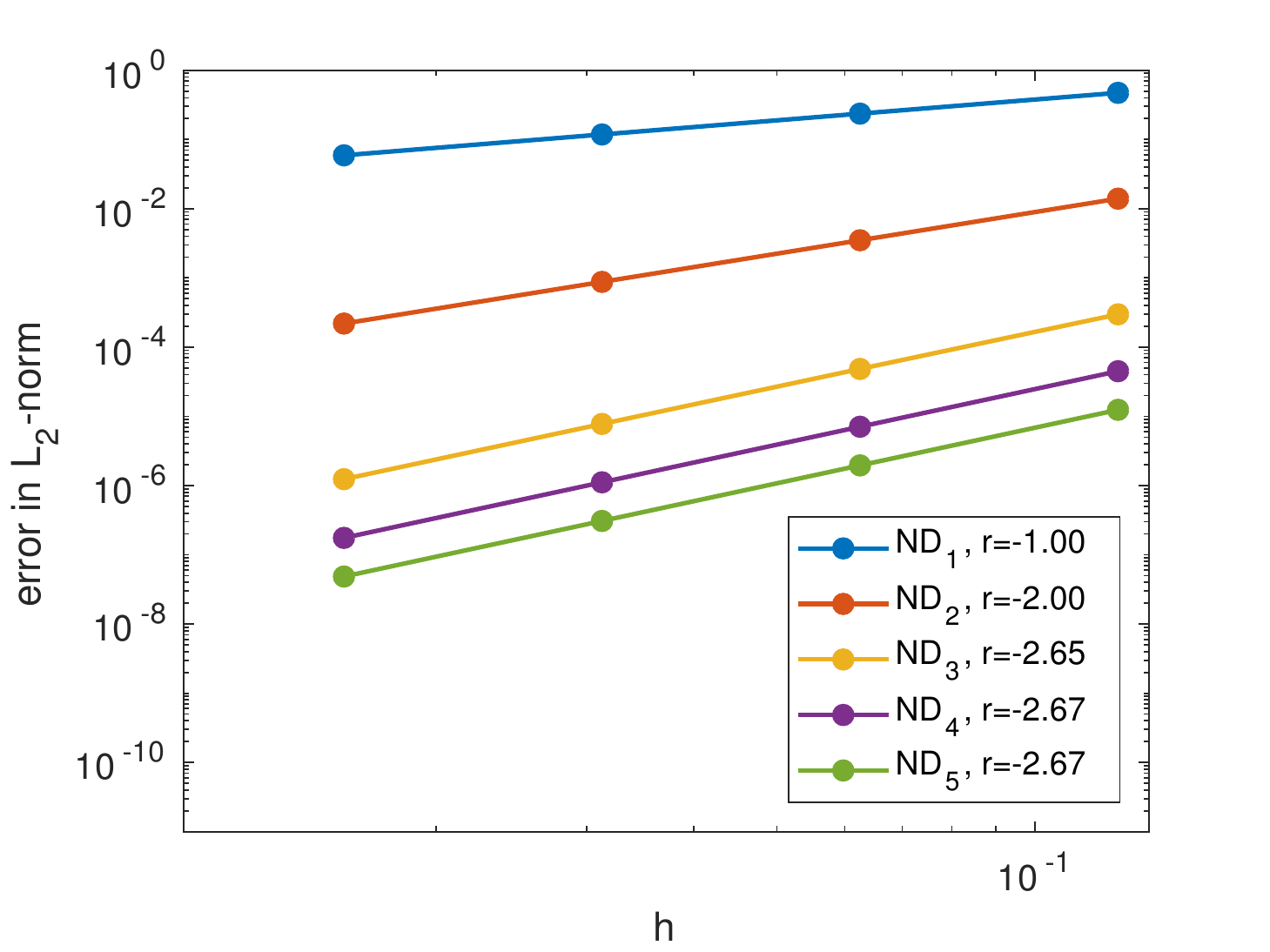}
        \caption{ $L_2$-norm of the error. }
        \label{fig-l2error_fich_2D_n4}
    \end{subfigure}
    \begin{subfigure}[t]{0.45\textwidth}
        \includegraphics[width=\textwidth]{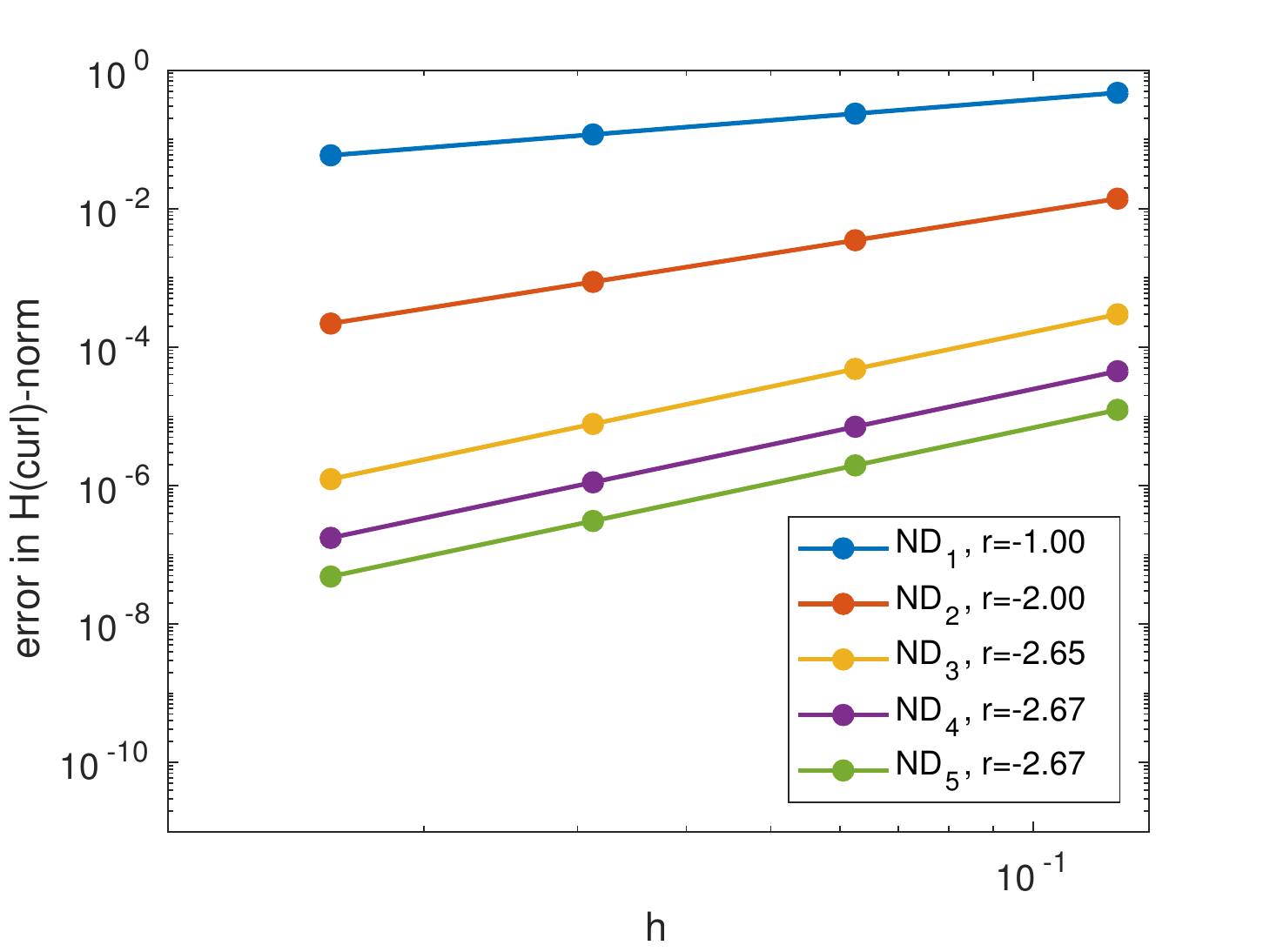}
        \caption{$H$(curl)-norm of the error. }
        \label{fig-Hcurl_error_fich_2D_n4}
    \end{subfigure}
    \caption{Error norms for the fichera 2D problem with uniform refinement for $n=4$.}\label{fig-err_fich_uni_2D_n4}
\end{figure}

In Figs.~\ref{fig-err_fich_uni_2D} and \ref{fig-err_fich_uni_2D_n4}, the theoretical convergence rates are achieved for every order of converge. In the case $n=1$, all \ac{fe} orders lead to the same convergence rate since $k>2/3$. For a smoother solution, corresponding to $n=4$, solutions converge to the expected order $\min(k, 8/3)$. Next, we analyze the error with adaptive refinement. The refinement process is such that, at every iterate, the 5\% of cells with higher local contribution to the $L_2$-error are marked for refinement. The fraction of cells to be refined is intentionally chosen to be small since we aim to obtain a localized refinement around the singularity. Figs.~\ref{fig-lshaped_2D} and \ref{fig-lfinal_2D} show the initial mesh and the final mesh after 18 refinement steps, resp. In Fig. \ref{fig-err_fich_ref_2D}, we plot the ($L_2-$ or $H({\rm curl})-$) error against the number of \acp{dof}. We show two plots for each order, namely uniform refinement (solid line) and adaptive refinement (solid line with circles). In all cases, better efficiency is achieved by adaptive meshes, i.e., less error for a given number of \acp{dof}. 
\begin{figure}[t!]
    \centering
    \begin{subfigure}[t]{0.45\textwidth}
        \includegraphics[width=\textwidth]{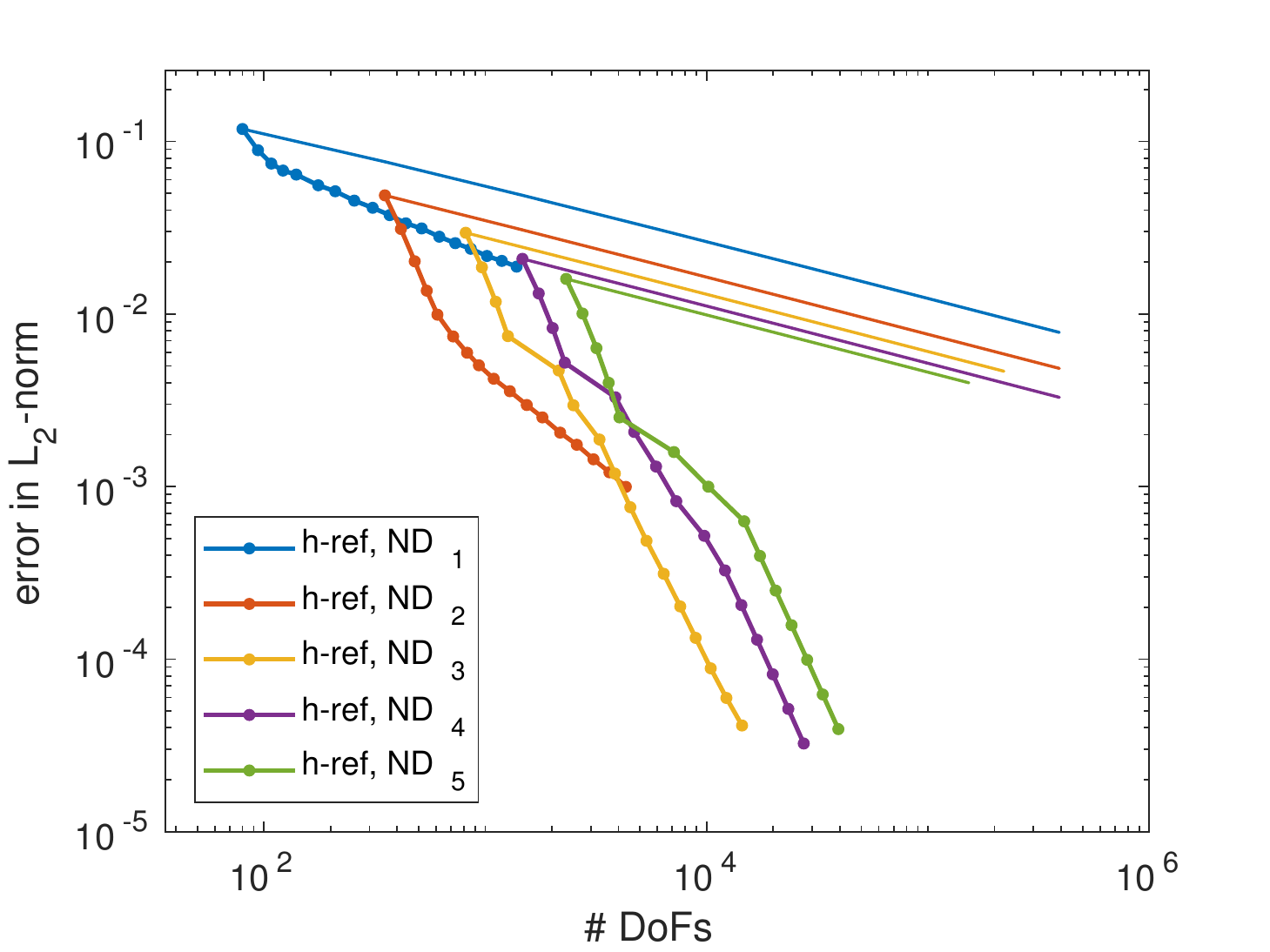}
        \caption{ $L_2$-norm of the error with number of \acp{dof}. }
        \label{fig-l2e_amr_2D}
    \end{subfigure}
    \begin{subfigure}[t]{0.45\textwidth}
        \includegraphics[width=\textwidth]{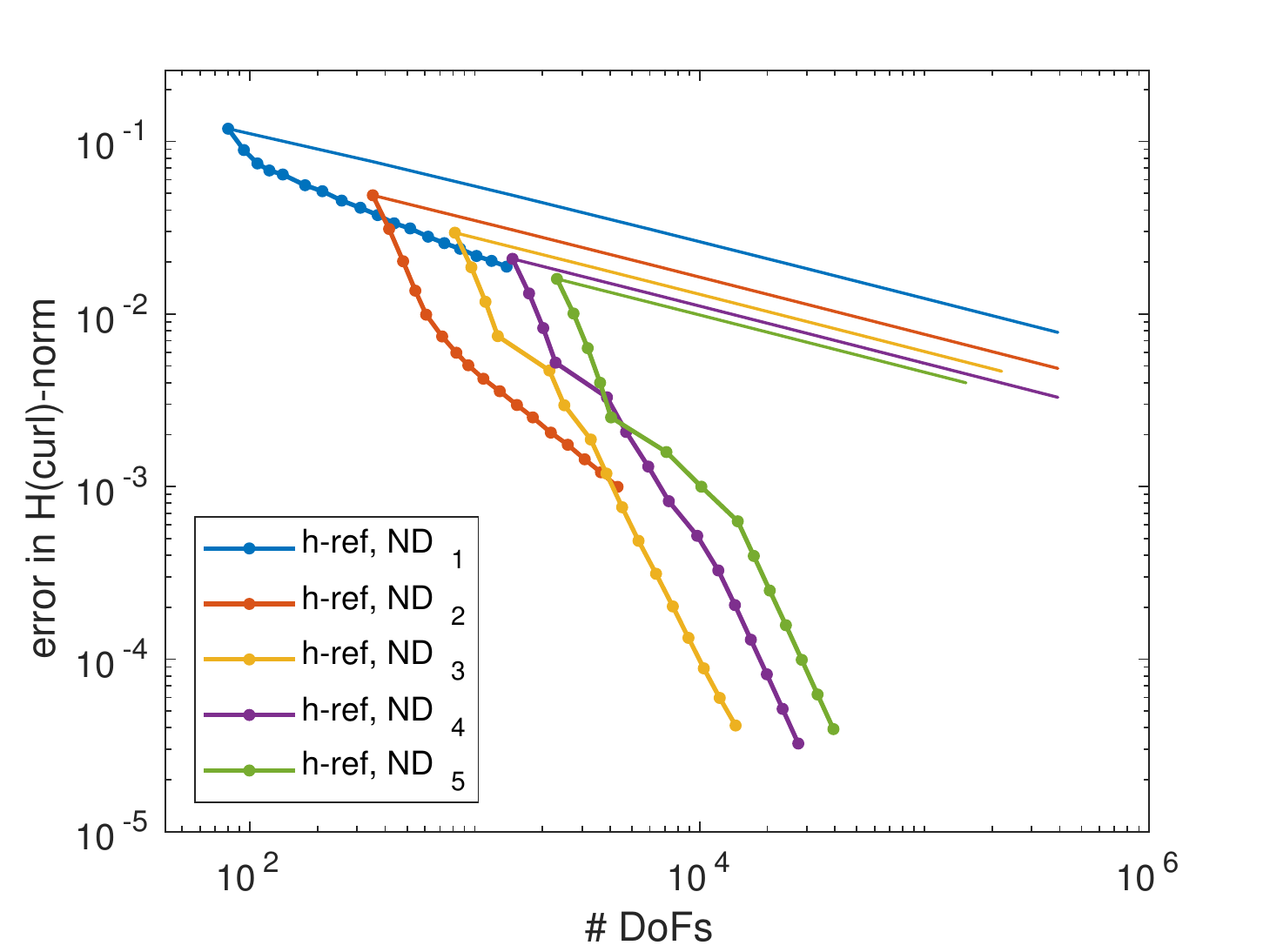}
        \caption{ $H$(curl)-norm of the error with number of \acp{dof}. }
        \label{fig-Hcurle_amr_2D}
    \end{subfigure}

    \caption{Error norms for the fichera 2D with $n=1$ and adaptive refinement, which at every iterate marks for refinement the 5\% of cells that show the highest local cell $L_2$-error. Lines without markers show the error convergence with uniform refinement process for every \ac{fe} order.}\label{fig-err_fich_ref_2D}
\end{figure}

Let us now consider the Fichera domain $\Omega=[-1,1]^3 \setminus [-1,0]^3$. The source term $\f$ is such that the solution is
\begin{eqnarray}
\uu = \bsnabla \left( r^{\frac{2}{3}} \sin \left( \frac{2t}{3} \right) \right) && t = \arccos{\left( \frac{xyz}{r} \right)},
\end{eqnarray}
where $r$ is the radius in 3D polar coordinates. The analytical solution has a singular behaviour near the origin and again $\uu \notin H^1(\Omega)$. We follow the same analysis to determine the efficiency of the $h$-adaptive scheme.
\begin{figure}[t!]
    \centering
    \begin{subfigure}[t]{0.45\textwidth}
        \includegraphics[width=\textwidth]{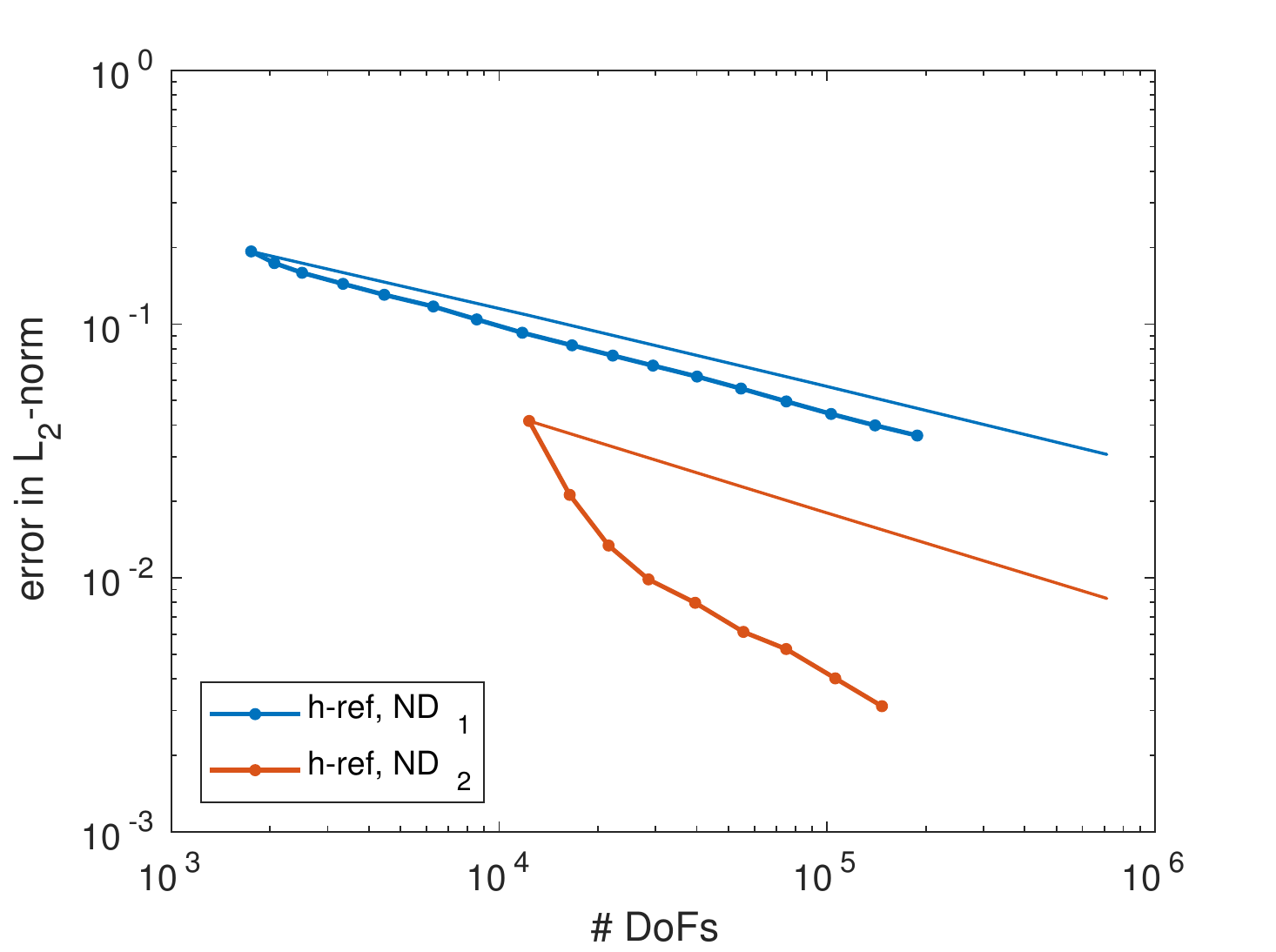}
        \caption{ $L_2$-norm of the error with number of \acp{dof}. }
        \label{fig-l2e_amr_3Dd}
    \end{subfigure}
    \begin{subfigure}[t]{0.45\textwidth}
        \includegraphics[width=\textwidth]{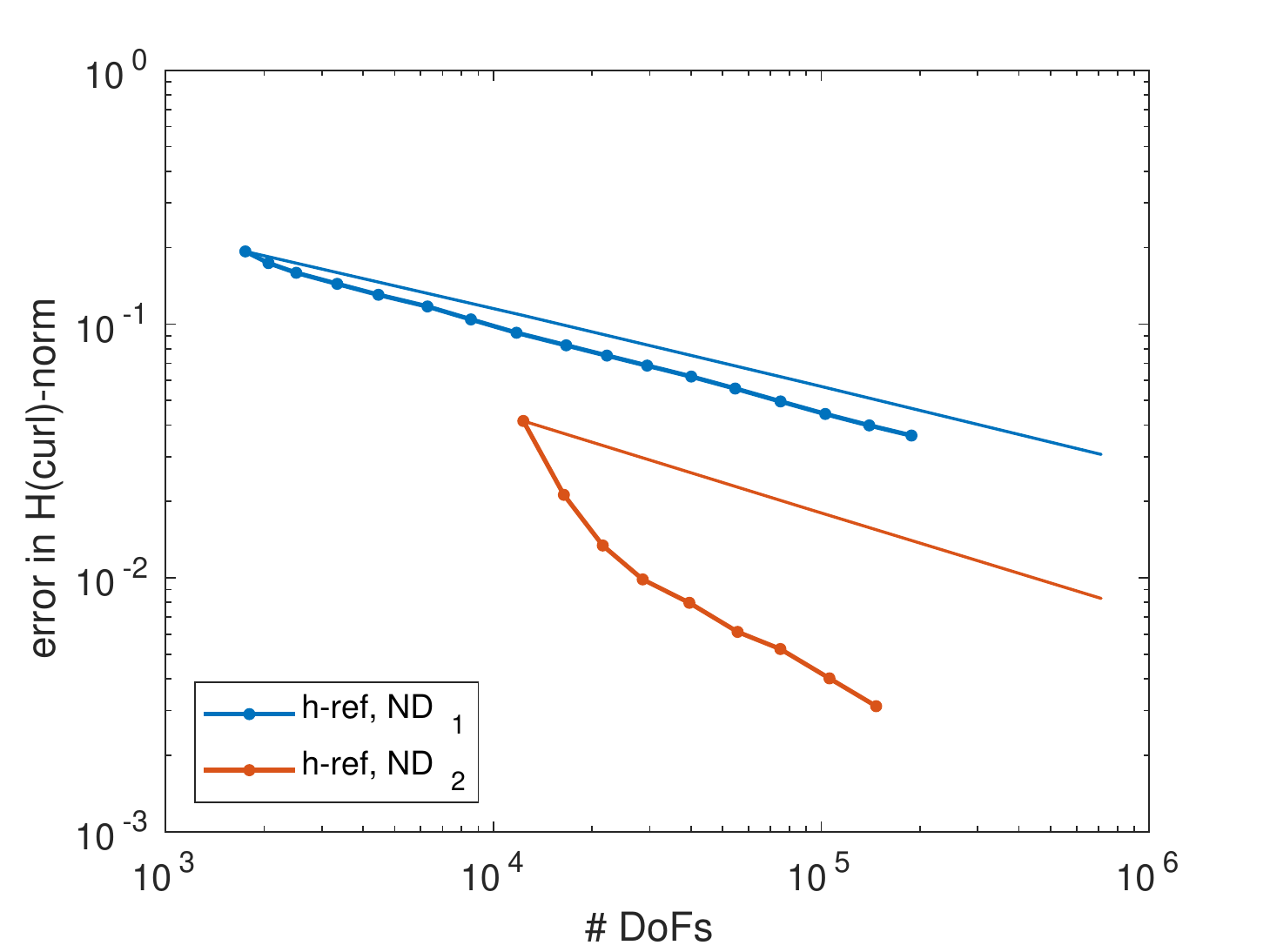}
        \caption{ $H$(curl)-norm of the error with number of \acp{dof}. }
        \label{fig-Hcurle_amr_3Dd}
    \end{subfigure}
       \caption{Error norms for the Fichera 3D problem with uniform (straight line) and adaptive refinement. 5\% of cells that have the highest local $L_2$-error marked for refinement at every refinement step. }\label{fig-fichera_4D}
\end{figure}

Fig.~\ref{fig-fichera_4D} shows the error vs. the number of \acp{dof} of the $h$-adaptive refinement process. In Fig.~\ref{fig-fichera_3Dd}, we illustrate the refinement process that takes place from an initial structured mesh composed of $8^3$ elements (see Fig.~\ref{fig-fichcubel2}). Clearly, the closer a cell K
is to the corner (see Fig.~\ref{fig-sol_fich_cubel2}), the higher the final $\ell(K)$. However, Fig.~\ref{fig-sol_fich_cubel2} also shows that the refinement is not as localized as in the 2D case, with a greater portion of the domain with higher levels of refinement. This fact has a clear impact on the efficiency achieved by the adaptive refinement for first order edge \acp{fe}, where the efficiency gain is mild. On the other hand, the gain for second order \acp{fe} is noticeably higher. Finally, Figs.~\ref{fig-fich_finall2} and \ref{fig-cubezooml2} show the refined mesh after 12 refinement iterates and a zoom at the corner with the singularity.

\begin{figure}[t!]
    \centering
    \begin{subfigure}[t]{0.45\textwidth}
        \includegraphics[width=\textwidth]{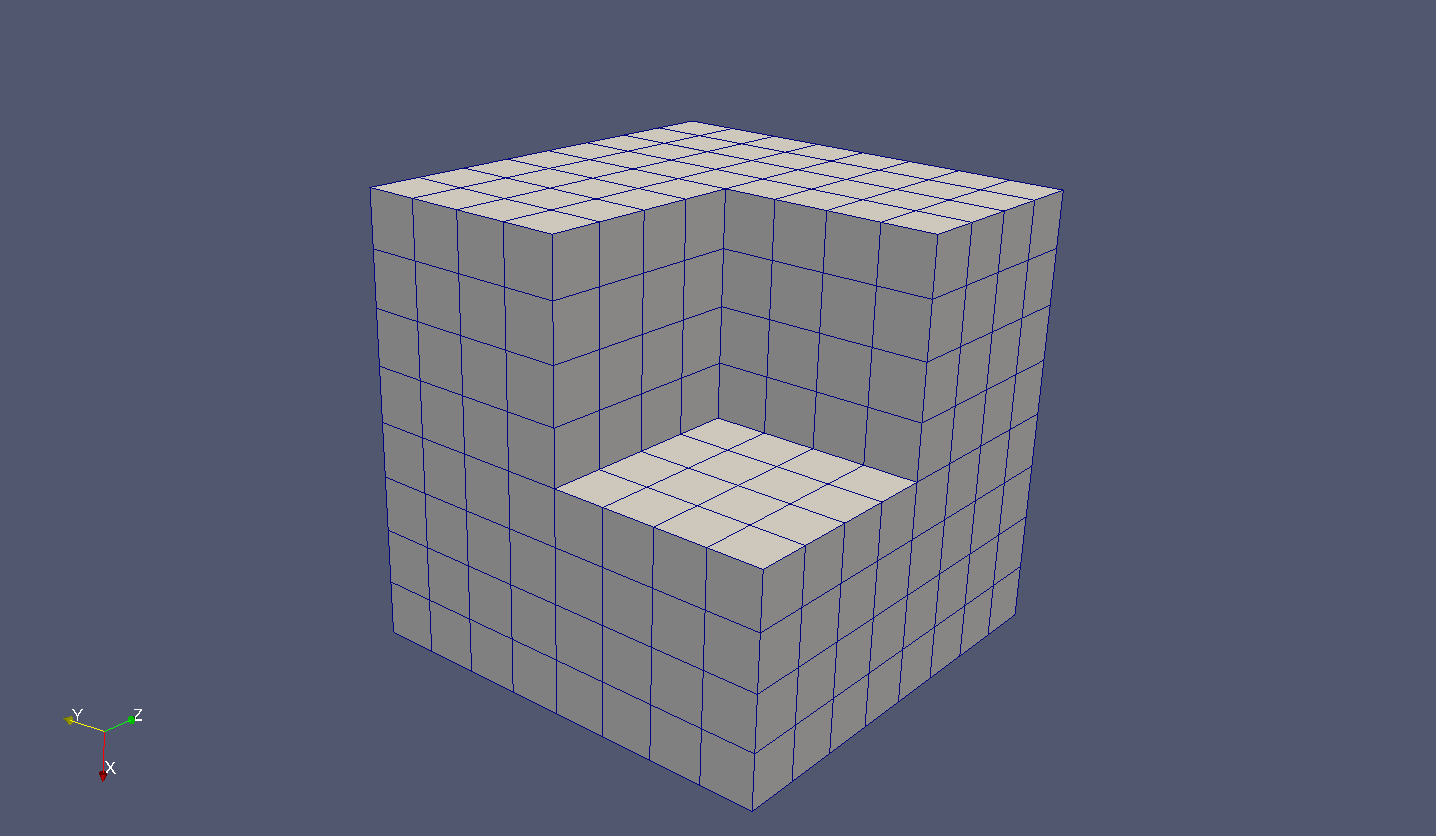}
        \caption{Initial mesh for the Fichera cube domain.}
        \label{fig-fichcubel2}
    \end{subfigure}
    \begin{subfigure}[t]{0.45\textwidth}
        \includegraphics[width=\textwidth]{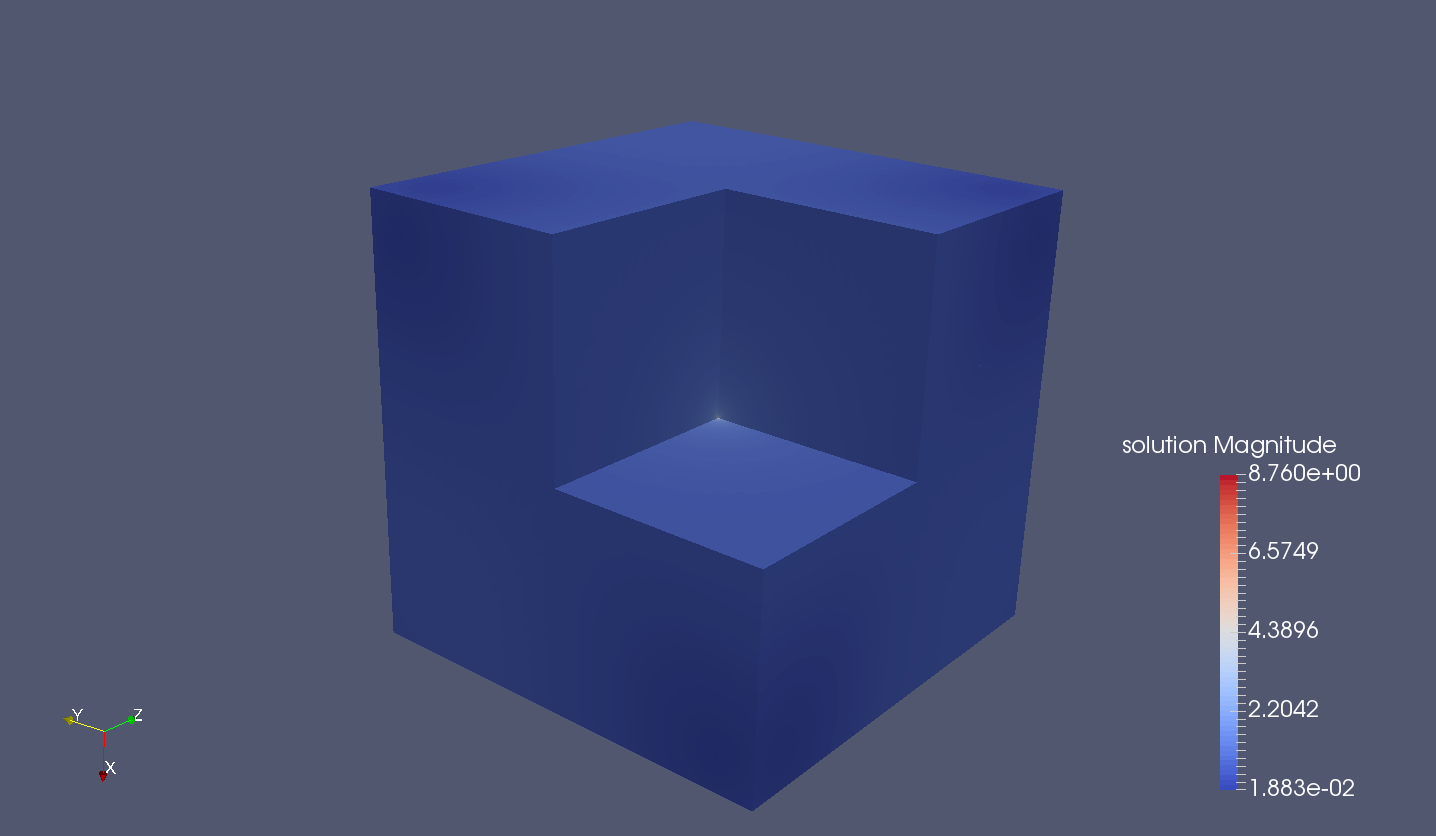}
        \caption{Analytical solution for the 3D Fichera cube. }
        \label{fig-sol_fich_cubel2}
    \end{subfigure}

        \begin{subfigure}[t]{0.45\textwidth}
        \includegraphics[width=\textwidth]{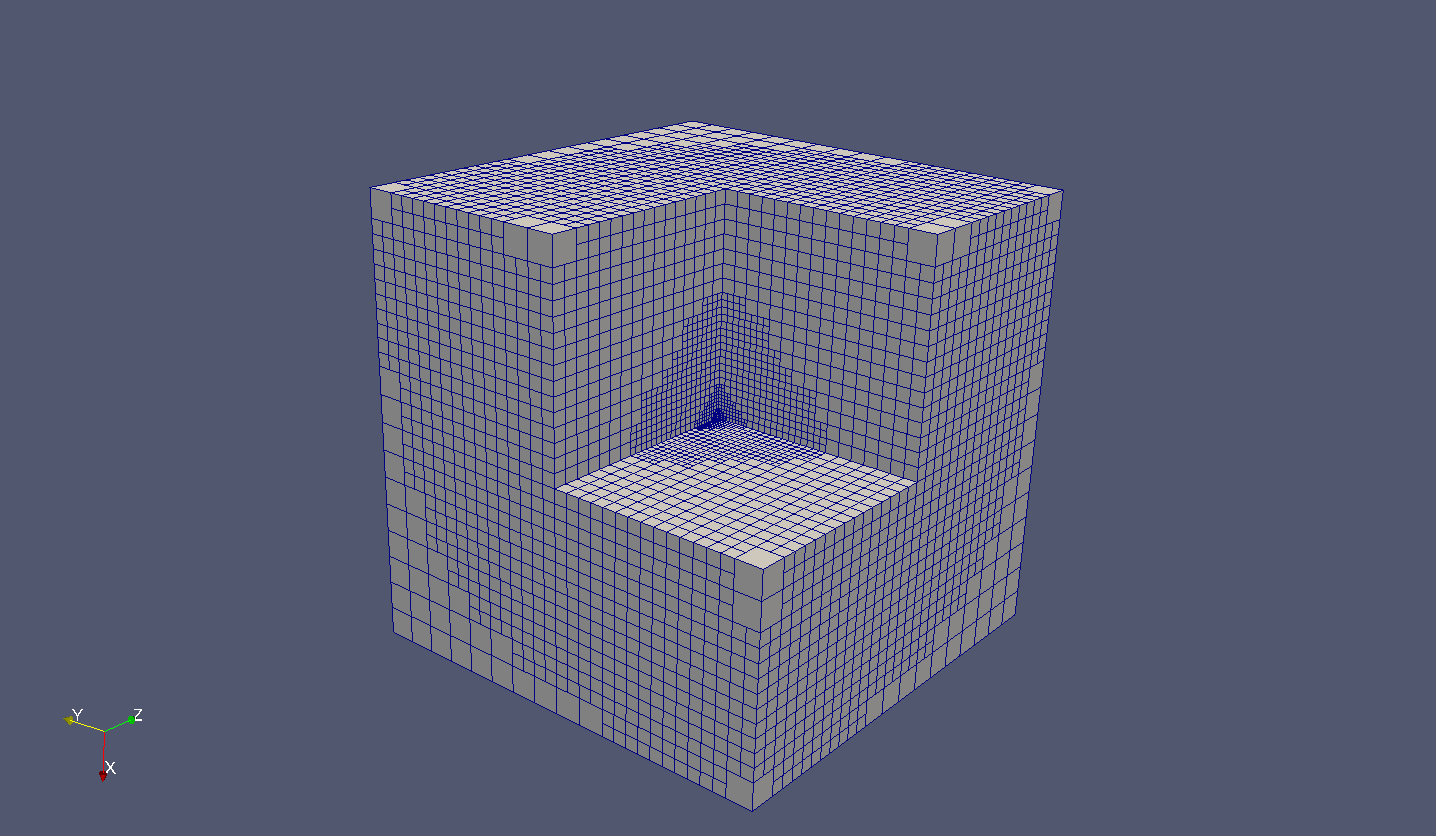}
        \caption{ Mesh after 12 iterative refinement steps for the Fichera cube domain. }
        \label{fig-fich_finall2}
    \end{subfigure}
    \begin{subfigure}[t]{0.45\textwidth}
        \includegraphics[width=\textwidth]{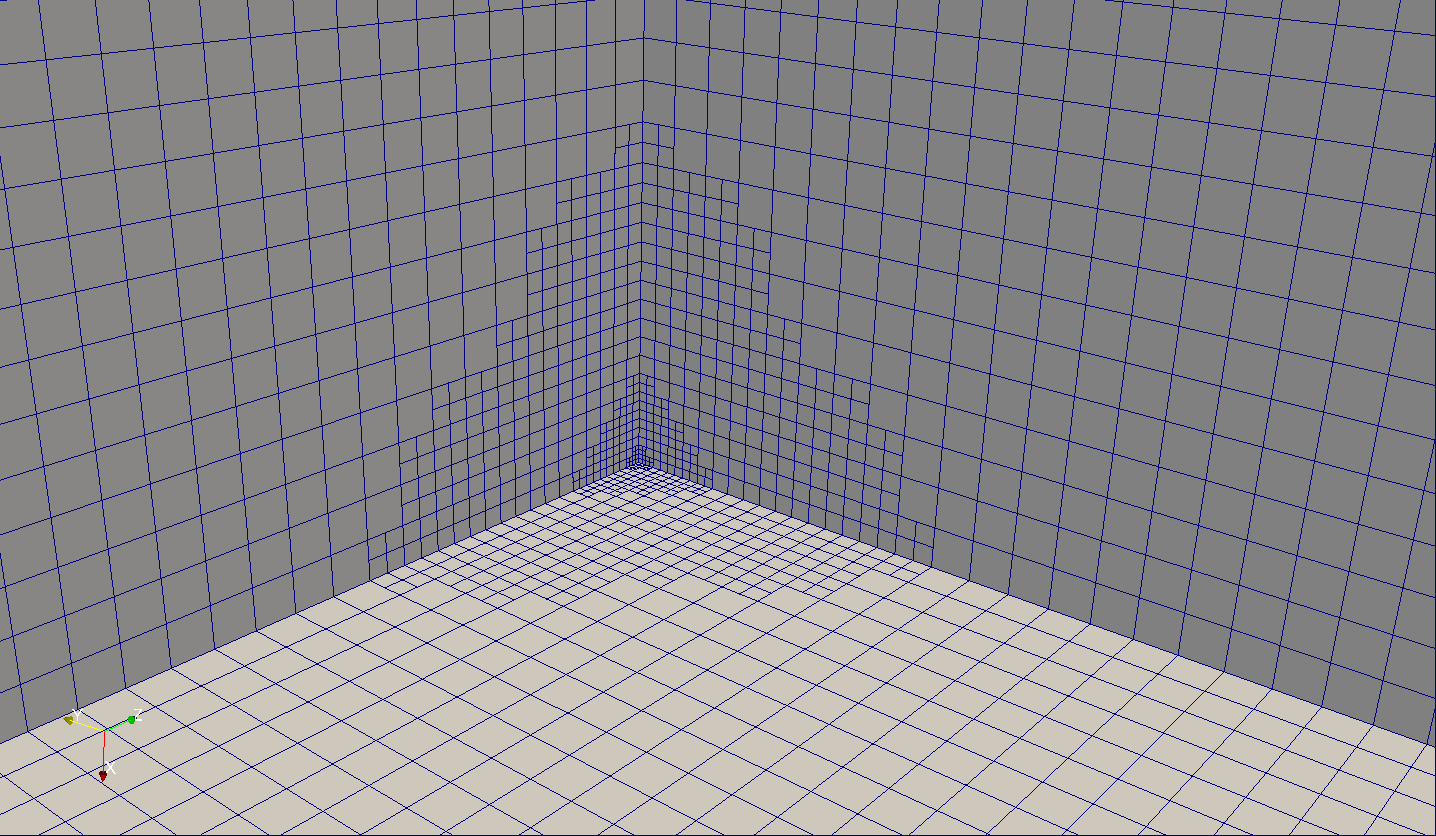}
        \caption{Corner zoom for the 3D Fichera cube. }
        \label{fig-cubezooml2}
    \end{subfigure}
    \caption{Adaptive meshes for the Fichera 3D problem. Cells with highest 5\% Local $L_2$-error(K) are refined at each refinement step.}\label{fig-fichera_3Dd}
\end{figure}

\section{Conclusions}\label{sec-conclusions}

In this work, we have covered in detail a general implementation of $p$-adaptive and $h$-adaptive tetrahedral and hexahedral edge \ac{fe} methods. We have implemented {pre-bases} that span the local \ac{fe} spaces (anisotropic polynomials) which combined with a change of basis automatically provide the shape functions bases. It leads to a general arbitrary order implementation, confronted with hard-coded implementations that preclude high order methods. In order to guarantee the tangent continuity of Piola-mapped elements, special care must be taken with the orientation of the cell geometrical entities. In the implementation, we require the local numbering of nodes within every element to rely on sorted global indices, i.e., oriented meshes. This manner, we automatically satisfy consistency in every geometrical entity shared by two or more \acp{fe}. Finally, we propose an original approach to implement global curl-conforming \ac{fe} space on hierarchically refined octree-based \emph{non-conforming} meshes. The strategy, which is straightforwardly extensible to any \ac{fe} based on polynomial spaces, is based on the original Lagrangian constraints and the interplay between the sets of Lagrangian and edge basis functions. To obtain every constraint, a sequence of spaces can be built so as we can avoid the evaluation of moments in the edge \ac{fe} space. A detailed set of numerical experiments served to test the implementation, where we show agreement between theoretical and numerical rates of convergence. The proposed approach has been implemented (for first kind edge $H$(curl)-conforming \ac{fe}) in \texttt{FEMPAR}, a scientific software for the simulation of problems governed by \acp{pde}.

We note that this framework can be extended to other polynomial-based \acp{fe}. Customizable ingredients are the original pre-basis of polynomials, the moments, the geometrical mapping, and the equivalence class of \acp{dof}. However, the change of basis approach to obtain the corresponding shape functions and the enforcement of continuity for non-conforming meshes is identical. In fact, the same machinery has already been used in \texttt{FEMPAR} to implement Raviart-Thomas \acp{fe} and can straightforwardly be used to implement Brezzi-Douglas-Marini \acp{fe} \cite{brezzi_mixed_1991}, second kind edge \acp{fe} \cite{nedelec_new_mixed_1986}, or recent divergence-free \acp{fe} \cite{neilan_stokes_2015}.

We believe that the comprehensive description of all the implementation issues behind edge \ac{fe} method provided herein will be of high value for other researchers and developers that have to purport similar developments and increase their penetration in the computational mechanics community.

\begin{small}
\bibliographystyle{unsrt}
\bibliography{art032}

\begin{thebibliography}{10}

\bibitem{nedelec_mixed_1980}
J.C. N\'ed\'elec.
\newblock Mixed finite elements in {$R^3$}.
\newblock {\em Numer. Math.}, 35:315--341, 1980.

\bibitem{monk_finite_2003}
P.~Monk.
\newblock {\em Finite {Element} {Methods} for {Maxwell}'s {Equations}}.
\newblock Oxford Science Publications, 2003.

\bibitem{mur_advantages}
G.~Mur.
\newblock Edge elements, their advantages and their disadvantages.
\newblock {\em IEEE Transactions on Magnetics}, 30(5):3552--3557, 1994.

\bibitem{badia_nodal_2012}
S.~Badia and R.~Codina.
\newblock {A Nodal-based Finite Element Approximation of the Maxwell Problem
  Suitable for Singular Solutions}.
\newblock {\em SIAM Journal on Numerical Analysis}, 50(2):398--417, 2012.

\bibitem{badia_HTS}
M.~Olm, S.~Badia, and A.F. Mart\'in.
\newblock Simulation of high temperature superconductors and experimental
  validation.
\newblock {\em Computer Physics Communications}, in press, 2018.

\bibitem{li_vectorial_2015}
Y.~L. Li, S.~Sun, Q.~I. Dai, and W.~C. Chew.
\newblock Vectorial {Solution} to {Double} {Curl} {Equation} {With}
  {Generalized} {Coulomb} {Gauge} for {Magnetostatic} {Problems}.
\newblock {\em IEEE Transactions on Magnetics}, 51(8):1--6, 2015.

\bibitem{perugia_dG_2003}
I.~Perugia and D.~Schötzau.
\newblock The {$hp$}-local discontinuous {G}alerkin method for low-frequency
  time-harmonic {M}axwell equations.
\newblock {\em Mathematics of Computation}, 72:1179--1214, 2003.

\bibitem{badia_dG_2011}
S.~Badia and R.~Codina.
\newblock A combined nodal continuous–discontinuous finite element
  formulation for the {M}axwell problem.
\newblock {\em Applied Mathematics and Computation}, 218(8):4276 -- 4294, 2011.

\bibitem{buffa_iga}
R.~Vazquez and A.~Buffa.
\newblock Isogeometric analysis for electromagnetic problems.
\newblock {\em IEEE Transactions on Magnetics}, 46(8):3305--3308, Aug 2010.

\bibitem{beirao_2013}
L.~Beirão~da Veiga, F.~Brezzi, A.~Cangiani, G.~Manzini, L.~D. Marini, and
  A.~Russo.
\newblock Basic principles of virtual element methods.
\newblock {\em Mathematical Models and Methods in Applied Sciences},
  23(01):199--214, 2013.

\bibitem{beirao_2016}
L.~Beirão da~Veiga, F.~Brezzi, L.D. Marini, and A.~Russo.
\newblock {$H$(div)} and {$H$(curl)}-conforming {VEM}.
\newblock {\em Numerische Mathematik}, 133:303--332, 2016.

\bibitem{gopalakrishnan_2005}
J.~Gopalakrishnan, L.E. García-Castillo, and L.F. Demkowicz.
\newblock {N\'ed\'elec} spaces in affine coordinates.
\newblock {\em Computers \& Mathematics with Applications}, 49(7):1285--1294,
  2005.

\bibitem{Bonazzoli_2017}
M.~Bonazzoli and F.~Rapetti.
\newblock High-order finite elements in numerical electromagnetism: degrees of
  freedom and generators in duality.
\newblock {\em Numerical Algorithms}, 74(1):111--136, 2017.

\bibitem{Bonazzoli_implementation_2017}
M.~Bonazzoli, V.~Dolean, F.~Hecht, and F.~Rapetti.
\newblock An example of explicit implementation strategy and preconditioning
  for the high order edge finite elements applied to the time-harmonic
  {M}axwell’s equations.
\newblock {\em Computers {\&} Mathematics with Applications}, 75(5):1498 --
  1514, 2018.

\bibitem{fuentes_nedelec}
F.~Fuentes, B.~Keith, L.~Demkowicz, and S.~Nagaraj.
\newblock Orientation embedded high order shape functions for the exact
  sequence elements of all shapes.
\newblock {\em Computers {\&} Mathematics with Applications}, 70(4):353 -- 458,
  2015.

\bibitem{demkowicz_book}
L.~Demkowicz.
\newblock {\em Computing with hp-adaptive finite elements: Volume I: One and
  Two dimensional elliptic and Maxwell problems}.
\newblock Chapman \& Hall/CRC Press, 2006.

\bibitem{demkowicz_book2}
L.~Demkowicz, J.~Kurtz, D.~Pardo, M.~Paszenski, W.~Rachowicz, and A.~Zdunek.
\newblock {\em Computing with hp-adaptive finite elements: Volume II: Three
  Dimensional Elliptic and Maxwell Problems with applications}.
\newblock Chapman \& Hall/CRC Press, 2007.

\bibitem{amor_castillo}
A.~Amor-Martin, L.~E. Garcia-Castillo, and D.~Daniel Garcia-Doñoro.
\newblock Second-order {Nédélec} curl-conforming prismatic element for
  computational electromagnetics.
\newblock {\em IEEE Transactions on Antennas and Propagation},
  64(10):4384--4395, 2016.

\bibitem{castillo_third}
L.~E. Garcia-Castillo, A.~J. Ruiz-Genoves, I.~Gomez-Revuelto, M.~Salazar-Palma,
  and T.~K. Sarkar.
\newblock Third-order {N\'ed\'elec} curl-conforming finite element.
\newblock {\em IEEE Transactions on Magnetics}, 38(5):2370--2372, 2002.

\bibitem{schneebeli_2003}
A.~Schneebeli.
\newblock An {H}(curl;{$\Omega$})-conforming {FEM}: {N\'ed\'elec} elements of
  first type.
\newblock Technical report, 2003.

\bibitem{rognes_efficient_2009}
M.~Rognes, R.~Kirby, and A.~Logg.
\newblock Efficient {Assembly} of {H}(div) and {H}(curl) {Conforming} {Finite}
  {Elements}.
\newblock {\em SIAM J. Sci. Comput.}, 31(6):4130--4151, 2009.

\bibitem{falk_2011}
R.~S. Falk, P.~Gatto, and P.~Monk.
\newblock Hexahedral {H}(div) and {H}(curl) finite elements.
\newblock {\em ESAIM: M2AN}, 45(1):115--143, 2011.

\bibitem{badia_unfitted_2018}
S.~Badia, F.~Verdugo, and A.F. Martín.
\newblock The aggregated unfitted finite element method for elliptic problems.
\newblock {\em Computer Methods in Applied Mechanics and Engineering}, 336:533
  -- 553, 2018.

\bibitem{anjam_2015}
I.~Anjam and J.~Valdman.
\newblock Fast {MATLAB} assembly of {FEM} matrices in {2D} and {3D}: {E}dge
  elements.
\newblock {\em Applied Mathematics and Computation}, 267:252 -- 263, 2015.

\bibitem{castillo_99}
L.~E. Garc\'ia~Castillo and M.~Salazar~Palma.
\newblock Second order {N\'ed\'elec} tetrahedral element for computational
  electromagnetics.
\newblock {\em International Journal of Numerical Modelling: Electronic
  Networks, Devices and Fields}, 13(3):261--287.

\bibitem{agelek_orientation_2017}
R.~Agelek, M.~Anderson, W.~Bangerth, and W.~Barth.
\newblock On orienting edges of unstructured two- and three-dimensional meshes.
\newblock {\em ACM Transactions on Mathematical Software}, 44(1):Article No. 5,
  2017.

\bibitem{BursteddeWilcoxGhattas11}
C.~Burstedde, L.~Wilcox, and O.~Ghattas.
\newblock {\texttt{p4est}}: Scalable algorithms for parallel adaptive mesh
  refinement on forests of octrees.
\newblock {\em SIAM J. Sci. Comput.}, 33(3):1103--1133, 2011.

\bibitem{alnaes_fenics_2015}
M.~Alnæs, J.~Blechta, J.~Hake, A.~Johansson, B.~Kehlet, A.~Logg,
  C.~Richardson, J.~Ring, M.~E. Rognes, and G.~N. Wells.
\newblock The {FEniCS} {Project} {Version} 1.5.
\newblock {\em Archive of Numerical Software}, 3(100), 2015.

\bibitem{bangerth_deal.ii-general-purpose_2007}
W.~Bangerth, R.~Hartmann, and G.~Kanschat.
\newblock deal.{II}-{A} general-purpose object-oriented finite element library.
\newblock {\em ACM Trans. Math. Softw.}, 33(4), 2007.

\bibitem{bangerth_textttdeal.ii_2016}
W.~Bangerth, D.~Davydov, T.~Heister, L.~Heltai, G.~Kanschat, M.~Kronbichler,
  M.~Maier, B.~Turcksin, and D.~Wells.
\newblock The deal.{II} {Library}, {Version} 8.4.
\newblock {\em Journal of Numerical Mathematics}, 24, 2016.

\bibitem{hecht_new_2012}
F.~Hecht.
\newblock New development in {FreeFem}++.
\newblock {\em Journal of Numerical Mathematics}, 20(3-4):251--265, 2012.

\bibitem{_mfem_????}
{MFEM} -- a free, lightweight, scalable {C}++ library for finite element
  methods.
\newblock \href{http://mfem.org/}{\tt http://mfem.org/}.

\bibitem{Schoberl1997}
Joachim Sch{\"o}berl.
\newblock Netgen an advancing front 2d/3d-mesh generator based on abstract
  rules.
\newblock {\em Computing and Visualization in Science}, 1(1):41--52, Jul 1997.

\bibitem{netgen}
{Netgen/NGSolve}.
\newblock \href{http://ngsolve.org/}{\tt http://ngsolve.org/}.

\bibitem{Shadid2016}
J.N. Shadid, R.P. Pawlowski, E.C. Cyr, R.S. Tuminaro, L.~Chac{\'{o}}n, and P.D.
  Weber.
\newblock {Scalable implicit incompressible resistive MHD with stabilized FE
  and fully-coupled Newton–Krylov-AMG}.
\newblock {\em Computer Methods in Applied Mechanics and Engineering},
  304:1--25, 2016.

\bibitem{Ortigosa2016}
R.~Ortigosa and A.~J. Gil.
\newblock {A new framework for large strain electromechanics based on convex
  multi-variable strain energies: Finite Element discretisation and
  computational implementation}.
\newblock {\em Computer Methods in Applied Mechanics and Engineering},
  302:329--360, 2016.

\bibitem{GuolinWang2015}
G.~Wang, S.~Wang, N.~Duan, Y.~Huangfu, H.~Zhang, W.~Xu, and J.~Qiu.
\newblock {Extended Finite-Element Method for Electric Field Analysis of
  Insulating Plate With Crack}.
\newblock {\em IEEE Transactions on Magnetics}, 51(3):1--4, 2015.

\bibitem{bergot_2013}
M.~Bergot and M.~Duruflé.
\newblock High-order optimal edge elements for pyramids, prisms and hexahedra.
\newblock {\em Journal of Computational Physics}, 232(1):189--213, 2013.

\bibitem{bergot_2010}
M.~Bergot and P.~Lacoste.
\newblock Generation of higher-order polynomial basis of {N\'ed\'elec H(curl)}
  finite elements for {M}axwell’s equations.
\newblock {\em Journal of Computational and Applied Mathematics},
  234(6):1937--1944, 2010.

\bibitem{Bangerth_2009}
W.~Bangerth and O.~Kayser-Herold.
\newblock Data structures and requirements for hp finite element software.
\newblock {\em ACM Trans. Math. Softw.}, 36(1):4:1--4:31, 2009.

\bibitem{brezzi_mixed_1991}
F.~Brezzi and M.~Fortin.
\newblock {\em {Mixed and hybrid finite element methods}}.
\newblock Springer-Verlag, 1991.

\bibitem{nedelec_new_mixed_1986}
J.C. N\'ed\'elec.
\newblock A new family of mixed finite elements in {$R^3$}.
\newblock {\em Numer. Math.}, 50:57--81, 1986.

\bibitem{neilan_stokes_2015}
M.~Neilan and D.~Sap.
\newblock {Stokes elements on cubic meshes yielding divergence-free
  approximations}.
\newblock {\em Calcolo}, 53(3):1--21, 2015.

\bibitem{hiptmair_adaptive_2000}
R.~Beck, R.~Hiptmair, R.H.W. Hoppe, and B.~Wohlmuth.
\newblock Residual based a posteriori error estimators for eddy current
  computation.
\newblock {\em ESAIM: M2AN}, 34(1):159--182, 2000.

\bibitem{schoberl_estimators}
J.~Schöberl.
\newblock A posteriori error estimates for {M}axwell equations.
\newblock {\em Mathematics of Computation}, 77(262):633--649, 2008.

\bibitem{nicaise_estimators}
S.~Nicaise and E.~Creus{\'e}.
\newblock A posteriori error estimation for the heterogeneous {M}axwell
  equations on isotropic and anisotropic meshes.
\newblock {\em CALCOLO}, 40(4):249--271, 2003.

\bibitem{beck_adaptive_1999}
R.~Beck, R.~Hiptmair, and B.~Wohlmuth.
\newblock Hierarchical error estimator for eddy current computation.
\newblock In {\em In ENUMATH99: Proceedings of the 3rd European Conference on
  Numerical Mathematics and Advanced Applictions}, pages 110--120. World
  Scientific, 1999.

\bibitem{braess_estimator}
D.~Braess and J.~Schöberl.
\newblock Equilibrated residual error estimator for edge elements.
\newblock {\em Mathematics of Computation}, 77(262):651--672, 2008.

\bibitem{nicaise_estimator}
S.~Nicaise.
\newblock On {Z}ienkiewicz–{Z}hu error estimators for {M}axwell's equations.
\newblock {\em Comptes Rendus Mathematique}, 340(9):697 -- 702, 2005.

\bibitem{badia-fempar}
S.~Badia, A.F. Mart\'in, and J.~Principe.
\newblock \texttt{FEMPAR}: An object-oriented parallel finite element
  framework.
\newblock {\em Arch. Comput. Methods Eng.}, 25(2):195--271, 2018.

\bibitem{fempar-web-page}
S.~Badia, A.F. Mart\'in, and J.~Principe.
\newblock \texttt{FEMPAR} {W}eb page.
\newblock {http://www.fempar.org}, 2018.

\bibitem{boffi_2002}
D.N. Arnold, D.~Boffi, and R.S. Falk.
\newblock Approximation by quadrilateral finite elements.
\newblock {\em Mathematics of computation}, 71(239):909--922, 2002.

\bibitem{badia_unfitted_2017}
S.~Badia, F.~Verdugo, and A.F. Martín.
\newblock The aggregated unfitted finite element method for elliptic problems.
\newblock {\em Computer Methods in Applied Mechanics and Engineering}, 336:533
  -- 553, 2018.

\bibitem{Cohen2017}
G.~Cohen and S.~Pernet.
\newblock {\em Definition of Different Types of Finite Elements}, pages 39--93.
\newblock Springer Netherlands, Dordrecht, 2017.

\bibitem{Tu_2005}
T.~Tu, D.~R. O'Hallaron, and O.~Ghattas.
\newblock Scalable parallel octree meshing for terascale applications.
\newblock In {\em Proceedings of the ACM/IEEE SC 2005 Conference,
  Supercomputing, 2005}, pages 1--15, 2005.

\bibitem{Solin2003}
P.~Solin, K.~Segeth, and I.~Dolezel.
\newblock {\em Higher-Order Finite Element Methods}.
\newblock Chapman \& Hall/CRC Press, 2003.

\bibitem{lssc_hadap_2018}
S.~Badia, A.F. Mart\'in, E.~Neiva, and F.~Verdugo.
\newblock A generic finite element framework on parallel tree-based adaptive
  meshes.
\newblock {\em Submitted for publication}, 2019.

\bibitem{alonso_1999}
A.~Alonso and A.~Valli.
\newblock An optimal domain decomposition preconditioner for low-frequency
  time-harmonic maxwell equations.
\newblock {\em Math. Comput.}, 68(226):607--631, 1999.

\bibitem{nicaise_2001}
S.~Nicaise.
\newblock {E}dge {E}lements on anisotropic meshes and approximation of the
  {M}axwell {E}quations.
\newblock {\em SIAM Journal of Numerical Analysis}, 39(3):784--816, 2001.

\end{thebibliography}
\end{small}

\end{document}